\documentclass[3p,10pt,review]{elsarticle}
\usepackage{tgtermes}
\usepackage{amsmath}
\usepackage{amssymb}
\usepackage{mathtools}
\usepackage{amsthm}
\usepackage[dvipsnames]{xcolor}
\usepackage{arydshln}
\usepackage{graphicx}
\usepackage{booktabs}
\usepackage{float}
\usepackage{booktabs}
\usepackage{colortbl}
\usepackage{multirow}
\usepackage{caption}
\usepackage{easybmat}
\usepackage{lineno}
\journal{Elsevier}
\biboptions{sort&compress}
\usepackage{bm}
\usepackage{subcaption}
\usepackage{empheq}
\usepackage{fancybox}
\usepackage[linesnumbered,ruled,vlined]{algorithm2e}
\usepackage{mdframed}
\usepackage[none]{hyphenat}
\usepackage[most]{tcolorbox}
\usepackage{threeparttable}
\usepackage{adjustbox}
\usepackage[colorlinks=true]{hyperref}
\usepackage{cleveref}
\AtBeginDocument{
	\hypersetup{
		colorlinks   = true,
		linkcolor    = red,
		citecolor    = green,
		anchorcolor  = pink,
		urlcolor=pink
	}
}

\newcommand{\rhob}{\rho_b}

\newcommand{\bmu}{\bm{u}}
\newcommand{\bmv}{\dot{\bm{u}}}
\newcommand{\bma}{\ddot{\bm{u}}}
\newcommand{\dt}{\Delta t}

\renewcommand{\Omega}{\varOmega}

\allowdisplaybreaks[4]
\newtheorem{defi}{Definition}[section]

\newcounter{alg}
 
\crefname{alg}{Alg.}{Algs.}           
\newcommand{\novelalgref}[1]{{\emph{Algorithm}}~\ref{#1}}   
\newcommand{\novelalgrefs}[2]{{\emph{Algorithms}}~\ref{#1} and \ref{#2}}  
\newcommand{\novelalgrefss}[2]{{\emph{Algs.}}~\ref{#1} and \ref{#2}}  
\crefname{figure}{Fig.}{Figs.} 
\crefname{subfigure}{Fig.}{Figs.} 
\crefname{table}{Table}{}
\crefname{equation}{Eq.}{Eqs.}
\newtheorem{theorem}{Theorem}[section]

\begin{document}
	
	\begin{frontmatter}
		
		\title{{\bf Two self-starting single-solve third-order explicit integration algorithms\\for second-order nonlinear dynamics}}
		
		\author[hit]{Yaokun Liu}
		\ead{liuyk2001@gmail.com}
		\author[hit]{Jinze Li\corref{cor1}}
		\ead{pinkie.ljz@gmail.com}
		\author[hit]{Kaiping Yu\corref{cor1}}
		\ead{kaipingyu1968@gmail.com}
		
		\cortext[cor1]{Corresponding authors.}

		
		\affiliation[hit]{organization={School of Astronautics, Harbin Institute of Technology},
			addressline={No.~92 West Dazhi Street},
			city={Harbin},
			postcode={150001},
			country={China}}
		
		\begin{abstract}
		The single-step explicit time integration methods have long been valuable for solving large-scale nonlinear structural dynamic problems, classified into single-solve and multi-sub-step approaches. However, no existing explicit single-solve methods achieve third-order accuracy. The paper addresses this gap by proposing two new third-order explicit algorithms developed within the framework of self-starting single-solve time integration algorithms, which incorporates 11 algorithmic parameters. The study reveals that fully explicit methods with single-solve cannot reach third-order accuracy for general dynamic problems. Consequently, two novel algorithms are proposed: \textit{Algorithm} 1 is a fully explicit scheme that achieves third-order accuracy in displacement and velocity for undamped problems; \textit{Algorithm} 2, which employs implicit treatment of velocity and achieves third-order accuracy for general dynamic problems. Across a suite of both linear and nonlinear benchmarks, the new algorithms consistently outperform existing single-solve explicit methods in accuracy. Their built-in numerical dissipation effectively filters out spurious high-frequency components, as demonstrated by two wave propagation problems. Finally, when applied to the realistic engineering problem, both of them deliver superior numerical precision at minimal computational cost.
		\end{abstract}

		
		
	\end{frontmatter}
	
	\section{Introduction}

	Explicit time integration methods have demonstrated exceptional efficiency in solving nonlinear dynamic problems and have gained widespread recognition in large-scale engineering simulations. Unlike implicit methods, explicit algorithms do not require solving linear equations at each time step for a linear system, making them particularly suitable for handling systems with millions of degrees of freedom and it avoids the complexity of iterative processes for a nonlinear system. This efficiency makes explicit methods the preferred choice for many practical applications, especially when computational resources are limited or large-scale problems need to be solved efficiently.
	
	Over the past decades, various explicit algorithms have been developed to tackle challenges posed by dynamic systems. These methods are typically divided into two categories: linear multi-step methods and single-step methods. While linear multi-step methods~\cite{zhaiTWOSIMPLEFAST1996,yangTwoDynamicExplicit2014,yangImprovedExplicitIntegration2020,liuExplicitIntegrationMethod2023} can provide reasonable accuracy, they are not self-starting and require additional steps or integration techniques to initiate calculations. In single-step methods, they are further classified into single-sub-step methods and multi-sub-step methods based on the number of the equation system to be solved. By using composite sub-step techniques~\cite{tarnowHowRenderSecond1994}, multi-sub-step explicit methods~\cite{nohExplicitTimeIntegration2013,liTwoThirdorderExplicit2022,liEnhancedSecondOrderExplicit2023,liSecondorderSubstepExplicit2023,liSuiteSecondorderComposite2023,liDirectlySelfstartingSecondorder2024,wang_GeneralizedSinglestepMultistage_2025} have achieved significant advancements, offering higher-order accuracy and larger stability domains. However, multi-sub-step methods involve more complex programming, requiring multiple calls to the equilibrium equation within a single time step, which increases the computational cost. Single-sub-step (A.k.a. Single-solve ) explicit algorithms~\cite{chungNewFamilyExplicit1994,liIdenticalSecondOrder2021,zhaoSelfstartingDissipativeAlternative2023,kimSimpleExplicitSingle2019} , due to their programming simplicity and theoretical maturity, remain the most widely used. These methods require only one solution per time step, making them highly efficient for large-scale nonlinear problems. Therefore, the main research objects in this paper also focus on the development of self-starting single-solve explicit algorithms.
	
	The earliest single-solve explicit time integration methods could be traced back to the work of Fox and Goodwin~\cite{foxNewMethodsNumerical1949}. This method necessitates both the calculation of initial acceleration and the additional computation (${\bmu}_{-1}$) of the initial displacement, classifying it as a linear multi-step method rather than a self-starting one. Zhao et al.~\cite{zhaoSelfstartingDissipativeAlternative2023} demonstrated that the central difference (CD) method ~\cite{foxNewMethodsNumerical1949}, half-step central difference method~\cite{cookConceptsApplicationsFinite2007}, and Newmark explicit (NE) method~\cite{newmarkMethodComputationStructural1959} are algebraically identical under fixed time step sizes. These three explicit methods have the same spectral properties, the order of accuracy , the implicit treatment of velocity and stability. The only difference is that NE is self-starting, while the other two are not.
	In general, implicit time integration methods can be paired with corresponding explicit algorithms in the form of predictor-corrector schemes~\cite{hughesImplicitexplicitFiniteElements1978,hughesImplicitexplicitFiniteElements1978a,hilberImprovedNumericalDissipation1977,hughesAnalysisTransientAlgorithms1983,hulbertExplicitTimeIntegration1996}. Hulbert et al.~\cite{hulbertExplicitTimeIntegration1996} analyzed the predictor-corrector scheme of the TPO/G-$\alpha$ method~\cite{shaoThree,chungTimeIntegrationAlgorithm1993} and initially referred to it as the EG-$\alpha$ method. Li et al.~\cite{liIdenticalSecondOrder2021} pointed out that the algorithm~\cite{hulbertExplicitTimeIntegration1996} should actually be the predictor-corrector scheme of the WBZ-$\alpha$ method~\cite{woodAlphaModificationNewmark1980}, rather than TPO/G-$\alpha$. Therefore this algorithm~\cite{hulbertExplicitTimeIntegration1996} is abbreviated as the EWBZ-$\alpha$ method in this paper, while the true predictor-corrector scheme corresponding to the TPO/G-$\alpha$ method is given in the literature~\cite{liIdenticalSecondOrder2021}, referred to as the EG-$\alpha$ method. The CL method~\cite{chungNewFamilyExplicit1994} achieves second-order accuracy in both displacement and velocity, and first-order accuracy in acceleration. Kim~\cite{kimSimpleExplicitSingle2019} developed an enhanced version, the ICL method, which improves the accuracy of the CL method to identical second-order accuracy. Nonetheless, both methods require $1/2\leq\rhob\leq1$, limited in their ability to control numerical dissipation over the full range.
	Subsequently, Li et al.~\cite{liIdenticalSecondOrder2021} proposed an optimal explicit algorithm (GSSE), which achieves the identical second-order accuracy, a full range of dissipation control at the bifurcation point ($0\leq\rhob\leq1$) and the maximal conditionally stable region. Given the widespread use of the CD method but its numerous limitations (non-self-starting and zero dissipation), Zhao et al.~\cite{zhaoSelfstartingDissipativeAlternative2023} further generalized GSSE~\cite{liIdenticalSecondOrder2021}  with the implicit treatment of velocity and proposed the GSSI method. Compared to the CD method, GSSI~\cite{zhaoSelfstartingDissipativeAlternative2023} provides a significantly larger conditional stability domain. It is worth noting that when the parameters are chosen as $\rhob=1$ and $\rho_s=0$, GSSI~\cite{zhaoSelfstartingDissipativeAlternative2023} degenerates into the NE~\cite{newmarkMethodComputationStructural1959} method. The review article \cite{LXJZ20250311001} offers a thorough examination and discussion of self-starting single-solve explicit algorithms, highlighting the GSSE~\cite{liIdenticalSecondOrder2021} and GSSI~\cite{zhaoSelfstartingDissipativeAlternative2023} methods as the optimal methods currently available in this field. The key distinction between the two algorithms lies in their treatment of velocity, either explicitly or implicitly. For a more detailed analysis, readers are referred to the literature~\cite{LXJZ20250311001}.
	
	Usually, self-starting high-order explicit algorithms are often designed in  multi-sub-step schemes. Kim and Reddy~\cite{kimNovelExplicitTime2020} proposed two two-sub-step explicit algorithms (Kim–IJMS3/4), where Kim–IJMS3 and Kim–IJMS4 are capable of achieving third-order accuracy in displacement and velocity, as well as second-order accuracy in acceleration when solving general dynamic problems ($\xi\neq0$). Li et al.~\cite{liTwoThirdorderExplicit2022} introduced two two-sub-step algorithms, referred to as SS2HE\textsubscript{1} and SS2HE\textsubscript{2}, which maintain identical third-order accuracy and achieve controllable numerical dissipation throughout the entire process. Wang et al.~\cite{wang_GeneralizedSinglestepMultistage_2025} introduced two two-sub-step algorithms that also attain third-order accuracy in displacement and velocity but second-order accuracy in acceleration. The first scheme proposed  by Wang et al.~\cite{wang_GeneralizedSinglestepMultistage_2025} exhibits spectral characteristics identical to the previously published SS2HE\textsubscript{2}. 
	Additionally, Liu and Guo~\cite{liu_NovelPredictorCorrector_2021a} proposed a three-sub-step explicit integration algorithm, which achieves identical third-order accuracy when the parameters are set as $\alpha = 1/18$ and $\beta = 1/6$. These advancements collectively establish a consensus that achieving high-order accuracy necessitates the use of multi-sub-step algorithms.
	
	As previously discussed, the currently available single-solve explicit algorithms have achieved identical second-order accuracy at most. Algorithms that attain high-order accuracy are exclusively multi-sub-step methods. None of  self-starting single-solve explicit algorithms have reached high-order accuracy. To address this gap, this paper introduces two novel self-starting single-solve explicit time integration algorithms that achieve third-order accuracy. These algorithms require only a single solution to the equilibrium equation per time step, thus maintaining computational efficiency while enhancing accuracy. The key features of the proposed algorithms include self-starting property, single solution per time step, explicit integration, third-order accuracy, numerical dissipation at the bifurcation point, and no adjustable parameters.
	
	The remainder of this paper is organized as follows: Section \ref{sec:structure} presents the framework of the self-starting single-solve time integration method, which represents any self-starting single-solve time integration algorithm through 11 algorithmic parameters. Section \ref{sec:optimization} focuses on optimizing the parameters to achieve third-order accuracy, discusses stability, and proposes the novel algorithms. Section \ref{sec:properties} analyzes and compares the properties of the new and published algorithms, including spectral characteristics, amplitude and phase errors, and convergence rate tests. Section \ref{sec:examples} provides computational verification of the new algorithms using various tests and compares them with the published algorithms. These tests include nonlinear systems such as the Van der Pol system, spring pendulum system; one-dimensional impact rod and two-dimensional membrane for testing numerical dissipation; as well as a large-scale nonlinear practical engineering problem. Section \ref{sec:conclusion} presents a summary of the new algorithms.
	
	\section{The self-starting single-solve algorithms}
	\label{sec:structure}
	Following spatial discretization (e.g., finite element, isogeometric analysis or virtual element methods), the partial differential equations governing the structural dynamic behavior can be transformed into the semi-discretized equations of motion as shown in \cref{eq1}. It is essentially a second-order ordinary differential equations in terms of the time variable.
	\begin{equation}
		\bm{M\bma}(t) = \bm{f}(\bmv(t),~\bmu(t),~t)\label{eq1}
	\end{equation}
	where $\bm{M}$ denotes the global mass matrix after spatial discretization and $\bm{f}$ represents the total force vector acting on the structure, including both internal force $\bm{p}(\bmu,\bmv)$ and external forces $\bm{q}(t)$, which depend on the solution variables $\bmu(t)$, $\bmv(t)$ and time $t$. The vector $\bmu(t)$ represents the unknown displacement; $\bmv(t)$ and $\bma(t)$ represent the unknown nodal velocity and acceleration vectors, respectively. To ensure a unique solution to \cref{eq1}, initial conditions are provided as ${\bmu}_0 = \bmu(t_0)$ and ${\bmv}_0 = \bmv(t_0)$, where $t_0$ represents the initial time.
	
	It is well-known that directly integrating \cref{eq1} using time integration methods may require additional initial conditions to initiate the numerical integration process. Most of self-starting algorithms require computing the initial acceleration $\bma_0$:
	\begin{equation}\label{eq:2}
		\bma_0  = \bm{M}^{-1}\bm{f}(\bmv_0,~\bmu_0,~t_0).
	\end{equation}
	To solve \cref{eq1}, the following self-starting integration scheme is presented:
	
	\begin{mdframed}
		The acceleration $\bma_{n+p}$ at $t_{n+p}:=t_n+p\dt$ is first calculated by
		\begin{subequations}\label{eqstru}
			\begin{align}
				\bm{M}{{\bma}}_{n+p} & = \bm{f}\left(\bmu_{n+p}, \ \bmv_{n+p}, \ t_{n+p}\right), \label{eqstrunonlinear}                                   \\
				{{\bmu}_{n+p}}       & = {{\bmu}_{n}} + p\dt{{{\bmv}}_{n}} + \dt^{2} ( \alpha_{1}{{\bma}}_{n} + \alpha_{2}{{\bma}}_{n+p} ), \label{eq:pc1} \\
				{{\bmv}}_{n+p}       & = \bmv_{n} + \dt ( \alpha_{3}{{\bma}}_{n} + \alpha_{4}{{\bma}}_{n+p} ), \label{eq:pc2}
			\end{align}
			and solutions at $t_{n+1}$ are updated by
			\begin{align}
				{{\bmu}_{n+1}} & = {{\bmu}_{n}} + \dt {{\bmv}}_{n} + \dt^2 ( \alpha_5 \bma_n + \alpha_6 \bma_{n+p} ), \\
				{{\bmv}}_{n+1} & = {{\bmv}}_{n} + \dt ( \alpha_7 \bma_n + \alpha_8 \bma_{n+p} ),                      \\
				{{\bma}}_{n+1} & = \alpha_9 \bma_n + \alpha_{10} \bma_{n+p},
			\end{align}
		\end{subequations}
		where $p \in \mathbb{R}$ and ${\alpha_{j}} \in \mathbb{R} ~(j = 1,~\cdots,~10)$ are algorithmic parameters to be designed.
	\end{mdframed}
	
	As described in \cref{eqstrunonlinear}, the algorithm framework requires only a single solution to the balance equation, making all methods within it single-solve. Studies~\cite{LXJZ20250311001,ZDGC20250116002,wang_GeneralizedSinglestepMultistage_2025}  have demonstrated that with the proper selection of the 11 parameters (${\alpha_{j}}$ for $j = 1,~\dots,~10$ and $p$), this framework is capable of encompassing nearly all current self-starting time integration schemes. Thus, the Butcher table for defining the self-starting single-solve time integration algorithms is given by:
	\begin{equation}
		\begin{array}{c|cc|cc|cc}
			p & \alpha_{1} & \alpha_{2} & \alpha_{3} & \alpha_{4} &                              \\ \hline
			& \alpha_{5} & \alpha_{6} & \alpha_{7} & \alpha_{8} & \alpha_{9} & \alpha_{10}
		\end{array}
		\label{eq:tabular}
	\end{equation}
	
	Once the Butcher table of the algorithm is provided, one self-starting single-solve algorithm can be determined succinctly. In \ref{app1}, various Butcher tables for some published explicit algorithms are presented. For a more in-depth discussion on the self-starting single-solve Butcher table mentioned above, please refer to the literature~\cite{LXJZ20250311001}.
	
	When the self-starting single-solve framework \eqref{eqstru} is employed to solve linear problems, the force vector simplifies to $\bm{f} = \bm{q}(t) - \bm{C}\bmv(t) - \bm{K}\bmu(t)$ where $\bm{C}$ and $\bm{K}$ are the global damping and stiffness matrices respectively. Substituting \cref{eq:pc1,eq:pc2} into \cref{eqstrunonlinear} yields a linear system of the form $\bm{\widetilde{M}} \, \bma_{n+p} =\bm{\widetilde{f}}$. The effective mass matrix $\bm{\widetilde{M}} = \bm{M} + {\alpha_4} \dt \bm{C} + {\alpha_2} \dt^{2} \bm{K}$ must be considered. Since the solution procedure requires the inversion of $\bm{\widetilde{M}}$, the computational cost of the algorithm is dictated by the composition of the effective mass matrix. Typically, the global damping matrix $\bm{C}$ and the global stiffness matrix $\bm{K}$ are generally non-diagonal, leading to high computational costs when calculating the inverse of $\bm{\widetilde{M}}$. If $\bm{\widetilde{M}}$ does not include $\bm{C}$ and $\bm{K}$, then
	\begin{equation} \label{eq:fullexp}
		\alpha_2 = \alpha_4 = 0.
	\end{equation}
	
	The algorithm shown in \eqref{eqstru} that satisfies \cref{eq:fullexp} is referred to as a self-starting fully explicit time integration algorithm. Such algorithms do not require the Newton-Raphson iterative method for solving any nonlinear problems, significantly reducing computational overhead. This makes explicit algorithms particularly well-suited for solving nonlinear problems.
	
	When $\bma_{n+p}$ is employed exclusively for the prediction of velocity, one should require
	\begin{equation}\label{eq:7}
		\alpha_{2} = 0,\ \alpha_{4} \neq 0.
	\end{equation}
	
	Such algorithms are referred to as explicit algorithms with the implicit treatment of the velocity. It can be seen that when solving linear problems, the algorithm remains explicit only if the mass and damping matrices are diagonal. However, when solving nonlinear damping problems, the Newton iterative method is still required. Examples of such algorithms include the CD and NE~\cite{newmarkMethodComputationStructural1959} methods. Although explicit algorithms with the implicit treatment of velocity require additional computational work when solving damping problems, they offer greater stability and more accurate solutions. Both fully explicit algorithms and explicit algorithms with the implicit treatment of the velocity are the subjects to be developed in this research.
	
	\section{Optimization of parameters}
	\label{sec:optimization}
	
	The standard single-degree-of-freedom (SDOF) modal system is commonly used to analyze numerical characteristics of time integration methods.
	\begin{equation}\label{eq:single}
		\ddot{u}(t) + 2\xi \omega \dot{u}(t) + \omega^2 u(t) = f(t)
	\end{equation}
	with initial conditions $\dot{u}({{t}_0}) = {\dot{u}}_0$ and $u({{t}_0}) = {u}_0$, where $\xi \in [0,\ 1)$ and $\omega \in (0,\ +\infty )$ represent the physical damping ratio and undamped natural frequency, respectively; $f(t)$ is the modal external force; $u(t)$, $\dot{u}(t)$ and $\ddot{u}(t)$ represent displacement, velocity and acceleration, respectively. 
	
	When the algorithm framework~\eqref{eqstru} with the fixed time step $\dt$ is applied to solve the SDOF system \eqref{eq:single}, the following single-step recursive scheme can be obtained as
	\begin{equation}\label{eq:singlestep}
		{{\bm{X}}_{n+1}} = {{\bm{D}}_{\text{num}}} \cdot {{\bm{X}}_n} + {{\bm{L}}_{\text{num}}}({{t}_n})
	\end{equation}
	where  ${\bm{X}}_{n} = \left[u_n\quad  \dot{u}_n \quad \ddot{u}_n  \right]^\mathsf{T}$ and the numerical amplification matrix ${\bm{D}}_{\text{num}}$ is given as
	\begin{equation}\label{amplification:matrix}
		\bm{D_{\text{num}}} = \frac{1}{D}
		\left[
		\begin{array}{ccc}
			\displaystyle 1 + (\alpha_{2} - \alpha_{6}) \Omega^{2} + 2 \Omega \xi \alpha_{4} & \displaystyle  \dt D_{12}                             & \displaystyle \dt^2 D_{13} \\[1ex]
			\displaystyle  - \Omega \omega \alpha_{8}                                        & \displaystyle D_{22}                                  & \displaystyle \dt D_{23}   \\[1ex]
			\displaystyle - \omega^{2} \alpha_{10}                                           & \displaystyle - \omega \alpha_{10} (\Omega p + 2 \xi) & \displaystyle D_{33}
		\end{array}
		\right]
	\end{equation}
	where $\Omega := \omega \dt$, $D   =  \alpha_{2}\Omega^{2} + 2\alpha_{4}\xi  \Omega  + 1 $ and
	\begin{equation}\label{amplification:coefficient}
		\begin{aligned}
			D_{12} & = (\alpha_{2} - p \alpha_{8}) \Omega^{2} + 2  (\alpha_{4} - \alpha_{8})\xi\Omega  + 1
			& D_{13}                                                                                                                                       & =  (\alpha_{2} \alpha_{5} - \alpha_{1} \alpha_{6} ) \Omega^{2} + 2  (\alpha_{4} \alpha_{5} - \alpha_{3} \alpha_{6})\xi \Omega + \alpha_{5}   \\
			D_{22} & = (\alpha_{2} - p \alpha_{6}) \Omega^{2} + 2  (\alpha_{4} - \alpha_{6})\xi\Omega  + 1
			& D_{23}                                                                                                                                       & =    (\alpha_{2} \alpha_{7} - \alpha_{1} \alpha_{8} ) \Omega^{2} + 2  ( \alpha_{4} \alpha_{7}- \alpha_{3} \alpha_{8})\xi \Omega + \alpha_{7} \\
			D_{33} & = ( \alpha_{2} \alpha_{9}- \alpha_{1} \alpha_{10} ) \Omega^{2} - 2  (\alpha_{10} \alpha_{3} - \alpha_{4} \alpha_{9})\xi\Omega  + \alpha_{9}.
		\end{aligned}
	\end{equation}
	The numerical load operator ${\bm{L}_{\text{num}}}({{t}_{n}})$ as defined in \cref{eq:singlestep} is expressed as
	\begin{equation}
		\label{externalforce}
		\bm{L_{\text{num}}}(t_n) = \frac{f(t_{n+p})}{D}
		\begin{bmatrix}
			\displaystyle \dt^{2} \alpha_{6} &
			\displaystyle \dt \alpha_{8}     &
			\displaystyle \alpha_{10}
		\end{bmatrix}^\mathsf{T}.
	\end{equation}
	
	The characteristic polynomial of $\bm{D}_{\text{num}}$ is computed as 
	\begin{equation}\label{eq:lambda}
		\det( \lambda {\bm{I}}_{3} - {\bm{D}}_{\text{num}} ) = \sum_{j=0}^{3} A_{j} \lambda^{j} = 0
	\end{equation}
	where ${\bm{I}}_3$ is the identity matrix of dimension $3$, and the coefficients $A_j$ for $j = 0, 1, 2, 3$ are outlined as
	\begin{subequations}\label{eigenvalues}
		\begin{align}
			A_3 & = 1                                                                                                                                                                                                                                                                                                                                                   \\
			A_2 & = \frac{\left[  p \alpha_{8} + \alpha_{1} \alpha_{10} + \alpha_{6}- (\alpha_{9}+2) \alpha_{2}  \right] \Omega^{2}
				- 2  \left[ (\alpha_{9}+2) \alpha_{4} - \alpha_{10} \alpha_{3} - \alpha_{8} \right]\xi \Omega - \alpha_{9} - 2 }{\alpha_{2}\Omega^{2} + 2\alpha_{4}\xi  \Omega  + 1}                                                                                                                                                                                        \\
			A_1 & = \frac{\left\{\begin{array}{l}
					\left[ (2 \alpha_{2} - p \alpha_{8} - \alpha_{6}) \alpha_{9} + (\alpha_{7} p - 2 \alpha_{1} + \alpha_{5}) \alpha_{10}  + (1 - p) \alpha_{8} + \alpha_{2} - \alpha_{6} \right] \Omega^{2} \\
					+ 2  \left[ (2 \alpha_{4} - \alpha_{8}) \alpha_{9} + (\alpha_{7} - 2 \alpha_{3}) \alpha_{10} - \alpha_{8} + \alpha_{4} \right]\xi \Omega + 2 \alpha_{9} + 1
				\end{array}\right\}}{\alpha_{2}\Omega^{2} + 2\alpha_{4}\xi  \Omega  + 1} \\
			A_0 & = \frac{\left\{
				\begin{array}{l}
					\big( \left[ (p - 1) \alpha_{8} - \alpha_{2} + \alpha_{6} \right] \alpha_{9} - \left[ (p - 1) \alpha_{7} - \alpha_{1} + \alpha_{5} \right] \alpha_{10} \big) \Omega^{2} \\
					- 2 \xi \left[ (\alpha_{4} - \alpha_{8}) \alpha_{9} + \alpha_{10} (\alpha_{7} - \alpha_{3}) \right] \Omega - \alpha_{9}
				\end{array}
				\right\}}{\alpha_{2}\Omega^{2} + 2\alpha_{4}\xi  \Omega  + 1}.
		\end{align}
	\end{subequations}

	The accuracy analysis often neglects the influence of external forces (assuming $f(t) \equiv 0$). In prior work~\cite{liDesigningDevelopingSingle2023}, based on the theory of linear multistep methods~\cite{hughesFiniteElementMethod2012}, a general formula to analyze the order of accuracy for any time integration algorithm was proposed, considering the impact of external forces on accuracy.
	
	With a fixed time step, the single-step recursive scheme \eqref{eq:singlestep} can equivalently be expressed in linear multistep form as
	\begin{equation}\label{eq:multisteps}
		\sum_{j=0}^3 A_j {\bm{X}}_{n+j} = \sum_{j=0}^3 A_j \sum_{i=0}^{j-1} {\bm{D}}_{\text{num}}^i \cdot {\bm{L}}_{\text{num}}(t_{n+j-i-1}).
	\end{equation}
	
	\begin{defi}
		For solving the standard SDOF system, the local truncation errors of the linear multistep format \eqref{eq:multisteps} is defined as
		\begin{equation}\label{eq:accuracy}
			\bm{\tau} = \frac{1}{\dt^2} \left( \sum_{i=0}^{3} A_i \bm{X}(t_{n+i}) - \sum_{i=0}^{3} A_i \sum_{j=0}^{i-1} {\bm{D}}_{\mathrm{num}}^j {\bm{L}}_{\mathrm{num}}(t_{n-j+i-1}) \right)
		\end{equation}
		where $\bm{X}(t_{n+i})$ corresponds to the exact quantity of the numerical iteration vector. $\bm{\tau}=\left[\tau_u \quad \tau_{\dot{u}} \quad \tau_{\ddot{u}}\right]^\mathsf{T}$ defines the local truncation errors for displacement, velocity, and acceleration. If $\tau_{\square} = \mathcal{O}(\dt^r)$, the method exhibits $r$-th-order accuracy for that variable. Similarly, if $\bm{\tau} = \mathcal{O}(\dt^r)$, the corresponding method achieves identical $r$-th-order accuracy.
	\end{defi}
	
	Assuming that the external force $f(t) \in C^\infty$, ${\bm{L}}_{\mathrm{num}}(t_{n+i})$ are expanded using a Taylor series around $t = t_n$. Substituting Eqs.~\eqref{amplification:matrix}-\eqref{externalforce} and \eqref{eigenvalues} into \cref{eq:accuracy} yields
	\begin{equation}
		\begin{aligned}
			\tau_u & = \ddot{u}(t_n) \left( 1 - \alpha_7 \alpha_{10} + \alpha_8 \alpha_9 - \alpha_8 - \alpha_9 \right) \\
			& + \left\{
			\begin{array}{l}
				\Big[ \left( 2 \alpha_7 \alpha_4 + 2 \alpha_3 + \alpha_7 \right) \alpha_{10}
				+ \left[ 2 \left( 1 - \alpha_9 \right) \alpha_4 - \alpha_9 + 3 \right] \alpha_8 + 2 \alpha_9 - 4 \Big] \xi \omega \ddot{u}(t_n) \\
				+ \Big[ \left( p \alpha_9 - p - 1 \right) \alpha_8 + \left( \alpha_6 - 1 \right) \alpha_9 - \left( p \alpha_7 + \alpha_5 \right) \alpha_{10} - \alpha_6 + 2 \Big]
				\left( \dot{f}(t_n) - \omega^2 \dot{u}(t_n) \right)
			\end{array}
			\right\} \dt + \mathcal{O}(\dt^2).
		\end{aligned}
	\end{equation}
	For $\tau_{u}=\mathcal{O}(\dt^2)$, it is necessary to meet the following conditions:
	\begin{subequations}\label{eq:17}
		\begin{align}
			1 - \alpha_7 \alpha_{10} + \alpha_8 \alpha_9 - \alpha_8 - \alpha_9 &= 0 \\
			\left( 2 \alpha_7 \alpha_4 + 2 \alpha_3 + \alpha_7 \right) \alpha_{10} + \left[ 2 \left( 1 - \alpha_9 \right) \alpha_4 - \alpha_9 + 3 \right] \alpha_8 + 2 \alpha_9 - 4 &= 0 \\
			\left( p \alpha_9 - p - 1 \right) \alpha_8 + \left( \alpha_6 - 1 \right) \alpha_9 - \left( p \alpha_7 + \alpha_5 \right) \alpha_{10} - \alpha_6 + 2 &= 0.
		\end{align}
	\end{subequations}
	After achieving second-order accuracy for displacement, the local truncation errors for velocity can be simplified as
	\begin{align}
		\tau_{\dot{u}} = \left[ \left( 2p - 1 \right) \alpha_9 - 2p - 2 \alpha_8 + 3 \right] \frac{\ddot{f}(t_n)}{2} \dt + \mathcal{O}(\dt^2).
	\end{align}
	The following equation is required to ensure $\tau_{\dot{u}}=\mathcal{O}(\dt^2)$:
	\begin{equation}\label{eq:2accuracy2}
		(2p-1)\alpha_9-2p-x\alpha_8+3=0.
	\end{equation}
	Once \cref{eq:17,eq:2accuracy2} are satisfied, the local truncation error for acceleration is calculated as 
	\begin{equation}\label{eq:200}
		\begin{aligned}
			\tau_{\ddot{u}}  = \left( 1 - \alpha_{10} - \alpha_9 \right) \ddot{f}(t_n)
			+ \Big( \left[ 2 - \alpha_9 - (p + 1) \alpha_{10} \right] \dddot{f}(t_n) + 2 \alpha_4 \left( \alpha_{10} + \alpha_9 - 1 \right) \ddot{f}(t_n) \omega \xi \Big) \dt + \mathcal{O}(\dt^2).
		\end{aligned}
	\end{equation}
	
	Upon simplifying \cref{eq:200} by using \cref{eq:17,eq:2accuracy2} and setting $\tau_{\ddot{u}} = \mathcal{O}(\dt^2)$, the conditions for identical second-order accuracy are summarized as
	\begin{align}\label{eq:2accuracy3}
		\alpha_3 + \alpha_4=p \qquad	\alpha_5 + \alpha_6=\frac{1}{2} \qquad \alpha_7+\alpha_8=1 \qquad \alpha_9+\alpha_{10}=1 \qquad \alpha_8 = \frac{1}{2p}   \qquad \alpha_{10} = \frac{1}{p}.
	\end{align}
	
	At this stage, identical second-order accuracy has been successfully attained. After achieving consistency (at least first-order accuracy), the time integration method should be stable. The spectral radius is commonly used to describe the spectral stability of an algorithm. An algorithm is spectrally stable if
	\begin{equation}\label{eq:stab}
		\rho := \max (|\lambda_j|,\ j=1,\ 2,\ 3) \leq 1
	\end{equation}
	where $\rho$ is the spectral radius and $\lambda_j$ are the eigenvalues of the amplification matrix $\bm{D_{\text{num}}}$. It is well known that explicit time integration methods are generally conditionally stable. Directly deriving the conditionally stable domain for explicit algorithms through \cref{eq:lambda} is tedious and difficult. The Routh-Hurwitz criterion~\cite{lambertComputationalMethodsOrdinary1973} provides a method for deriving the conditionally stable domain. For an  amplification matrix $\bm{D_{\text{num}}}$ of size $s = 3$, the inequality constraints of the conditionally stable domain can be expressed as
	\begin{subequations}\label{eq:staba}
		\begin{align}
			B_0 & = A_0 + A_1 + A_2 + A_3          \ge 0 \\
			B_1 & = -3A_0 - A_1 + A_2 + 3A_3       \ge 0 \\
			B_2 & = 3A_0 - A_1 - A_2 + A_3         \ge 0 \\
			B_3 & = -A_0 + A_1 - A_2 + A_3         \ge 0 \\
			B_4 & = B_1 \cdot B_2 - B_0 \cdot B_2  \ge 0.
		\end{align}
	\end{subequations}
where $A_i$ ($i=0,\dots,3$) are the principal invariants of the amplification matrix $\bm{D}_{\text{num}}$ as shown in \cref{eigenvalues}.

\begin{theorem} \label{th:1}
	A self‐starting time integration algorithm which satisfies fully explicit formulation and single‐solve per time step, achieves at most identical 2nd‐order accuracy when $\xi \neq 0$.
\end{theorem}
\begin{proof}
	For the third-order accuracy, reconsider the local truncation error for displacement using \cref{eq:fullexp}:	
\begin{equation}
	\begin{aligned}\label{eq:24}
		\tau_u  =                                                                 
		  \frac{2\left( 6p - 1 \right)  \omega\xi \dddot{u}(t_n) + \left( 12p \alpha_6 - 6p + 12 \alpha_1 - 1 \right) \omega^2 \ddot{u}(t_n) + \left( 1 - 6p^2 - 12p \alpha_6 + 6p \right) \ddot{f}(t_n)}{12p} \dt^2 + \mathcal{O}(\dt^3).
	\end{aligned}
\end{equation}
In order to achieve third-order accuracy in displacement, the following conditions must be satisfied:
\begin{subequations}\label{eq:37}
	\begin{align}
		6p - 1 &= 0 \\
		1 - 6p^2 - 12p \alpha_6 + 6p &= 0  \\
		12p \alpha_6 - 6p + 12 \alpha_1 - 1 &= 0.
	\end{align}
\end{subequations}
Solving these equations results in
\begin{align}\label{eq:38}
\alpha_1 = \frac{1}{72},	 \quad  \alpha_6 = \frac{11}{12}, \quad p = \frac{1}{6}.
\end{align}
Once third-order accuracy for displacement is satisfied, velocity and acceleration response which are critical aspects of dynamic problems, should be examined. Recalculating the truncation errors for velocity and acceleration
\begin{align}
	\tau_{\dot{u}}=-\frac{\dddot{f}(t_n)}{12}\dt^2+\mathcal{O}(\dt^3),\qquad \tau_{\ddot{u}}=\frac{5\ddddot{f}(t_n)}{12}\dt^2+\mathcal{O}(\dt^3),
\end{align}
 shows that these variables remain second-order accurate, and third-order accuracy is not attainable.

The coefficients $ B_j \ (j = 0, \dots, 4)$ in \cref{eq:staba} are provided as
\begin{equation}\label{eq:66}
	\begin{alignedat}{3}
	&B_4  = 4 \Omega^{4}+16 \xi\Omega^{3}  -384\xi^{2} \Omega^{2} +576\xi \Omega   \geq 0 , &B_3  = 2 \Omega^{2}-8\xi \Omega   -16\geq 0 ,   \\
	&B_2 = -4 \Omega^{2}-16 \xi\Omega   +24 \geq 0    ,   \qquad \qquad
	B_1  = -4 \Omega^{2}+24 \xi\Omega    \geq 0,  &B_0  =  6 \Omega^{2}\geq 0  .
\end{alignedat}
\end{equation}
where $ \Omega\in(0,\infty) $ and $ \xi\in[0,1) $, the conditions 
\begin{subequations}
	\begin{align}
&B_3 \ge 0 \quad\Longrightarrow\quad \Omega \in \bigl[2\sqrt{\xi^2+2} + 2\xi,\;\infty\bigr), \\
&B_2 \ge 0 \quad\Longrightarrow\quad \Omega \in \bigl(0,\;\sqrt{4\xi^2+6} - 2\xi\bigr].
	\end{align}
\end{subequations}
Hence, one requires
\begin{equation}
	\begin{cases}
		2\sqrt{\xi^2+2} + 2\xi \le \sqrt{4\xi^2+6} - 2\xi,\\
		0 \le \xi < 1
	\end{cases}
	\Longrightarrow \quad \xi\in\varnothing.
\end{equation}
Therefore, there do not exist $ \Omega\in(0,\infty) $ and $ \xi\in[0,1) $ that satisfy the inequalities~\eqref{eq:66}.
\end{proof}
In summary, the integration framework \cref{eqstru} is unconditionally unstable for solving damping problems after achieving higher order accuracy in displacement.
	
	 \subsection{Fully explicit algorithms without physical damping}
	
	Based on the Theorem~\ref{th:1}, it is evident that for generalized dynamic problems ($\xi\neq 0$), the fully explicit algorithm can not achieve high-orger accuracy. Therefore, the undamped problem ($\xi=0$) is considered. Substituting $\xi = 0$ into \cref{eq:24} and setting $\tau_{u} = \mathcal{O}(\dt^3)$ gives
	\begin{align}\label{eq:39}
		\alpha_6 &= -\frac{1 - 6p^2 + 6p}{12p}, \quad \alpha_1 = \frac{1}{2} p^2.
	\end{align}
	Taking into account third-order accuracy for velocity as described in \cref{eq:accuracy}, $\tau_{\dot{u}}$ is calculated as
	\begin{align}\label{eq:27}
		\tau_{\dot{u}} = -\frac{(6p^2 - 6p + 1) \dddot{f}(t_n)}{12p} \dt^2 + \mathcal{O}(\dt^3).
	\end{align}
	By setting $\tau_{\dot{u}} = \mathcal{O}(\dt^3)$, the parameter $p$ can be derived as
	\begin{align}\label{eq:28}
		p = \frac{1}{2} \pm \frac{\sqrt{3}}{6}.
	\end{align}
	
	 Under these conditions, the parameters of the fully explicit algorithm, which achieves third-order accuracy in both displacement and velocity, are fully determined in the absence of physical damping ($\xi=0$). The coefficients $ B_j \ (j = 0, \dots, 4)$ are calculated as 
	\begin{equation}
		\begin{aligned}\label{eq:6}
			&B_4  = \frac{2 \Omega^{4}}{3 p}\geq0,  \qquad
			&B_3  =& \frac{{2\left( {1 - 3p} \right){\Omega ^2} + 24p - 12}}{{3p}}\geq0, \qquad
			&B_2  =& \frac{{12 - 2{\Omega ^2}}}{{3p}}\geq0,         \\
			&B_1  = \frac{\left(2 p-1\right) \Omega^{2}}{p}\geq0,         \qquad
			&B_0  =&  \frac{\Omega^{2}}{p}\geq0.\qquad & &
		\end{aligned}
	\end{equation}
	
According to $B_1 \ge 0$, the negative sign in \cref{eq:28} should be discarded. When $p = (3+\sqrt{3})/{6}$, the conditionally stable domain is determined as $\Omega \in \left(0,\;\sqrt{6-2\sqrt{3}}\right]$. Therefore, a self-starting single-solve fully explicit algorithm with third-order ($\xi=0$) is as 
	
	\begin{mdframed}[innertopmargin=8pt]
		\noindent\textbf{Algorithm 1}  \refstepcounter{alg}\label{1} \\
		\noindent\rule{\linewidth}{0.4pt}
		\begin{subequations}
			\noindent\textbf{Solve} $\bma_{n+p}$ by using:
			\begin{align}
				\bm{M} \bma_{n+p}& = \bm{f}(\bmv_{n+p},~\bmu_{n+p},~t_{n+p})\\
				\bmv_{n+p} & = \bmv_n + p\dt   \bma_n\\
				\bmu_{n+p} & = \bmu_n + p\dt \bmv_n + \frac{p^2}{2} \dt^2 \bma_n 
			\end{align}
			\noindent then, \textbf{update} solutions at $t_{n+1}$ with
			\begin{align}
				\bmu_{n+1} & = \bmu_n + \dt \bmv_n + \dt^2 \left( \frac{6p^2 - 1}{12p} \bma_n + \frac{-6p^2 + 6p + 1}{12p} \bma_{n+p} \right) \\
				\bmv_{n+1} & = \bmv_n + \dt \left( \frac{2p - 1}{2p} \bma_n + \frac{1}{2p} \bma_{n+p} \right)                                 \\
				\bma_{n+1} & = \frac{p - 1}{p} \bma_n + \frac{1}{p} \bma_{n+p}
			\end{align}
			where $p=\left(3+\sqrt{3}\right)/6$. 
		\end{subequations}
	\end{mdframed}

 \subsection{Explicit algorithms with the implicit treatment of velocity}
	
	The self-starting single-solve explicit algorithm with the implicit treatment of velocity is considered, which achieves identical second-order accuracy. By substituting the parameters shown in \cref{eq:7,eq:2accuracy3} into  \cref{eq:accuracy} for general dynamics problems ($\xi\neq0$), the local truncation error for displacement can be derived as
	\begin{equation}
		\begin{aligned}
			\tau_u =\Big[
			\left( 24p\alpha_4 - 12p + 2 \right) \xi \omega \dddot{u}(t_n)  &+( 12p\alpha_6 - 6p  + 12\alpha_1 - 1 )\omega^2 \ddot{u}(t_n)                   
			\\ &+ \left( 1 - 6p^2 + 6p - 12\alpha_6 p \right)\ddot{f}(t_n)
			\Big] \cdot \frac{\dt^2}{12p} 
			+ \mathcal{O}(\dt^3).
		\end{aligned}
	\end{equation}
	To achieve ${\tau}_u = \mathcal{O}(\dt^3)$, the following results are obtained:
	\begin{align}\label{eq:31}
		\alpha_6 = - \frac{6p^2 - 6p - 1}{12p}, \qquad
		\alpha_4 = \frac{6p - 1}{12p},  \qquad
		\alpha_1 = \frac{p^2}{2}.
	\end{align}
	Likewise, the truncation errors for velocity is accounted for
	\begin{align}\label{eq:33}
		\tau_{\dot{u}} = -\frac{(6p^2 - 6p + 1) \dddot{f}(t_n)}{12p} \dt^2 + \mathcal{O}(\dt^3).
	\end{align}
	
	It can be observed that the form of \cref{eq:33} is identical to that of \cref{eq:27}. In order to  $\tau_{\dot{u}} = \mathcal{O}(\dt^3)$, \cref{eq:28} can be similarly obtained. However, while \cref{eq:27} was derived under the condition $\xi = 0$ for a fully explicit algorithm, \cref{eq:33} was established without imposing this constraint. The difference highlights that algorithms employing implicit treatment of velocity yield more accurate solutions for damped problems.

		For the explicit algorithm with the implicit treatment of velocity,  the coefficients $ B_j \ (j = 0, \dots, 4)$ are calculated as 
	\begin{equation}
			\begin{alignedat}{3}\label{eq:40}
			&B_4  = \frac{24 \Omega^{3} p+1576 \left(2p-1\right) \xi^{2} \Omega +576 \xi }{\left[  \left(6p-1\right)\xi \Omega +6p\right]^{2}} \geq0,  \qquad\qquad\qquad\ \ 
			B_3  = \frac{4 4\left(1-3 p\right) \Omega^{2}-6 \Omega  \xi +48 p-24}{  \left(6 p-1\right)\xi \Omega +6 p}   \geq0,&                                          \\
			&B_2  =  \frac{-4 \Omega^{2}+24 \left(2 p-1\right) \xi  \Omega+24}{  \left(6 p-1\right)\xi \Omega +6 p}\geq0 , \quad
			B_1  = \frac{\left(12 p-6\right) \Omega^{2}+24\xi \Omega }{  \left(6 p-1\right)\xi \Omega +6 p} \geq0,\quad
			B_0  = \frac{6 \Omega^{2}}{  \left(6 p-1\right)\xi \Omega +6 p}\geq0.&
		\end{alignedat}
	\end{equation}
	where the parameter $p$ is defined in \cref{eq:28}. Similarly, choosing the negative sign in \cref{eq:28} makes \cref{eq:40} invalid. Therefore, once $p = (3+\sqrt{3})/{6}$ is determined, the conditionally stable domain. can be derived as follows:
	\begin{equation}\label{eq:399}
		\omega_\text{max} \dt \leq \Omega_s = \left( \sqrt{3+\sqrt{3}+\xi^{2}}-\xi\right) \left(\sqrt{3}-1\right).
	\end{equation}
	The term \( \Omega_s \) represents the maximum stability domain, while \( \omega_\text{max} \) denotes the highest natural frequency of the system under consideration.
	A self-starting single-solve third-order explicit algorithm with the implicit treatment of velocity is 
	\begin{mdframed}[innertopmargin=8pt]
		\noindent\textbf{Algorithm 2}  \refstepcounter{alg}\label{2} \\
		\noindent\rule{\linewidth}{0.4pt}
		\begin{subequations}
			\noindent\textbf{Solve} $\bma_{n+p}$ by using:
			\begin{align}
				\bm{M} \bma_{n+p} & = \bm{f}(\bmv_{n+p}, \bmu_{n+p}, t_{n+p})\\     
				\bmv_{n+p} & = \bmv_n + \dt \left( \frac{12p^2 - 6p + 1}{12p} \bma_n + {\frac{6p - 1}{12p} \bma_{n+p}} \right) \label{eq:43c}\\ 
				\bmu_{n+p} & = \bmu_n + p\dt\, \bmv_n + \frac{p^2}{2} \dt^2 \bma_n
			\end{align}
			\noindent then, \textbf{update} solutions at $t_{n+1}$ with
			\begin{align}
				\bmu_{n+1} & = \bmu_n + \dt \bmv_n + \dt^2 \left( \frac{6p^2 - 1}{12p} \bma_n + \frac{-6p^2 + 6p + 1}{12p} \bma_{n+p} \right) \\
				\bmv_{n+1} & = \bmv_n + \dt \left( \frac{2p - 1}{2p} \bma_n + \frac{1}{2p} \bma_{n+p} \right)                                 \\[4pt]
				\bma_{n+1} & = \frac{p - 1}{p} \bma_n + \frac{1}{p} \bma_{n+p}
			\end{align}
			\noindent where $p=\left(3+\sqrt{3}\right)/6$.
		\end{subequations}
	\end{mdframed}
	
	It is observed from \novelalgref{2} that when the force vector $\bm{f} = \bm{f}(\bmu, t)$ does not involve velocity, \cref{eq:43c} is not invoked, and the scheme becomes fully explicit. In this case, it is observed that \novelalgref{1} and \novelalgref{2} are completely identical. On the other hand, \novelalgref{1} achieves third-order accuracy only when solving undamped problems, while it remains identical second-order accurate in all other cases.
	 Therefore, \novelalgref{1} can be regarded as a special case of \novelalgref{2}. In summary, \novelalgref{2} degenerates into \novelalgref{1} when applied to velocity-independent equations (i.e., $\xi = 0$), whereas it consistently maintains third-order accuracy in general scenarios ($\xi\neq0$). Hence, for enhanced accuracy, \novelalgref{1} is not recommended for problems involving velocity.
	
	Building on the preceding analysis, this paper highlights that within the class of self-starting single-solve explicit time integration algorithms, identical third-order accuracy cannot be realized. There exists exactly one fully explicit algorithm that achieves third-order accuracy ($\xi = 0$ for both displacement and velocity), namely \novelalgref{1}.
	There exists a unique explicit algorithm with the implicit treatment of velocity that achieves third-order accuracy ($\xi \neq 0$ for both displacement and velocity), denoted by \novelalgref{2}.

	\novelalgref{2} still requires solving nonlinear systems of equations when dealing with nonlinear damping problems, and it is not explicit in this case. Code~\ref{code:1}  presents the pseudocode for \novelalgref{2} when solving the nonlinear damping problem in \cref{eq1}. In the Newton-Raphson iteration scheme, the tangent damping matrix is calculated as
	\begin{equation}
		\bm{C}^{(i)}=\left. -\frac{\partial\bm{f}(\bmv(t),\bmu(t), t)}{\partial \bmv(t)}\right|_{\bmv(t)=\bmv_{n+p}^{(i)},\ \bmu(t)=\bmu_{n+p}}.
	\end{equation}
	
	\begin{algorithm}[h]
		\caption{\novelalgref{2} for solving nonlinear damping problems.}\label{code:1}
		\textbf{Solve}: the initial acceleration $\bma_0$ using the  equilibrium equation at $t_0$.\\
		\textbf{Select}: the time step $\dt$, convergence tolerance $\epsilon$, and maximum number of iterations \texttt{max-iter}.\\
		
		\For{$n = 0$ \KwTo $N$}{
			\textbf{Compute}: $\bmu_{n+p} \leftarrow \bmu_n + p \dt \bmv_n + \dfrac{p^2}{2} \dt^2 \bma_n$.\\
			\textbf{Predict}: $\bma_{n+p}^{(1)} \leftarrow \bma_n$.\\
			
			\texttt{\textcolor{gray}{//========== nonlinear iterative procedure ==========}}
			
			\For{$i = 1$ \KwTo \normalfont{\texttt{max-iter}}}{
				\textbf{Compute}: $\bmv_{n+p}^{(i)} = \bmv_n +  \dt\left(\dfrac{12 p^2 -6 p+1}{12p}\bma_n + \dfrac{6p-1}{12p}\bma^{(i)}_{n+p}\right)$.\\
				
				\textbf{Compute}: $\bm{f}(\bmv^{(i)}_{n+p}, \bmu_{n+p},t_n + p\dt)$ and $\bm{\widetilde{M}} = \bm{M} + \dfrac{6p-1}{12p} \dt \bm{C}^{(i)}$.\\
				
				\textbf{Solve}: $\Delta \bma^{(i)}$ using $\bm{\widetilde{M}} \Delta \bma^{(i)} =\bm{f}(\bmv^{(i)}_{n+p}, \bmu_{n+p},t_n + p\dt) - \bm{M}\bma^{(i)}_{n+p}$.\\
				
				\If{$||\Delta \bma^{(i)}|| < \epsilon $}{
					\textbf{Compute}: $\bma_{n+p} \leftarrow \bma^{(i)}_{n+p}$.\\
					\textbf{Update}: $\bma_{n+1}  \leftarrow \dfrac{p-1}{p}\bma_n + \dfrac{1}{p}\bma_{n+p}$.\\
					\textbf{Update}: $\bmv_{n+1}  \leftarrow \bmv_n + \dt\left(\dfrac{2p-1}{2p}\bma_n + \dfrac{1}{2p}\bma_{n+p}\right)$.\\
					\textbf{Update}: $\bmu_{n+1}  \leftarrow \bmu_n + \dt \bmv_n + \dt^2\left(\dfrac{6p^2-1}{12p}\bma_n + \dfrac{-6p^2+6p+1}{12p}\bma_{n+p}\right)$.\\
					\texttt{Break}.
				}
				\If{$i =\normalfont{\texttt{max-iter} }$} {
					\texttt{Return to Step 2 and re-select the time step $\dt$.}\\
				}
				\textbf{Update}: $\bma^{(i+1)}_{n+p} \leftarrow \bma^{(i)}_{n+p} + \Delta \bma^{(i)}$.\\
			}
		}		
	\end{algorithm}

	\section{Numerical properties}\label{sec:properties}
	In this section, numerical characteristics of \novelalgrefs{1}{2} are analyzed, focusing on spectral analysis, amplitude and phase error analysis, as well as convergence rate analysis. Furthermore, the algorithms are compared to the published methods to demonstrate advantages.
	
	\subsection{Spectral analysis}
	The spectral radii $\rho$, numerical damping ratio $\bar{\xi}$ and the period error (PE) are often used to characterize  numerical dissipation and dispersion of an algorithm. For brevity, their formulas are not provided here (see the literature~\cite{hughesFiniteElementMethod2012}).
	The spectral properties of \novelalgref{1} are shown in \cref{fig:1}, where spectral radii are given in the subplot \ref{fig:1a}, indicating that \novelalgref{1} exhibits numerical dissipation ability to eliminate spurious high-frequency modes. The undamped case is the main focus of \novelalgref{1}. Based on the numerical damping ratio and the period errors shown in \cref{fig:1b,fig:1c}, it is evident that in the undamped case, it shows significant numerical dissipation and small period errors, demonstrating certain advantages. However, as the physical damping ratio increases, the stable domain rapidly decreases, and the period error also increases. It is also one of the reasons why \novelalgref{1} is not recommended for damped problems.
	
	The spectral properties of \novelalgref{2} are shown in \cref{fig:2}. When $\xi$ is small, the spectral radii exhibits an upward convex trend before the bifurcation point $\Omega_b$. It is important to note that \novelalgrefs{1}{2} does not have an algorithmic sub-member with a spectral radii $\rho \equiv 1$ for $\Omega \in (0, \Omega_s)$. It is well-known that for structural dynamics problems, spatial discretization techniques can introduce spurious high-frequency components. Therefore, \novelalgrefs{1}{2} always exhibits dissipation, as there always exists an $\Omega$ such that $\rho \leq 1$. This is not a disadvantage, as it can eliminate spurious high-frequency components in certain problems, acting as a filter. \cref{fig:2b} shows the variation of the numerical damping ratio $\bar{\xi}$ with $\Omega \in (0, \Omega_s)$ for \novelalgref{2}. When $\Omega$ is close to $0$, the numerical damping ratio matches the physical damping ratio. As $\Omega$ increases, the numerical damping ratio also increases. Comparing \cref{fig:1,fig:2}, \novelalgref{2} provides the better stability and smaller period errors when solving problems with any physical damping ratio, especially in high-frequency problems. It also demonstrates the superiority of the implicit treatment of velocity when solving damped problems. \novelalgrefs{1}{2} maintain identical when solving undamped problems. Therefore, \novelalgref{1} will be subsumed under \novelalgref{2} in the following discussions when $\xi=0$. 
	\begin{figure}[htbp]
		\centering
		\includegraphics[width=\textwidth]{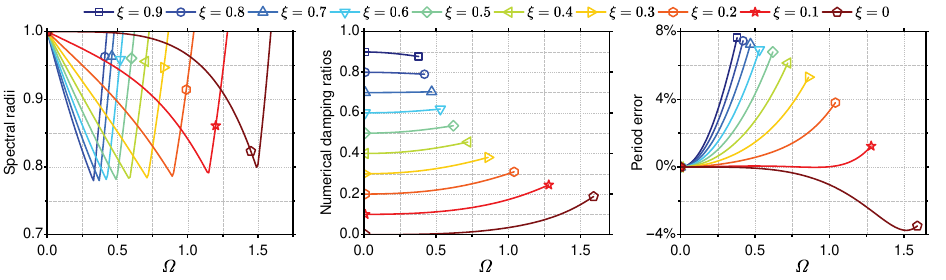}
		\begin{minipage}{0.33\textwidth}
			\centering
			\subcaption{Spectral radii} \label{fig:1a}  
		\end{minipage}%
		\begin{minipage}{0.33\textwidth}
			\centering
			\subcaption{Numerical damping ratio} \label{fig:1b}  
		\end{minipage}
		\begin{minipage}{0.33\textwidth}
			\centering
			\subcaption{Period error} \label{fig:1c}  
		\end{minipage}
		\caption{Spectral properties of \novelalgref{1} under varying $\xi$}
		\label{fig:1}
	\end{figure}
	
	 The spectral characteristics of different algorithms in the absence of physical damping are illustrated in \cref{fig:3}. Among the single-solve algorithms (GSSE~\cite{liIdenticalSecondOrder2021}, NE~\cite{newmarkMethodComputationStructural1959}), \novelalgrefs{1}{2} clearly exhibit superior performance. When $\Omega \in (0, 1]$, \novelalgrefs{1}{2} demonstrate excellent numerical spectral characteristics, including high numerical dissipation, while significantly smaller period errors~(this means the \novelalgrefs{1}{2} are more effective at eliminating high-frequency components introduced by spatial discretization, without excessively dissipating the useful components, benefits that will be confirmed in the subsequent test cases). Among the two‐sub‐step algorithms~(Noh-Bathe~\cite{nohExplicitTimeIntegration2013}, SS2HE\textsubscript{1}~\cite{liTwoThirdorderExplicit2022} and the algorithm from Wang et al.~\cite{wang_GeneralizedSinglestepMultistage_2025}), \novelalgrefs{1}{2} also exhibit strong spectral properties; for instance, when $\Omega < 0.5$, their period error is competitive with those of the two‐sub‐step methods, yet \novelalgrefs{1}{2} remain single‐solve. It must be noted that \novelalgrefs{1}{2} are compared against existing two-sub-step algorithms, which incur double the computational workload, purely to present the new methods’ features objectively. There is no attempt to suggest that \novelalgrefs{1}{2} can fully surpass those two-sub-step schemes—doing so would be unreasonable and impossible, as they consume significantly more computation.
	 These properties ensure that \novelalgrefs{1}{2} holds a distinguished position even among these two-sub-step algorithms
	\begin{figure}[h]
		\centering
		\includegraphics[width=\textwidth]{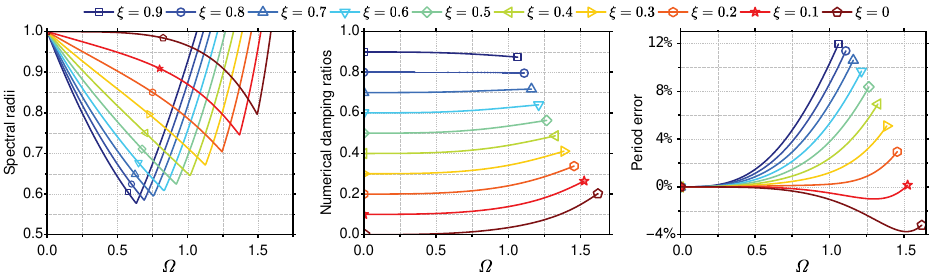}
		\begin{minipage}{0.33\textwidth}
			\centering
			\subcaption{Spectral radii} \label{fig:2a}  
		\end{minipage}%
		\begin{minipage}{0.33\textwidth}
			\centering
			\subcaption{Numerical damping ratio} \label{fig:2b}  
		\end{minipage}
		\begin{minipage}{0.33\textwidth}
			\centering
			\subcaption{Period errors} \label{fig:2c}  
		\end{minipage}
		\caption{Spectral properties of \novelalgref{2} under varying $\xi$}
		\label{fig:2}
	\end{figure}
	
	\cref{fig:4} illustrates the spectral characteristics of these algorithms when the physical damping ratio is set to $\xi = 0.1$. Compared to the single-solve algorithms, \novelalgref{2} exhibits clearer advantages: its spectral radius remains smooth both before and after the bifurcation point, and its period errors are much smaller than those of NE~\cite{newmarkMethodComputationStructural1959} and GSSE~\cite{liIdenticalSecondOrder2021}, highlighting its superiority in the presence of physical damping. The Noh-Bathe method~\cite{nohExplicitTimeIntegration2013}, although it maintains good stability regardless of the presence of damping, exhibits amplified period errors in the presence of $\xi$. This occurs because it achieves second-order accuracy rather than third-order. The results show that the third-order SS2HE\textsubscript{1}~\cite{liTwoThirdorderExplicit2022} and the schemes proposed by Wang et al.~\cite{wang_GeneralizedSinglestepMultistage_2025} experience a sharp decline in stability, whereas \novelalgref{2}, which incorporates implicit velocity treatment, maintains excellent stability. Additionally, \novelalgref{2} achieves minimal period errors while retaining significant numerical dissipation. More importantly, it provides exceptional accuracy in period errors using only one single solution.
	\begin{figure}[t]
		\centering
		\includegraphics[width=\textwidth]{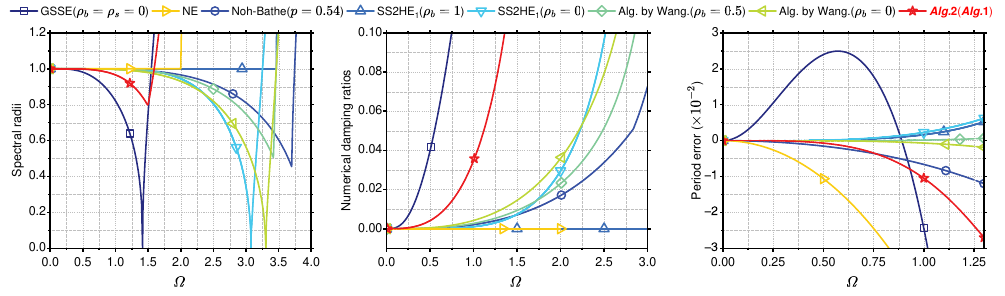}
		\begin{minipage}{0.33\textwidth}
			\centering
			\subcaption{Spectral radii} \label{fig:3a}  
		\end{minipage}%
		\begin{minipage}{0.33\textwidth}
			\centering
			\subcaption{Numerical damping ratio} \label{fig:3b}  
		\end{minipage}
		\begin{minipage}{0.33\textwidth}
			\centering
			\subcaption{Period errors} \label{fig:3c}  
		\end{minipage}
		\caption{Comparison of spectral properties under $\xi = 0$ for various explicit methods: GSSE~\cite{liIdenticalSecondOrder2021}, NE~\cite{newmarkMethodComputationStructural1959}, Noh-Bathe~\cite{nohExplicitTimeIntegration2013}, SS2H\textsubscript{1}~\cite{liTwoThirdorderExplicit2022} and algorithm 1 by Wang et al.~\cite{wang_GeneralizedSinglestepMultistage_2025}.}
		\label{fig:3}
	\end{figure}
	\begin{figure}[h]
		\centering
		\includegraphics[width=\textwidth]{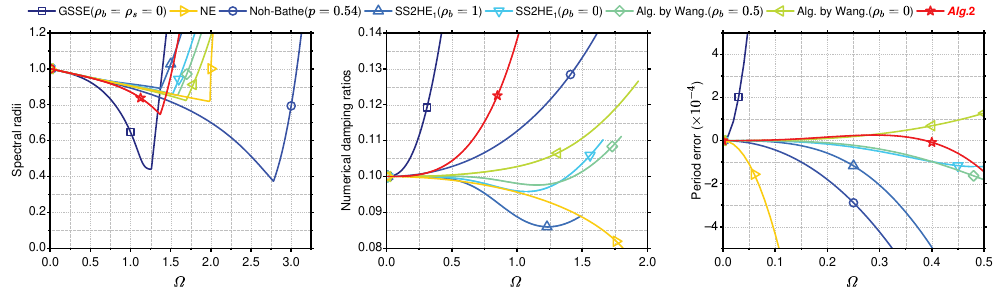}
		\begin{minipage}{0.33\textwidth}
			\centering
			\subcaption{Spectral radii} \label{fig:4a}  
		\end{minipage}%
		\begin{minipage}{0.33\textwidth}
			\centering
			\subcaption{Numerical damping ratio} \label{fig:4b}  
		\end{minipage}
		\begin{minipage}{0.33\textwidth}
			\centering
			\subcaption{Period errors} \label{fig:4c}  
		\end{minipage}
		\caption{Comparison of spectral properties under $\xi = 0.1$ for various explicit methods: GSSE~\cite{liIdenticalSecondOrder2021}, NE~\cite{newmarkMethodComputationStructural1959}, Noh-Bathe~\cite{nohExplicitTimeIntegration2013}, SS2H\textsubscript{1}~\cite{liTwoThirdorderExplicit2022} and algorithm 1 by Wang et al.~\cite{wang_GeneralizedSinglestepMultistage_2025}.}
		\label{fig:4}
	\end{figure}
	
	In order to show a more intuitive comparison of the stability domains of \novelalgrefs{1}{2} , \cref{fig:stab} illustrates the variation of $\Omega_s$ with the physical damping ratio $\xi$ for the proposed algorithms, alongside the published algorithms. \cref{fig:5a} shows the stable domains of different single-solve fully explicit algorithms, including TW~\cite{maheoNumericalDampingSpurious2013} and GSSE~\cite{liIdenticalSecondOrder2021}, along with the \novelalgrefs{1}{2}. From the figure, it is evident that the stability domains of single-solve fully explicit algorithms cannot exceed $2$. GSSE~\cite{liIdenticalSecondOrder2021} and TW~\cite{maheoNumericalDampingSpurious2013} have a stability domain of $\Omega_s = 2$ when $\rhob = 1$, but as damping increases, their stability domains decrease over-linearly. When the physical damping is large, \novelalgref{2} demonstrates the best stable domain. The stable domains of different single-solve explicit algorithms with implicit velocity treatment (GSSI~\cite{zhaoSelfstartingDissipativeAlternative2023} and NE~\cite{newmarkMethodComputationStructural1959}) and the \novelalgrefs{1}{2} are shown in \cref{fig:5b}. Among them, GSSI~\cite{zhaoSelfstartingDissipativeAlternative2023} has the best stability within the category of single-solve explicit algorithms. Under certain conditions, its stability domain exceeds the limit of 2 imposed by the CD method and it achieves identical second-order accuracy. When $\rhob = 1$ and $\rho_s = 0$, it degenerates into NE ~\cite{newmarkMethodComputationStructural1959}. While the stability of \novelalgref{2} decreases as $\xi$ increases, its primary advantage is the enhanced calculation accuracy. \cref{fig:5c} shows the stability of different two-sub-step explicit algorithms. These two-sub-step algorithms have stable domains greater than 2 in the undamped or low damping cases, but they are highly sensitive to damping. As the damping ratio $\xi$ increases, the third-order accurate algorithms (the algorithm proposed  by Wang et al.\cite{wang_GeneralizedSinglestepMultistage_2025} and SS2HE\textsubscript{1}\cite{liTwoThirdorderExplicit2022}) experience a rapid decrease in stability. When the physical damping ratio $\xi > 0.15$, their stability is worse than the third-order accurate \novelalgref{2} proposed in this paper. The stability domain of the two-sub-step Noh-Bathe method~\cite{nohExplicitTimeIntegration2013} also decreases as damping increases, and at a certain point ($\xi\approx0.75$), it remains smaller than \novelalgref{2}. It remains important to emphasize that the new methods are evaluated against the two-sub-step algorithms using half the computational workload; this is admittedly unfair but required for objectivity. In conclusion, the implicit velocity treatment technique significantly enhances stability, and \novelalgref{2} benefits from this, providing good stability within the realm of single-solve explicit algorithms, even in multi-sub-step algorithms.
	\begin{figure}[h]
		\centering
		\includegraphics[width=\textwidth]{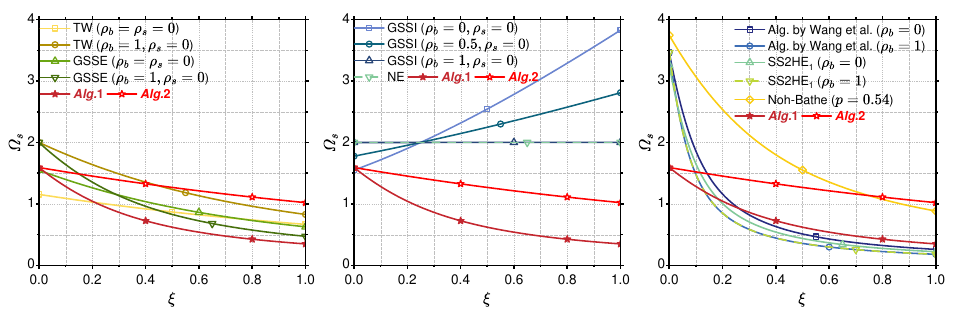}
		\begin{minipage}{0.33\textwidth}
			\centering
			\subcaption{Single-solve fully explicit algorithms} \label{fig:5a}  
		\end{minipage}%
		\begin{minipage}{0.33\textwidth}
			\centering
			\subcaption{\centering Single-solve explicit algorithms with the implicit treatment of velocity} \label{fig:5b}
		\end{minipage}
		\begin{minipage}{0.33\textwidth}
			\centering
			\subcaption{\centering Two sub-steps explicit algorithms} \label{fig:5c}  
		\end{minipage}
		\caption{Comparison of the conditional stability domains of \novelalgrefs{1}{2}, as well as various published explicit algorithms, under different parameters.} \label{fig:stab}
	\end{figure}
	
	\subsection{Amplitude and phase errors}
	Currently, almost all literature measures amplitude and phase errors of time integration methods by calculating numerical damping ratios and period errors numerically. Numerical techniques struggle to provide substantial guidance for optimizing algorithm performance. Gladwell and Thomas~\cite{gladwellDampingPhaseAnalysis1983}  provided an analysis technique for amplitude and phase errors but only derived the case for zero-dissipation algorithms without refining the amplitude and phase error analysis, not calculating the analytical expressions for amplitude and phase errors. Li et al.~\cite{liDesigningDevelopingSingle2023} refined a general method for analytically computing the amplitude and phase errors for any time integration methods. This technique for analyzing amplitude and phase errors allows for a quantitative assessment of the errors in algorithms with the same accuracy order, providing a novel metric for evaluating the quality of an algorithm. Additionally, it offers valuable guidance for algorithm design.
	
		\begin{figure}[htbp]
		\centering
		\includegraphics[width=\textwidth]{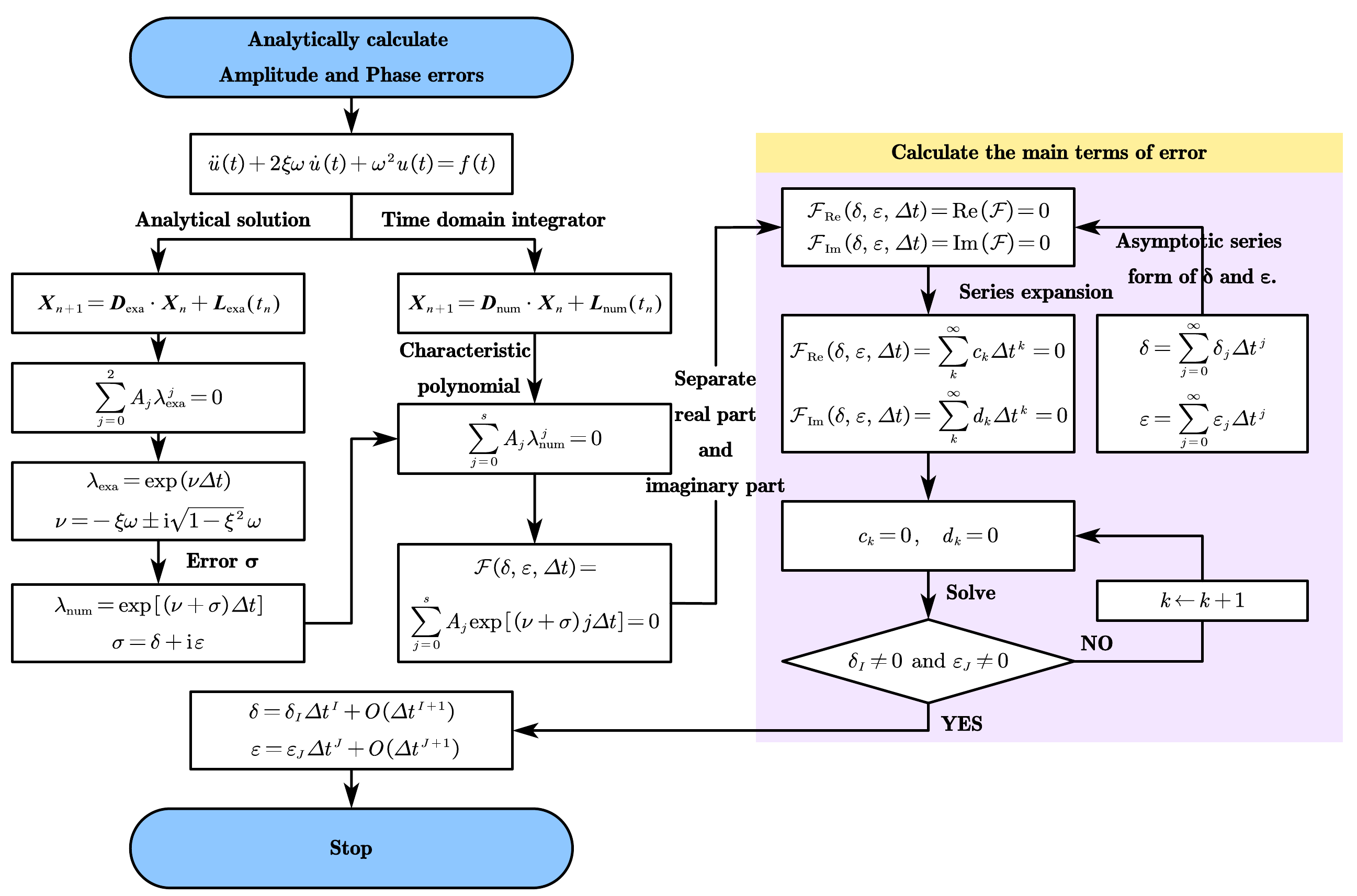}
		\caption{The analytical technique for computing amplitude and phase errors for any time integration schemes}
		\label{fig:6}
	\end{figure}
	\cref{fig:6} shows the general analytical process for calculating the leading terms of amplitude and phase errors for any time integration methods~\cite{liDesigningDevelopingSingle2023}. The amplitude and phase errors for \novelalgref{1} are calculated as
	\begin{subequations}\label{eq:42}
		\begin{empheq}[left=\empheqlbrace]{align}
			\varepsilon &= \frac{\left(2+\sqrt{3}\right) \left(4 \xi^{2}-3\right) \omega^{3}\xi^{2}}{12 \sqrt{1-\xi^{2}}} \dt^2 + \mathcal{O}(\dt^3) \\
			\delta &= \frac{\left(2+\sqrt{3}\right) \left(4 \xi^{2}-1\right)\omega^{3}\xi}{12} \dt^2 + \mathcal{O}(\dt^3)
		\end{empheq}
	\end{subequations}
	The leading-order error terms of other self-starting single-solve explicit algorithms can be found in the  literature~\cite{LXJZ20250311001}. From \cref{eq:42}, it can be observed that the dominant error term of \novelalgref{1} is proportional to the square of the time step $\Delta t$, which is identical with its second-order accuracy in the presence of damping. 
	
	The amplitude and phase errors of \novelalgref{2} are calculated as
	\begin{subequations}\label{eq:43}
		\begin{empheq}[left=\empheqlbrace]{align}
			\varepsilon &= \frac{  \left[ 16  \sqrt{3} \xi^{4}+12\left(1-\sqrt{3}\right) \xi^{2}-9-\sqrt{3} \right] \omega^{4} \xi}{144 \sqrt{1 - \xi^{2}}} \Delta t^3 + \mathcal{O}(\Delta t^4) \\
			\delta &= \frac{ \left[ 16  \sqrt{3}\xi^{4}+4\left(3- \sqrt{3}\right) \xi^{2}-3-\sqrt{3} \right]\omega^{4}}{144} \Delta t^3 + \mathcal{O}(\Delta t^4)
		\end{empheq}
	\end{subequations}
	It is clear that the leading error terms of \novelalgref{2} are proportional to the cube of the time step $\Delta t$, whereas those of other single-solve explicit algorithms are typically proportional to $\Delta t^2$, highlighting the superior accuracy of \novelalgref{2}. When $\xi = 0$, \cref{eq:42,eq:43} reduce to
	\begin{subequations}
		\begin{empheq}[left=\empheqlbrace]{align}
			\varepsilon &= \frac{\left(8+5 \sqrt{3}\right) \omega^{5}}{1440}\Delta t^4 + \mathcal{O}(\Delta t^5) \\
			\delta &= -\frac{ \left(3+\sqrt{3}\right)\omega^{4}}{144} \Delta t^3 + \mathcal{O}(\Delta t^4)
		\end{empheq}
	\end{subequations}
	
	In this case, the phase error $\varepsilon$ reaches an order of $\mathcal{O}(\Delta t^4)$, which leads to the performance advantages illustrated in \cref{fig:s1}. The tested system adopts the parameters $\xi = 0$, $\omega = 2$ and $f(t) = 0$ in \cref{eq:single} with initial conditions $u(0) = 1$ and $\dot{u}(0) = 0$. For this undamped system with the exact solution being $u(t) = \cos(2 t)$, \novelalgrefs{1}{2} exhibit perfectly aligned solutions due to their higher-order phase accuracy, while all other single-solve algorithms using zero-dissipation ($\rho_b = 1$) configurations show noticeable phase drift.
	\begin{figure}[h]
	\centering
	\begin{subfigure}[h]{0.49\textwidth}
		\centering
		\includegraphics[width=\textwidth]{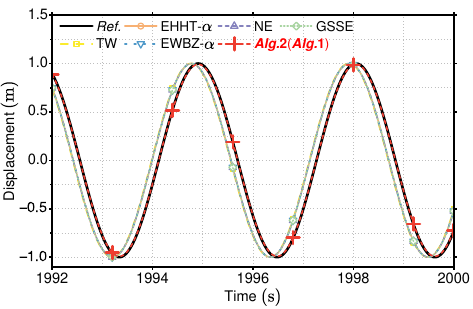}
		\captionsetup{justification=centering}
		\caption{Numerical displacements of the SDOF system \eqref{eq:single} at $\xi=0$ and $\omega=2$}
		\label{fig:s1}
	\end{subfigure}
	\begin{subfigure}[h]{0.49\textwidth}
		\centering
		\includegraphics[width=\textwidth]{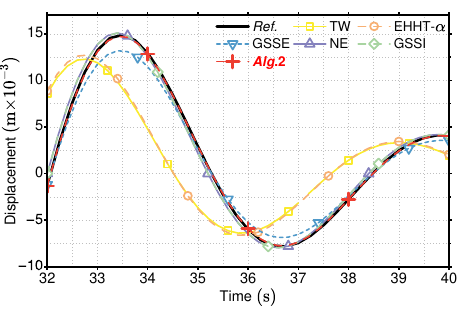}
		\captionsetup{justification=centering}
		\caption{Numerical displacements of the SDOF system \eqref{eq:single} at $\xi=0.2$ and $\omega=2$}
		\label{fig:s2}
	\end{subfigure}
	\caption{Numerical displacements for SDOF without external forcing ($f(t)=0$) using different algorithms}
\end{figure}
	
	The physical parameters of the SDOF system \eqref{eq:single} are set to $\xi = 0.2$, $\omega = 2$ and $f(t) = 0$. The initial conditions are set to $u(0) = 0$ and $\dot{u}(0) = 12$, and the exact solution of this system is $u(t) = 5\sqrt{6} \cdot \exp(-t/5) \cdot \sin \left[ 2t/(5\sqrt{6}) \right]$. Applying zero dissipation ($\rho_b = 1$, note: GSSE~\cite{liIdenticalSecondOrder2021} and GSSI~\cite{zhaoSelfstartingDissipativeAlternative2023} algorithms cannot set $\rho_b = \rho_s = 1$ when solving problems with physical damping ($\xi\neq0$), therefore non-dissipative algorithms of the two are set to $\rho_b = 1, \rho_s = 0$), different single-solve algorithms are used to calculate the numerical solution for this system as shown in \cref{fig:s2}. It is evident that \novelalgref{2} provides the most accurate solutions. Following that, GSSI~\cite{zhaoSelfstartingDissipativeAlternative2023} and NE~\cite{newmarkMethodComputationStructural1959} also yield relatively accurate results. This can be attributed to the fact that all these algorithms incorporate implicit treatment of velocity.

	\subsection{Convergence rate tests}
	The standard SDOF system \eqref{eq:single} is used to analyze the system with initial conditions given as $\dot{u}(t_0) = \dot{u}_0$ and $u(t_0) = u_0$. A certain test duration $T$ and integration step size $\dt$ are selected and the global errors of the solution variables is used to represent the convergence rate.
	\begin{equation}
		\label{eq:51}
		\text{Global Error} = \left[ \frac{\sum_{i=1}^{N} \left( x(t_i) - x_i \right)^2}{\sum_{i=1}^{N} x^2(t_i)} \right]^{1/2}
	\end{equation}
	where $x(t_i)$ and $x_i$ represent the exact and numerical solutions of the solution variable at time $t_i$, and $N$ is the total number of time steps.

	\subsubsection{Undamped Vibration System}
	The physical parameters of the SDOF system \eqref{eq:single} are set to $\xi = 0$, $\omega = 1$ and $f(t) = \cos(2t)$. The initial conditions are given as $u(0) = -1/3$ and $\dot{u}(0) = 0$ with the exact solution  being $u(t) = -\cos(2t)/3$. For all tested algorithms, the non-dissipative member ($\rho_b = 1$) is selected. The Noh-Bathe method~\cite{nohExplicitTimeIntegration2013} is configured with the originally recommended parameter $p = 0.54$. The total simulation time is set to $T = 6.5$ seconds, and the resulting convergence rate curves are shown in \cref{fig:7}.
	The results reveal that \novelalgrefs{1}{2} achieve third-order accuracy in both displacement and velocity predictions under undamped conditions, which aligns with the theoretical analysis presented earlier. In contrast, other single-solve algorithms attain at most second-order accuracy when solving undamped problems. The algorithm proposed by Wang et al.~\cite{wang_GeneralizedSinglestepMultistage_2025} achieves third-order accuracy in displacement and velocity, matching the performance of \novelalgrefs{1}{2}. Similarly, SS2HE\textsubscript{1}~\cite{liTwoThirdorderExplicit2022} also achieves identical third-order accuracy. However, these two-sub-step algorithms require twice the computational cost compared to \novelalgrefs{1}{2}.
	\begin{figure}[htbp]
		\centering
		\includegraphics[width=\textwidth]{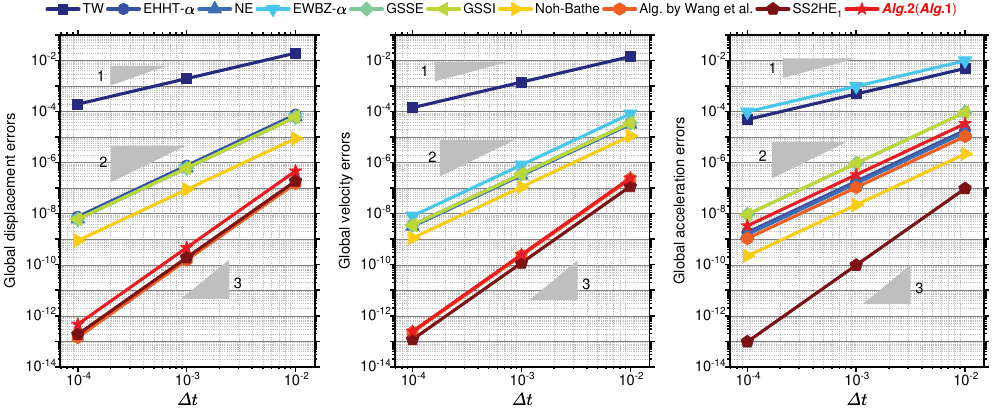}
		\caption{Convergence rates of various algorithms for solving forced vibration with $\xi = 0$, $\omega = 1$ and $f(t) = \cos(2t)$ in \cref{eq:single}}
		\label{fig:7}
	\end{figure}
	
	\subsubsection{Damped vibration system}
	For the forced vibration system with damping, the physical parameters of the SDOF system are set to $\xi = 2/\sqrt{5}$, $\omega = \sqrt{5}$ and $f(t) = \sin(2t)$. The initial conditions are specified as $u(0) = 57/65$ and $\dot{u}(0) = 2/65$ with the exact solution given by $u(t) = \exp(-2t)\left( \cos(t) + 2 \sin(t) \right) - \left( 8 \cos(2t) - \sin(2t) \right)/65$. Similarly, a test duration of $T = 5.6$ seconds is used and the dissipative members ($\rho_b = 0.8$) are selected to solve this damped forced vibration system. The convergence rates are shown in \cref{fig:8}.
	In this test, in terms of displacement accuracy, \novelalgref{2} yields the smallest error, outperforming even two-sub-step algorithms such as the first scheme by Wang et al.~\cite{wang_GeneralizedSinglestepMultistage_2025} and SS2HE\textsubscript{1}~\cite{liTwoThirdorderExplicit2022}, which require significantly more computational effort. \novelalgref{2} also clearly surpasses all other single-solve algorithms included in the test. Regarding velocity accuracy, \novelalgref{2} matches the performance of the third-order accurate two-sub-step algorithms, demonstrating a substantial advantage.

	Based on the aforementioned scenarios, the accuracy performance of nearly all published self-starting single-solve explicit algorithms is consolidated in \cref{tab:1}, which also includes several representative two-sub-step explicit methods. To date, no self-starting single-solve explicit algorithm (including those incorporating implicit velocity treatment) has achieved third-order accuracy. Moreover, those that have achieved third-order accuracy all require multiple solutions of the equilibrium equations.
	
	The fully explicit \novelalgref{1} developed in this study achieves third-order accuracy in displacement and velocity and second-order accuracy in acceleration under undamped conditions, while maintaining identical second-order accuracy in other cases. The explicit \novelalgref{2} attains third-order accuracy in displacement and velocity and second-order accuracy in acceleration for solving general structural dynamics problems ($\xi\neq 0$). Most importantly, both algorithms require only a single solution per time step. The proposed methods thus address a critical gap in the current literature on high-order self-starting single-solve explicit algorithms.
		\begin{figure}[htbp]
		\centering
		\includegraphics[width=\textwidth]{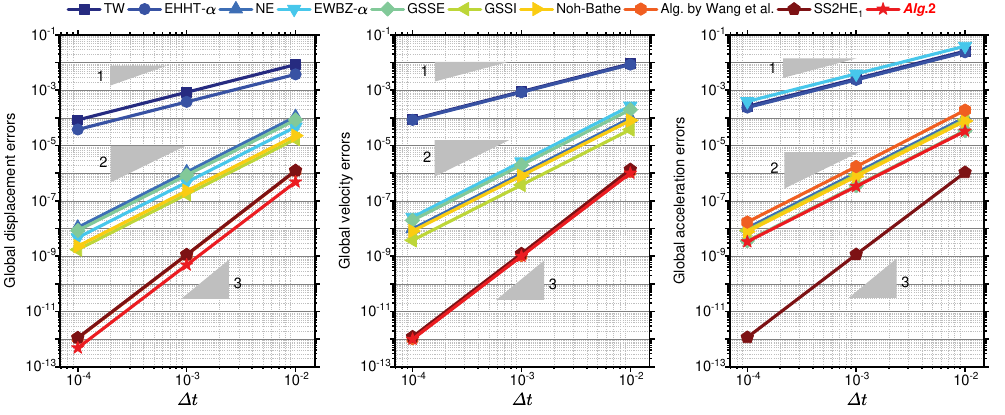}
		\caption{Convergence rates of various algorithms for solving damping forced vibration with $\xi = 2/\sqrt{5}$, $\omega = \sqrt{5}$ and $f(t) = \sin(2t)$ in \cref{eq:single}}
		\label{fig:8}
	\end{figure}

	\begin{table}[t] 
		\caption{The accuracy order of various self-starting explicit algorithms}
		\centering
		\renewcommand{\arraystretch}{1.5} 
			\begin{threeparttable}
				\footnotesize
				\begin{tabular}{c c c >{\centering\arraybackslash}p{1cm} >{\centering\arraybackslash}p{1cm} >{\centering\arraybackslash}p{1cm} >{\centering\arraybackslash}p{1cm} >{\centering\arraybackslash}p{1cm} >{\centering\arraybackslash}p{1cm}}
					\toprule
					\multicolumn{1}{c}{\multirow{2}{*}{Algorithms}} & 
					\multicolumn{1}{c}{\multirow{2}{1.2cm}{\centering Num. of solutions}} & 
					\multicolumn{1}{c}{\multirow{2}{1.8cm}{\centering Imp. treat. of velocity}} & 
					\multicolumn{3}{c}{\centering Accuracy ($f(t) \neq 0,\ \xi = 0$)} & 
					\multicolumn{3}{c}{\centering Accuracy ($f(t) \neq 0,\ \xi \neq 0$)} \\ \cline{4-6} \cline{7-9}
					& & &\multicolumn{1}{c}{\centering Dis.} &Vel. &Acc. &Dis. &Vel. &Acc. \\
					\midrule
					TW~\cite{maheoNumericalDampingSpurious2013} & 1 & No & 1 & 1 & 1 & 1 & 1 & 1 \\
					EN-$\beta$~\cite{hughesImplicitexplicitFiniteElements1978} & 1 & No & 1 & 1 & 1 & 1 & 1 & 1 \\
					EDV1~\cite{zhangTwoNovelExplicit2019} & 1 & No & 1 & 1 & --- & 1 & 1 & --- \\
					TSSE~\cite{liDevelopmentCompositeSubstep2021} & 1 & No & 1 & 1 & 1 & 1 & 1 & 1 \\
					EHHT-$\alpha$~\cite{mirandaImprovedImplicitExplicit1989} & 1 & No & 2 & 2 & 1 & 1 & 1 & 1 \\
					CL~\cite{chungNewFamilyExplicit1994} & 1 & No & 2 & 2 & 1 & 2 & 2 & 1 \\
					NT~\cite{namburuGeneralizedGammaFamily1992} & 1 & Yes & 2 & 2 & 1 & 2 & 2 & 1 \\
					EWBZ-$\alpha$~\cite{hulbertExplicitTimeIntegration1996} & 1 & No & 2 & 2 & 1 & 2 & 2 & 1 \\
					EG-$\alpha$~\cite{liIdenticalSecondOrder2021} & 1 & No & 2 & 2 & 1 & 2 & 2 & 1 \\
					ICL~\cite{kimSimpleExplicitSingle2019} & 1 & No & 2 & 2 & 2 & 2 & 2 & 2 \\
					NE~\cite{newmarkMethodComputationStructural1959} & 1 & Yes & 2 & 2 & 2 & 2 & 2 & 2 \\
					GSSE~\cite{liIdenticalSecondOrder2021} & 1 & No & 2 & 2 & 2 & 2 & 2 & 2 \\
					GSSI~\cite{zhaoSelfstartingDissipativeAlternative2023} & 1 & Yes & 2 & 2 & 2 & 2 & 2 & 2 \\
					Noh-Bath~\cite{nohExplicitTimeIntegration2013} & 2 & No & 2 & 2 & 2 & 2 & 2 & 2 \\
					Algs. 1,2 ~\cite{wang_GeneralizedSinglestepMultistage_2025} & 2 & No & 3 & 3 & 2 & 3 & 3 & 2 \\
					SS2HE\textsubscript{1,2}~\cite{liTwoThirdorderExplicit2022} & 2 & No & 3 & 3 & 3 & 3 & 3 & 3 \\
					\rowcolor{gray!20}\textcolor{red}{\novelalgref{1}} & \textcolor{red}{1} & \textcolor{red}{No} & \textcolor{red}{3} & \textcolor{red}{3} & \textcolor{red}{2} & \textcolor{red}{2} & \textcolor{red}{2} & \textcolor{red}{2} \\
					\rowcolor{gray!20}\textcolor{red}{\novelalgref{2}} & \textcolor{red}{1} & \textcolor{red}{Yes} & \textcolor{red}{3} & \textcolor{red}{3} & \textcolor{red}{2} & \textcolor{red}{3} & \textcolor{red}{3} & \textcolor{red}{2} \\
					\bottomrule
				\end{tabular}
				\begin{tablenotes}
					\small
					\item Note: '---' indicates that the value is not applicable.
				\end{tablenotes}
			\end{threeparttable}
		\label{tab:1}
	\end{table}

	\section{Numerical Examples}
	\label{sec:examples}
	This section will verify and compare the significant advantages of \novelalgrefs{1}{2} through various examples and algorithms. The examples include transient problems and high-frequency dissipation in impacted straight bars, as well as two-dimensional scalar wave propagation problems. Furthermore, nonlinear problems, where explicit algorithms demonstrate substantial advantages, will be tested. These include the Van der Pol system, the spring-pendulum system. Finally, a complex gap nonlinear rudder system, exhibiting strong nonlinear characteristics, will be tested. This system is used for practical simulations in engineering structures.
	
	\subsection{The Van der Pol system}
	The classic Van der Pol system is commonly used to evaluate the performance of algorithms on nonlinear problems. The governing equation is given by
	\begin{equation}\label{eq:van}
		{\ddot{x}} - \mu (1 - {x}^2) {\dot{x}} + {x} = A\sin(\omega_p\, t)
	\end{equation}
	where $x$ represents the state variable, $\mu$ is the control parameter, $A$ is the amplitude of the external excitation and $\omega_p$ is the frequency of the external excitation. The parameters are set to $\mu = 5$, $A = 5$, and $\omega_p = 2.5$, which characterize a strongly nonlinear system. To simulate this system, algorithms achieving identical second-order accuracy, including ICL~\cite{kimSimpleExplicitSingle2019}, GSSE~\cite{liIdenticalSecondOrder2021} and GSSI~\cite{zhaoSelfstartingDissipativeAlternative2023}, as well as the two-sub-step methods Noh-Bathe~\cite{nohExplicitTimeIntegration2013} and SS2HE\textsubscript{1}~\cite{liTwoThirdorderExplicit2022}, are employed with a time step size of $\Delta t = 0.005$ seconds.
	
	\begin{figure}[htbp]
		\centering
		\includegraphics[width=\textwidth]{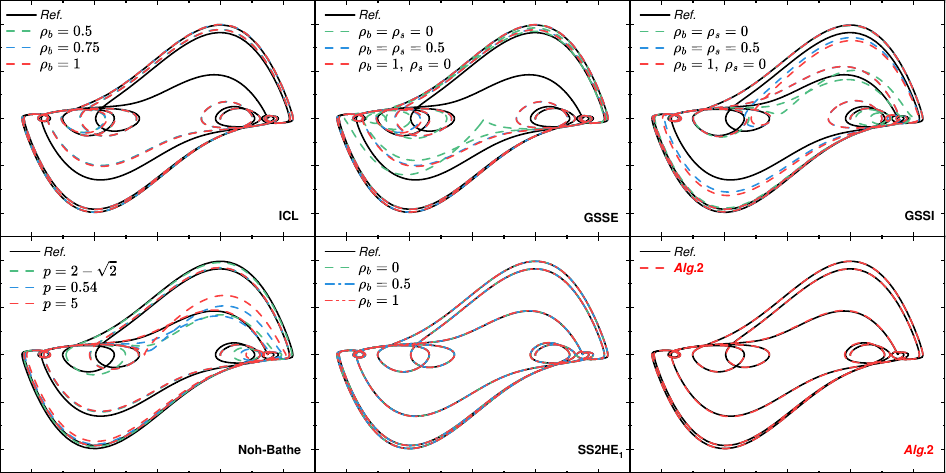}
		\caption{Phase portrait of the Van der Pol system \eqref{eq:van} with $\mu = 5$, $A = 5$, and $\omega_p = 2.5$ computed using various explicit algorithms}
		\label{fig:9}
	\end{figure}
	The phase portrait of the system provides insight into its dynamic behavior. \cref{fig:9} illustrates the phase portraits computed by different algorithms using the specified parameters, the reference solution obtained using the fourth-order explicit Runge-Kutta method, employing a time step size of $\Delta t = 10^{-8}$ seconds over a total simulation time of 30 seconds. Several conclusions can be drawn from the results: The fully explicit single-solve algorithms, GSSE~\cite{liIdenticalSecondOrder2021} and ICL~\cite{kimSimpleExplicitSingle2019}, produce significantly inaccurate trajectories, with the central portion of the phase portrait entirely missing. The algorithm GSSI~\cite{zhaoSelfstartingDissipativeAlternative2023} which incorporates implicit velocity treatment and the two-sub-step Noh-Bathe method~\cite{nohExplicitTimeIntegration2013} also exhibit notable deviations. In contrast, the third-order accurate two-sub-step algorithm SS2HE\textsubscript{1}~\cite{liTwoThirdorderExplicit2022} and the proposed single-solve \novelalgref{2} produce phase portrait curves that nearly perfectly overlap with the reference solution, indicating highly accurate predictions for both displacement and velocity.

	\subsection{Spring pendulum}
	As shown in \cref{fig:10a}, the spring pendulum represents a typical nonlinear system. When oscillating at small angles, its behavior can be approximated as linear; however, as the swing angle increases, the system exhibits pronounced nonlinear characteristics. Due to this property, the spring pendulum is commonly used to evaluate the performance of numerical algorithms in handling nonlinear dynamics.
	\begin{figure}[h]
		\begin{subfigure}[b]{0.49\textwidth}
			\centering
			\includegraphics[width=0.4\linewidth]{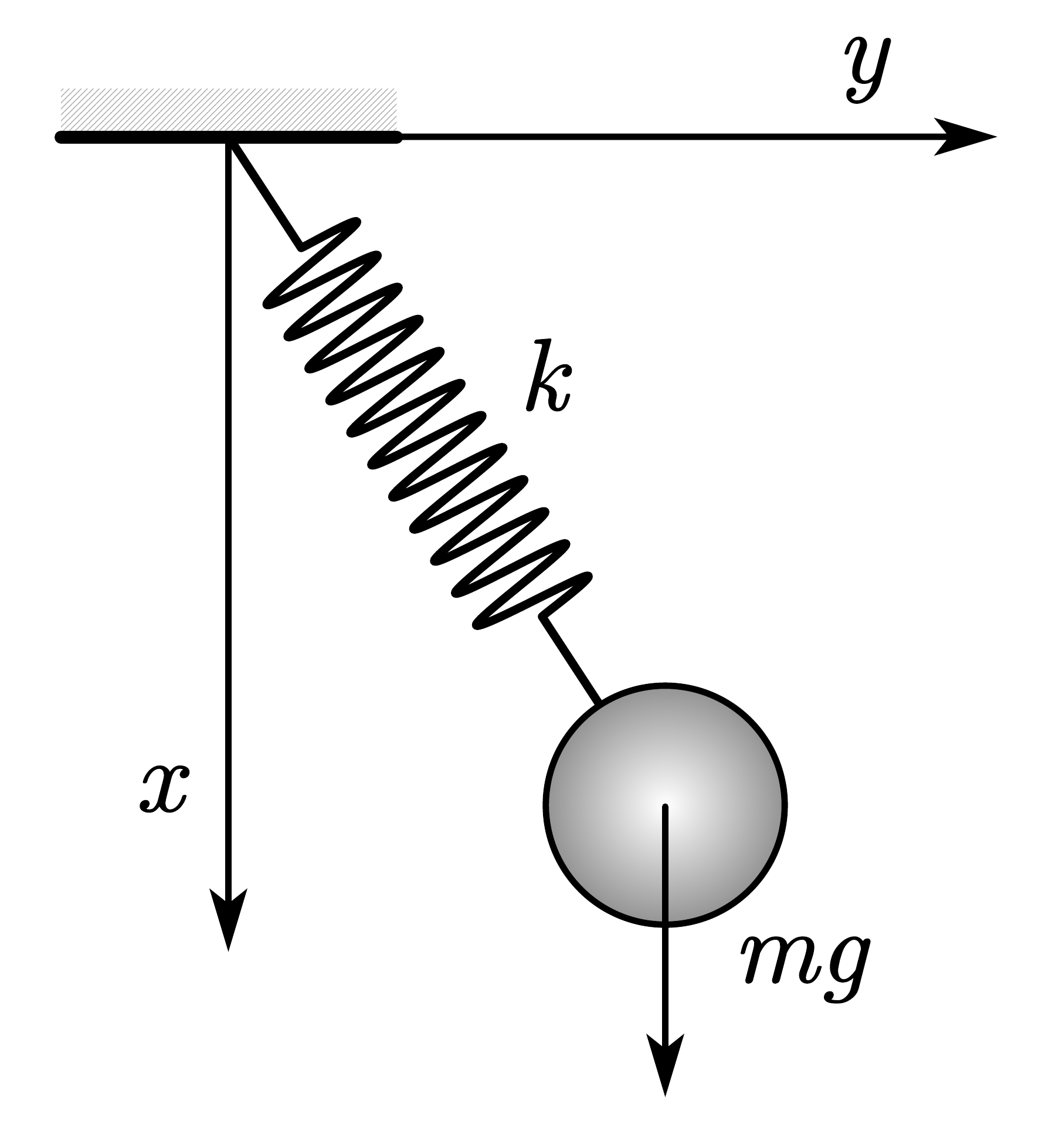}
			\caption{Planar spring pendulum}
			\label{fig:10a}
		\end{subfigure}
		\begin{subfigure}[b]{0.49\textwidth}
			\centering
			\includegraphics[width=0.4\linewidth]{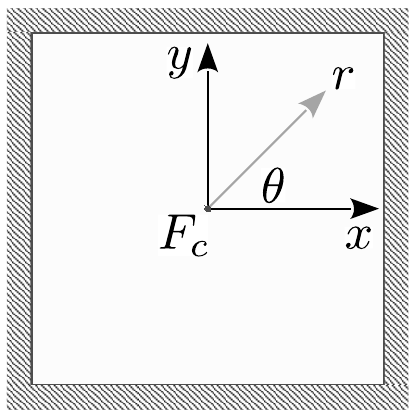}
			\caption{Two-dimensional membrane}
			\label{fig:wave2d}
		\end{subfigure}
		\caption{Schematic diagrams of the planar spring pendulum and two-dimensional membrane}
	\end{figure}
	
	The planar spring pendulum has two degrees of freedom, denoted as $x$ and $y$. The dynamic equation governing the motion of the spring pendulum is given by~\cite{geradinMechanicalVibrationsTheory2015}
	\begin{equation}\label{eq:pul}
		\left[ \begin{matrix}
			m & 0 \\
			0 & m \\
		\end{matrix} \right]
		\left[ \begin{matrix}
			\ddot{x} \\
			\ddot{y} \\
		\end{matrix} \right]
		+\begin{bmatrix}
			\begin{matrix}
				k\left( 1 - \dfrac{\ell_0}{\ell} \right)x \\[2ex]
				k\left( 1 - \dfrac{\ell_0}{\ell} \right)y \\
			\end{matrix}
		\end{bmatrix}
		=
		\left[ \begin{matrix}
			mg \\
			0  \\
		\end{matrix} \right]
	\end{equation}
	where $m = 1$ kg is the mass of the ball; $g = 10$ m/s\textsuperscript{2} is the gravitational acceleration; $\ell = \sqrt{x^2 + y^2}$ is the length of the pendulum; $\ell_0 = 1$ m is the initial length of the spring and $k = 30$ N/m is the spring stiffness. The initial conditions are given by:
	\begin{equation}
		\left\{
		\begin{aligned}
			& x(0) = 0             \\
			& y(0) = 1.5 \text{ m}
		\end{aligned}
		\right., \quad
		\left\{
		\begin{aligned}
			& \dot{x}(0) = 0 \\
			& \dot{y}(0) = 0.
		\end{aligned}
		\right.
	\end{equation}
	
	\cref{fig:11} presents the displacement and velocity results computed by different explicit algorithms with a time step of $\Delta t = 0.005$ seconds. The reference solution is obtained using the fourth-order explicit Runge-Kutta method with $\Delta t = 10^{-8}$ seconds. The non-dissipative member of TSSE~\cite{liDevelopmentCompositeSubstep2021} is employed, while ICL~\cite{kimSimpleExplicitSingle2019} and GSSE~\cite{liIdenticalSecondOrder2021} are used with parameters $\rho_b = \rho_s = 0.5$.
	It is observed that the first-order explicit algorithm (TSSE~\cite{liDevelopmentCompositeSubstep2021}) accumulates significant errors over long-term integration, even when the zero-dissipation setting $\rho_b = 1$ is applied. ICL~\cite{kimSimpleExplicitSingle2019} and GSSE~\cite{liIdenticalSecondOrder2021} yield nearly overlapping displacement predictions, but show substantial errors in the velocity responses. A clear discrepancy remains when compared to the reference solution, with noticeable errors emerging beyond $t = 20$ s. In contrast, the two-sub-step algorithms (Noh-Bathe method~\cite{nohExplicitTimeIntegration2013}, the proposed algorithm by Wang et al.~\cite{wang_GeneralizedSinglestepMultistage_2025} and SS2HE\textsubscript{1}~\cite{liTwoThirdorderExplicit2022}) as well as the proposed \novelalgrefs{1}{2} produce results that closely match the reference solution.
		\begin{figure}[htbp]
		\centering
		\includegraphics[width=\textwidth]{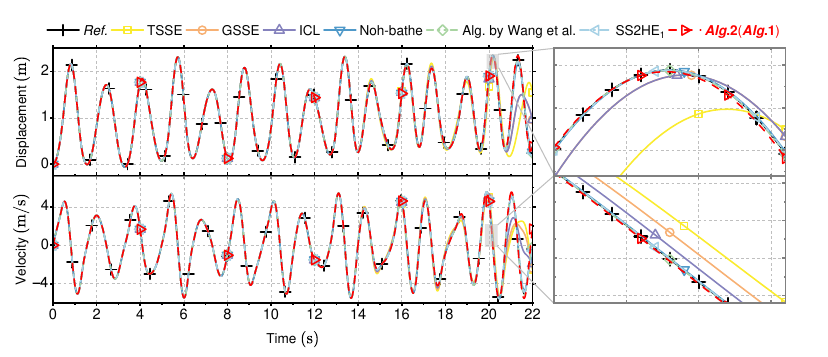}
		\caption{Numerical solutions of displacement and velocity computed using different algorithms for the spring pendulum system \eqref{eq:pul} by $\dt=0.005$s}
		\label{fig:11}
	\end{figure}
	
	Reducing the time step to $\Delta t = 0.0001\,$s and extending the simulation reveals, under identical parameter settings, all single-solve algorithms begin to deviate significantly from the reference solution at approximately $t = 39\,$s, shown in \cref{fig:122}.
	Around $t = 47$ s, the two-sub-step Noh-Bathe method~\cite{nohExplicitTimeIntegration2013} also exhibits noticeable phase errors. In contrast, the third-order accurate algorithms, including the method proposed by Wang et al.~\cite{wang_GeneralizedSinglestepMultistage_2025}, SS2HE\textsubscript{1}~\cite{liTwoThirdorderExplicit2022} and the proposed \novelalgrefs{1}{2}, maintain high accuracy and remain in close agreement with the reference solution for both displacement and velocity, even up to $t = 50$ s.
	\begin{figure}[htbp]
		\centering
		\includegraphics[width=\textwidth]{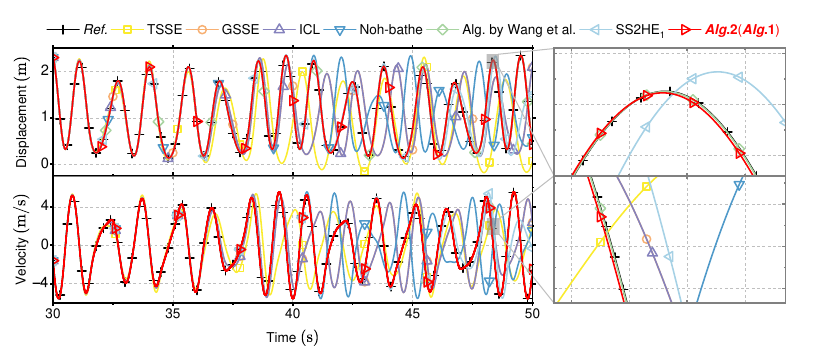}
		\caption{Numerical solutions of displacement and velocity computed using different algorithms for the spring pendulum system \eqref{eq:pul} by $\dt=0.0001$s}
		\label{fig:122}
	\end{figure}

	\subsection{Impact on a one-dimensional rod}
	\label{sec:rod}
	As shown in \cref{fig:12}, a one-dimensional rod under impact is commonly used to test numerical models, particularly for high-frequency spurious components, while also simulating basic impact dynamics and wave propagation problems. The elastic rod subjected to an axial load serves as a model for one-dimensional wave propagation.
	\begin{figure}[h]
		\centering
		\includegraphics[width=0.6\textwidth]{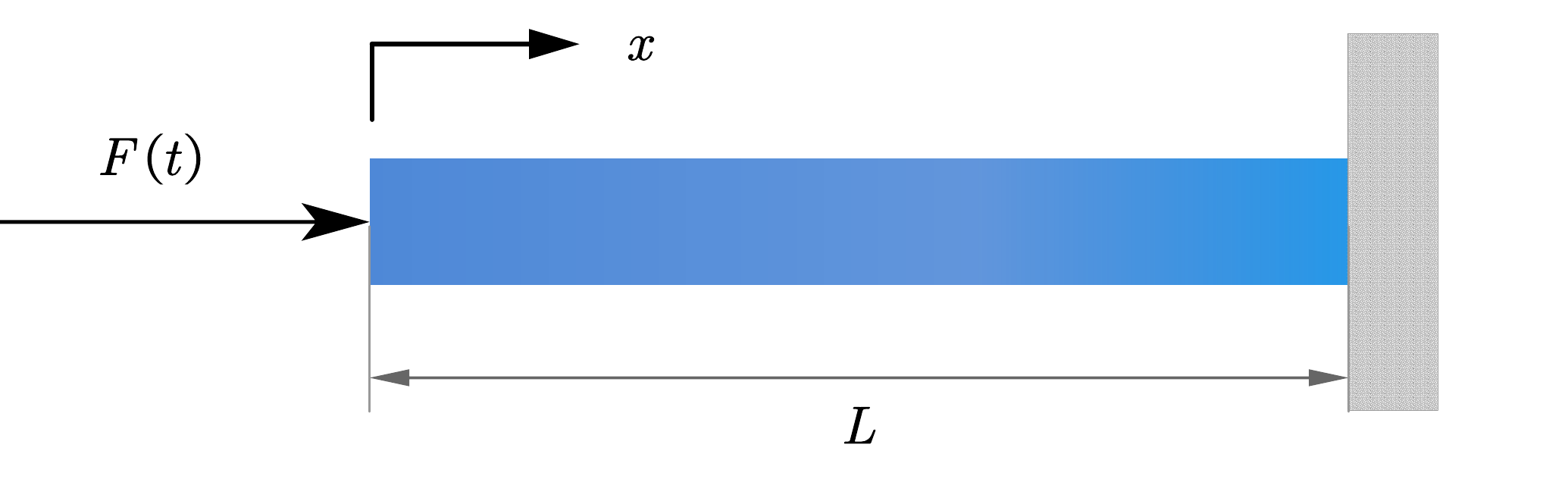}
		\caption{One-dimensional isotropic elastic bar initially at rest subjected to a sudden external force $F(t)=100$N, used to characterize high-frequency response}
		\label{fig:12}
	\end{figure}
	
	The governing partial differential equation for the axial displacement $u(x,t)$ is given by~\cite{hughesFiniteElementMethod2012}
	\begin{equation}\label{eq:55}
		\rho A\frac{{{\partial }^{2}}u(x,\;t)}{\partial {{t}^{2}}}-EA{{\nabla }^{2}}u(x,\;t)=F(t)
	\end{equation}
	where $u(x, t)$ represents the field scalar, and the other physical parameters include the cross-sectional area $A = 1$ m\textsuperscript{2}, the Young's modulus $E = 3 \times 10^7 $ Pa, the length $L = 20$ m and the rod density $\rho = 7.4 \times 10^{-4} $ kg/m\textsuperscript{3}. After spatial discretization using the finite element method, \cref{eq:55} can be reduced to a system of ordinary differential equations in time.
	\begin{equation}
		\bm{M}\bma(t)+c_0^2 \bm{K} \bmu(t) = \bm{F}(t)
	\end{equation}
	where $c_0 = \sqrt{E/\rho}$ is the wave speed. Spatial discretization is carried out using two-node linear elements. The element mass matrix is distinguished into two forms: a diagonal mass matrix and a distributed mass matrix, corresponding to $r = 0$ and $r = 1$, respectively.
	\begin{equation}\label{eq:56}
		{{\bm{m}}^{\text{e}}}=\frac{\Delta x}{6}\left[ \begin{matrix}
			3-r & r   \\
			r   & 3-r \\
		\end{matrix} \right]
	\end{equation}
	where $\Delta x$ is the element length. A total of 100 equally spaced elements are used, and an external force $F(t) = 100$ N is suddenly applied from rest. The exact solution for this system can be found in~\cite{geradinMechanicalVibrationsTheory2015}.
	
	For simplicity, a lumped mass matrix is employed ($r = 0$ in \cref{eq:56}). All explicit algorithms are configured with the most dissipative setting and the time step is set to $\Delta t = \Omega_b / \omega_{\text{max}}$ to induce maximum numerical damping. 
	\cref{fig:133} presents the displacement and velocity responses at the midpoint of the rod, computed by various explicit algorithms.
	It can be observed that non-dissipative algorithms, such as the NE method~\cite{newmarkMethodComputationStructural1959}, fail to filter out the spurious high-frequency components introduced by spatial discretization, leading to pronounced numerical oscillations in the velocity response. First-order algorithms, including TSSE~\cite{liDevelopmentCompositeSubstep2021} and TW~\cite{maheoNumericalDampingSpurious2013}, perform poorly in displacement. However, due to their strong numerical damping, they effectively suppress oscillations—albeit at the cost of significant solution errors. The remaining second-order algorithms (EHHT-$\alpha$~\cite{mirandaImprovedImplicitExplicit1989} when $\xi=0$, GSSE~\cite{liIdenticalSecondOrder2021}, EWBZ-$\alpha$~\cite{hulbertExplicitTimeIntegration1996} and Noh-Bathe~\cite{nohExplicitTimeIntegration2013}) perform well when using a lumped mass matrix. In contrast, the third-order accurate algorithm SS2HE\textsubscript{1} performs poorly, exhibiting large numerical oscillations in the velocity response. The proposed \novelalgrefs{1}{2}, which incorporate controlled numerical dissipation, are able to filter out certain high-frequency artifacts. While some fluctuation remains in the velocity response, the solution progressively smooths over time.
		\begin{figure}[htbp]
		\centering
		\includegraphics[width=\textwidth]{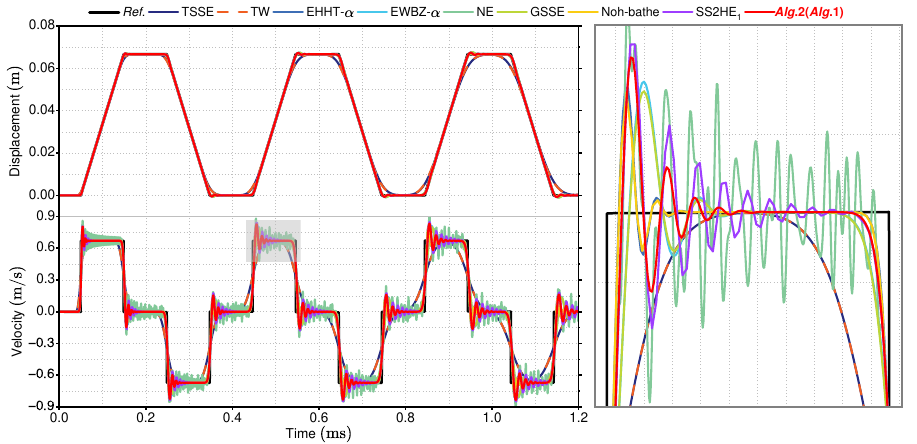}
		\caption{Midpoint displacement and velocity time histories of the bar shown in \cref{fig:12}, obtained via spatial discretization with $r=0$ and solved using various explicit schemes}
		\label{fig:133}
	\end{figure}
	\begin{figure}[htbp]
		\centering
		\includegraphics[width=\textwidth]{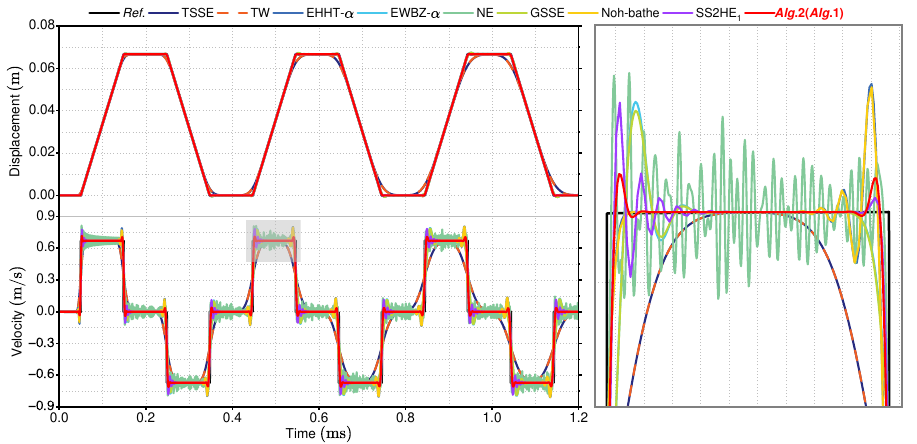}
		\caption{Midpoint displacement and velocity time histories of the bar shown in \cref{fig:12}, obtained via spatial discretization with $r=1/2$ and solved using various explicit schemes}
		\label{fig:13}
	\end{figure}
	
	As demonstrated in the study~\cite{liFurtherAssessmentThree2021}, the overall solution accuracy is influenced by the coupling errors between spatial and temporal discretization. To fully realize the potential of high-order time integration methods, they must be complemented by higher-order spatial discretization. Since the algorithms under consideration are third-order accurate, they should be used in conjunction with high-order mass matrices. Specifically, when $r = 1/2$, Eq.~\eqref{eq:56} yields a higher-order mass matrix.

	The displacement and velocity responses at the midpoint of the rod, using $r = 1/2$, are shown in \cref{fig:13}. Upon applying higher-order mass elements, second-order algorithms show little improvement, while the third-order algorithms SS2HE\textsubscript{1}~\cite{liTwoThirdorderExplicit2022}, \novelalgrefs{1}{2} demonstrate significant enhancements compared to the results in \cref{fig:133}. Specifically, SS2HE\textsubscript{1}~\cite{liTwoThirdorderExplicit2022} eliminates oscillations only at one end of the velocity step, whereas the proposed \novelalgrefs{1}{2} almost entirely remove all numerical oscillations, highlighting their strong advantage. This observation confirms that the inherent numerical dissipation of \novelalgrefs{1}{2} is not a drawback, but rather a beneficial feature.

	\subsection{Center-loaded square membrane}
	As illustrated in \cref{fig:wave2d}, the center-loaded square membrane, with a side length of 15 units, is commonly used to evaluate the futher numerical dissipation capabilities of algorithms. The governing partial differential equation for scalar wave propagation is given by
	\begin{equation}
		\frac{{{\partial }^{2}}u(x,\;y,\;t)}{\partial {{t}^{2}}}-{{c}_{0}}{{\nabla }^{2}}u(x,\;y,\;t)={{F}_{c}}(x,\;y,\;t)
	\end{equation}
	where $c_0$ is the wave speed, assumed to be $c_0 = 1$; $u(x, y, t)$ represents the displacement perpendicular to the plane of the membrane, and $F_c$ is the known external excitation function. Following the approach in \cref{sec:rod}, the governing equation is discretized spatially using finite elements, which transforms it into a system of ordinary differential equations in time. In the case of the two-dimensional problem, four-node rectangular elements are used. The element mass matrix and stiffness matrix are defined as $\bm{m}^{\text{e}} = \bm{m}^{\text{e}}(\alpha_m, r, \zeta, \Delta x)$ and ${\bm{k}^{\text{e}}} = \bm{k}^{\text{e}}(\alpha_k, r, \zeta)$, where $\Delta x$ represents the element width; $\zeta = \Delta y / \Delta x$ is the aspect ratio of the rectangular element; $r$ is the weighting parameter of the element mass matrix and $\alpha_k = \alpha_m = 1/\sqrt{3}$ are the Gauss quadrature points used in standard finite element.
	
	For simplicity, the parameters are set as $\zeta = 1$ and high-order mass elements with $r = 1/2$ are selected. The external excitation is defined as
	\begin{equation}
		{{F}_{c}}(0,\;0,\;t)=4\left[ 1-{{\left( 2t-1 \right)}^{2}} \right]\cdot H(1-t)
	\end{equation}
	where $H$ is the Heaviside step function. The analytical solution for this problem can be found in~\cite{yueDispersionreducingFiniteElements2005}.
	
	\begin{figure}[htbp]
		\centering
		\includegraphics[width=\textwidth]{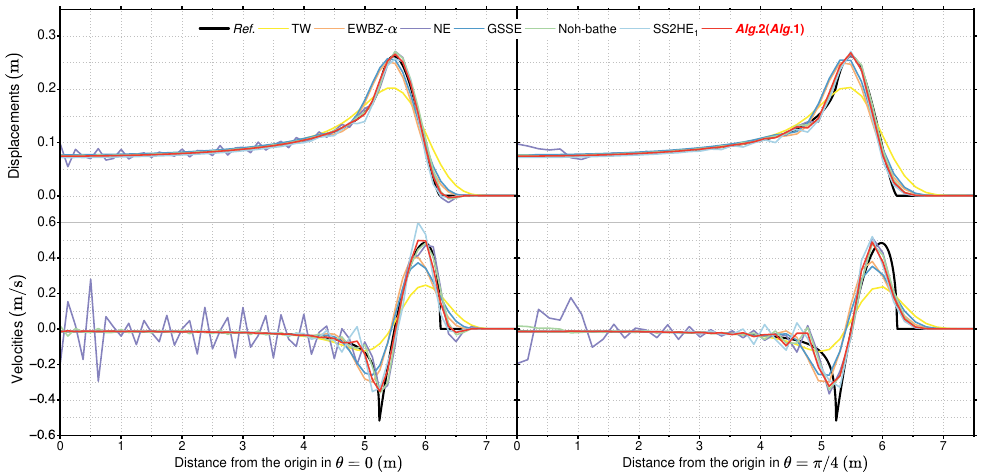}
		\caption{Response of 2D scalar wave shown in \cref{fig:wave2d} at $t=6.25\text{s}$ with $\alpha_k=1/\sqrt{3}, \mathrm{CFL=0.5}$ ($\mathrm{CFL=1}$ for two-sub-step algorithms) for different algorithms}
		\label{fig:14}
	\end{figure}
	To evaluate algorithmic performance, GSSE~\cite{liIdenticalSecondOrder2021}, NE~\cite{newmarkMethodComputationStructural1959}, TW~\cite{maheoNumericalDampingSpurious2013}, EWBZ-$\alpha$~\cite{hulbertExplicitTimeIntegration1996}, Noh-Bathe~\cite{nohExplicitTimeIntegration2013}, SS2E\textsubscript{1}~\cite{liTwoThirdorderExplicit2022} as well as the proposed \novelalgrefs{1}{2} are employed to solve the problem under a Courant–Friedrichs–Lewy (CFL) number of $\text{CFL} = {c_0 \Delta t}/{\Delta x} = 0.5$ (two-sub-step schemes adopt $\mathrm{CFL} = 1$ for comparability). Spatial discretization is performed using a $120 \times 120$ grid of rectangular elements. \cref{fig:14} presents the numerical solutions along the $\theta = 0$ and $\theta = \pi/4$ directions at $t = 6.25$ seconds.
	Among the methods tested, the zero-dissipation NE method~\cite{newmarkMethodComputationStructural1959} exhibits significant numerical oscillations in both displacement and velocity, due to the presence of spurious high-frequency components. TW~\cite{maheoNumericalDampingSpurious2013} is implemented with ${\rho_b} = 0.5$ while all other methods use ${\rho_b} = 0$. Despite the moderate numerical damping, TW~\cite{maheoNumericalDampingSpurious2013} shows notable amplitude errors. Similarly, the Noh-Bathe method~\cite{nohExplicitTimeIntegration2013} and SS2E\textsubscript{1}~\cite{liTwoThirdorderExplicit2022} algorithms also suffer from pronounced oscillations in velocity field (see \cref{fig:noh} for details of the oscillations).
	In contrast, the GSSE~\cite{liIdenticalSecondOrder2021}, EWBZ-$\alpha$~\cite{hulbertExplicitTimeIntegration1996} and the proposed \novelalgrefs{1}{2} demonstrate effective suppression of high-frequency noise in both displacement and velocity fields. However, it is worth noting that GSSE~\cite{liIdenticalSecondOrder2021} and EWBZ-$\alpha$~\cite{hulbertExplicitTimeIntegration1996} exhibit relatively larger velocity errors far from the origin. Furthermore, they show discrepancies between $\theta = 0$ and $\theta = \pi/4$ directions, indicating a lack of isotropy in their solutions.
	By comparison, only \novelalgrefs{1}{2} not only filter out high-frequency oscillations effectively but also preserve isotropy in the solution, delivering consistently accurate results without introducing significant directional bias.
	
	As discussed in \cref{sec:rod}, high-order time integration algorithms should be paired with appropriate spatial discretization techniques to minimize coupling errors and achieve more accurate numerical solutions. Unlike one-dimensional elements, in this case, it is essential not only to set the element mass matrix with $r = 1/2$, but also to adjust the Gauss quadrature points in the element stiffness matrix to eliminate anisotropy in the solution. Specifically, ${\alpha_k} = \sqrt{2/3}$ is required~\cite{liFurtherAssessmentThree2021}.
	\cref{fig:15} illustrates that when ${\alpha_k} = \sqrt{2/3}$, high-frequency oscillations are significantly suppressed, and all algorithms exhibit improved isotropy in the displacement and velocity curves along both the $\theta = \pi /4$ and $\theta = 0$ directions. Among the tested methods, \novelalgrefs{1}{2} deliver the best results in both displacement and velocity.

\begin{figure}[htbp]
	\centering
	\includegraphics[width=\textwidth]{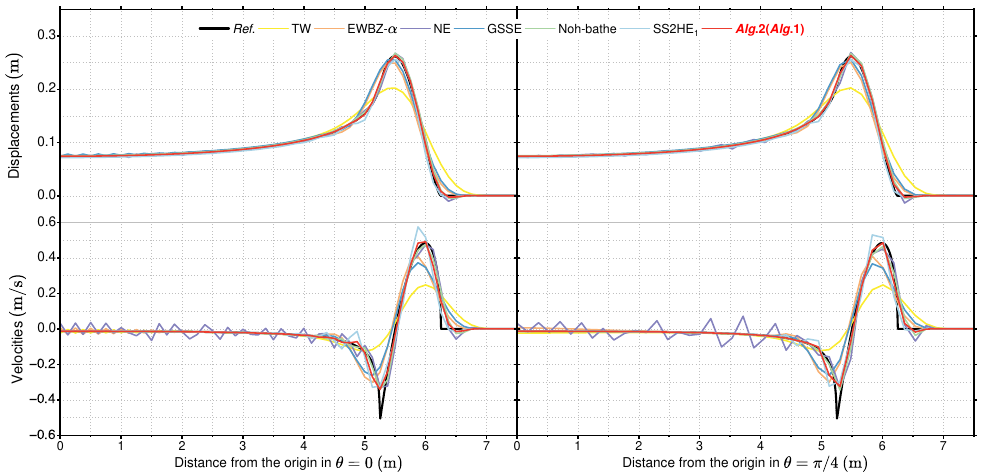}
	\caption{Response of 2D scalar wave shown in \cref{fig:wave2d} at $t=6.25\text{s}$ with $\alpha_k=2/\sqrt{3}, \mathrm{CFL=0.5}$ ($\mathrm{CFL=1}$ for two-sub-step algorithms) for different algorithms}
	\label{fig:15}
\end{figure}

	\begin{figure}[htbp]
		\centering
		\begin{subfigure}[h]{0.22\textwidth}
			\centering
			\includegraphics[width=\textwidth]{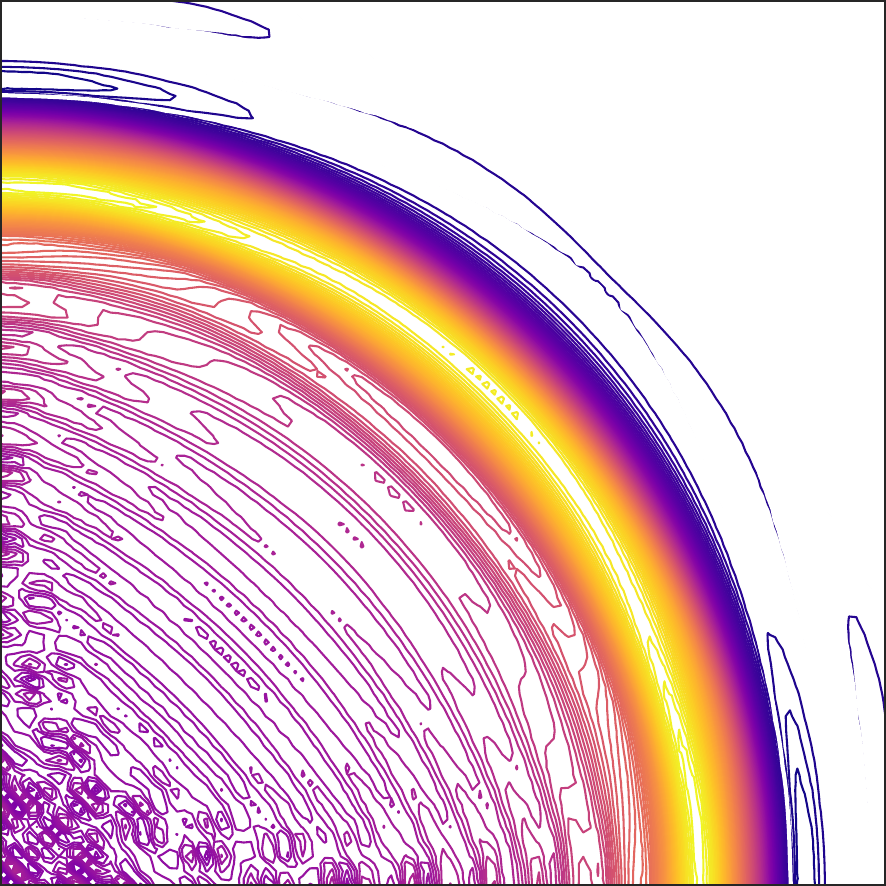}
			\caption{NE\cite{newmarkMethodComputationStructural1959} with $\alpha_k=\sqrt{1/3}$}
		\end{subfigure}
		\begin{subfigure}[h]{0.22\textwidth}
			\centering
			\includegraphics[width=\textwidth]{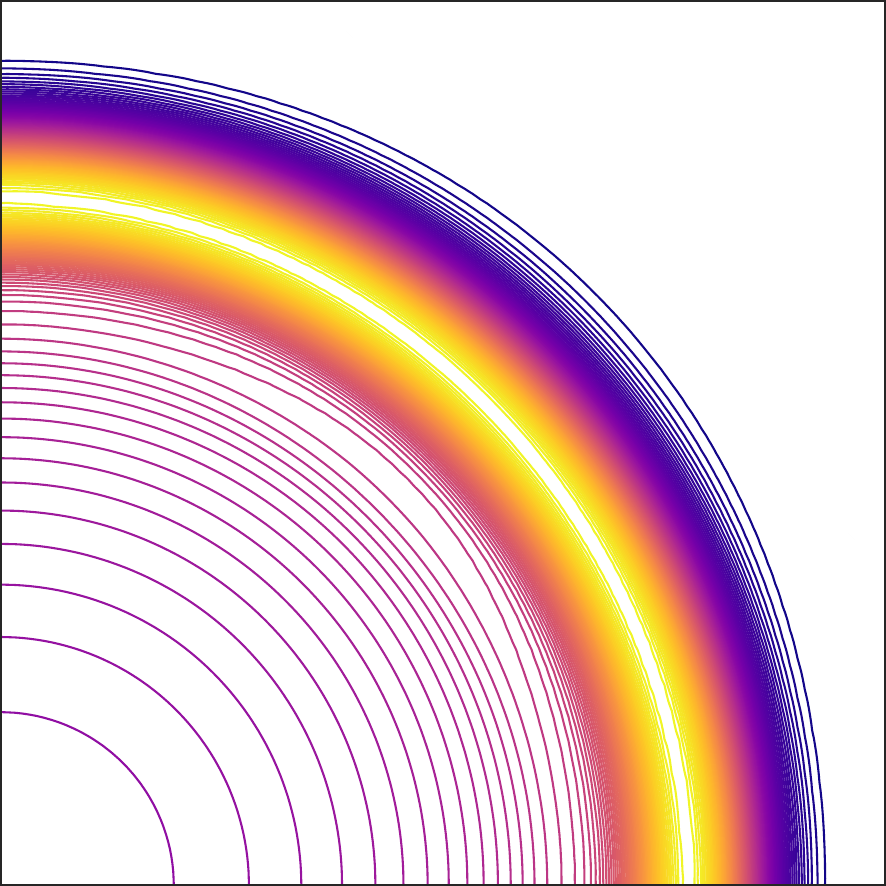}
			\caption{GSSE~\cite{liIdenticalSecondOrder2021} with $\alpha_k=\sqrt{1/3}$}
		\end{subfigure}
		\begin{subfigure}[h]{0.22\textwidth}
			\centering
			\includegraphics[width=\textwidth]{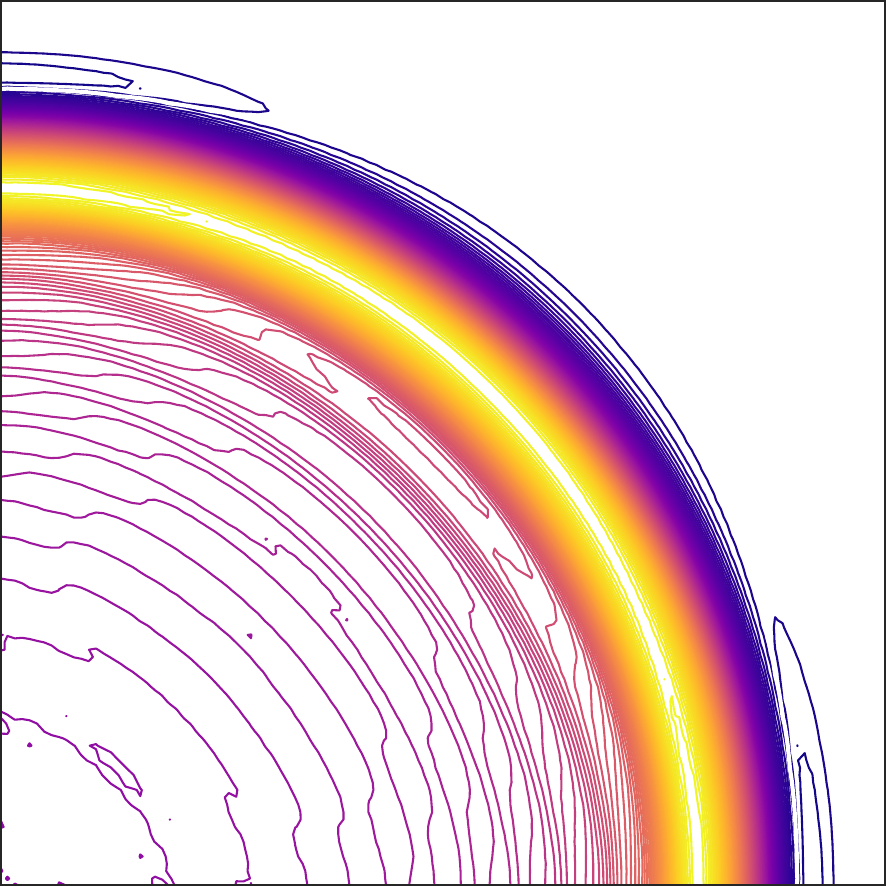}
			\caption{Noh-Bathe~\cite{nohExplicitTimeIntegration2013} with $\alpha_k=\sqrt{1/3}$}
		\end{subfigure}
		\begin{subfigure}[h]{0.22\textwidth}
			\centering
			\includegraphics[width=\textwidth]{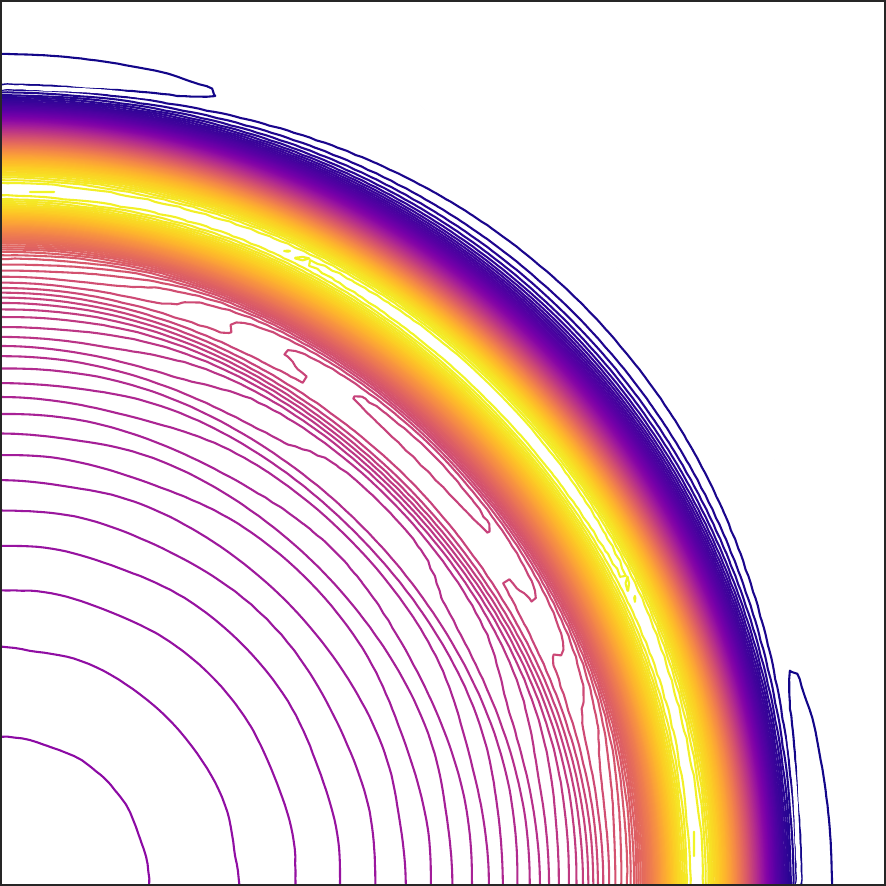}
			\caption{\novelalgrefss{1}{2} with $\alpha_k=\sqrt{1/3}$}
		\end{subfigure}
		\\
		\begin{subfigure}[h]{0.22\textwidth}
			\centering
			\includegraphics[width=\textwidth]{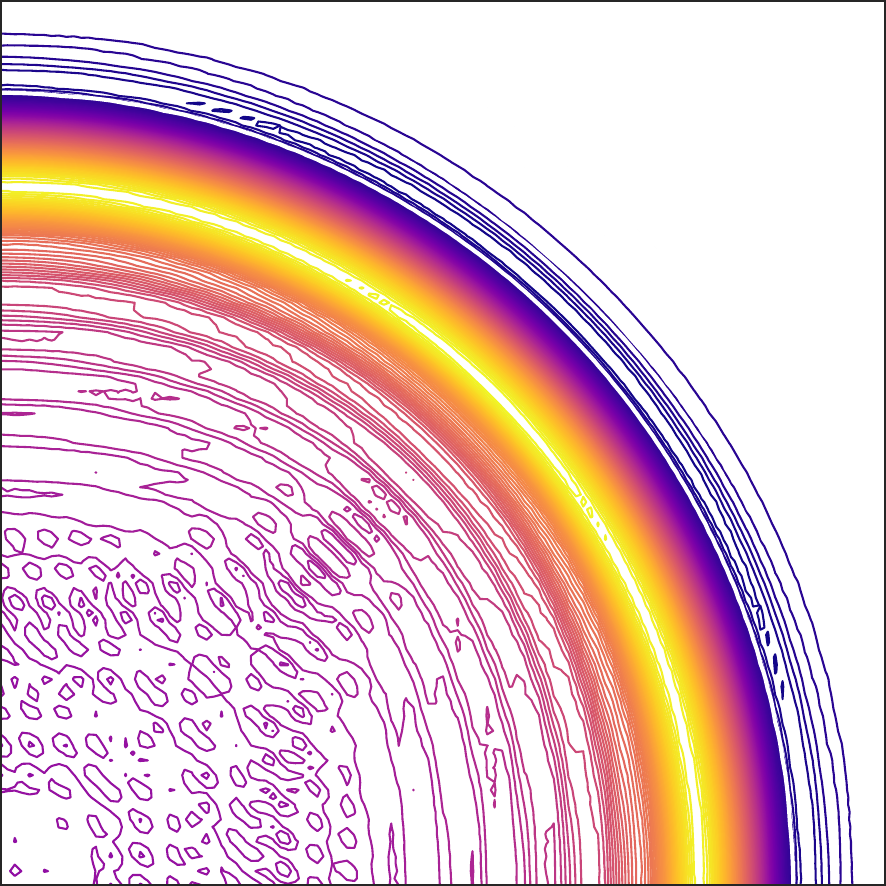}
			\caption{NE~\cite{newmarkMethodComputationStructural1959} with $\alpha_k=\sqrt{2/3}$}
		\end{subfigure}
		\begin{subfigure}[h]{0.22\textwidth}
			\centering
			\includegraphics[width=\textwidth]{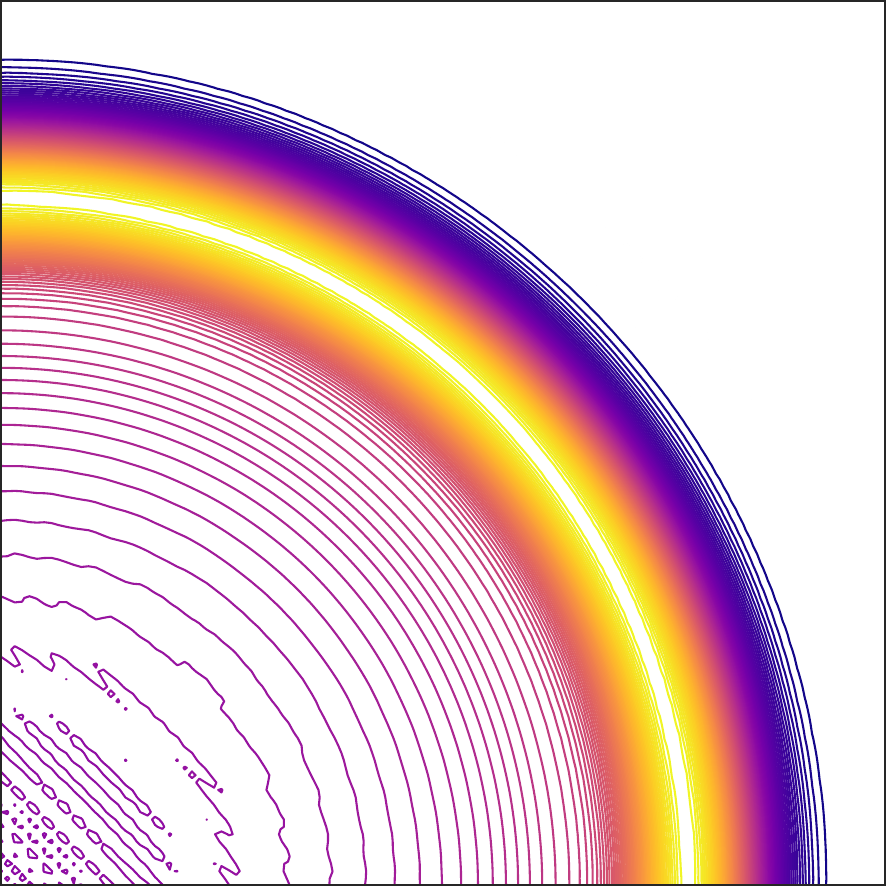}
			\caption{GSSE~\cite{liIdenticalSecondOrder2021} with $\alpha_k=\sqrt{2/3}$}
		\end{subfigure}
		\begin{subfigure}[h]{0.22\textwidth}
			\centering
			\includegraphics[width=\textwidth]{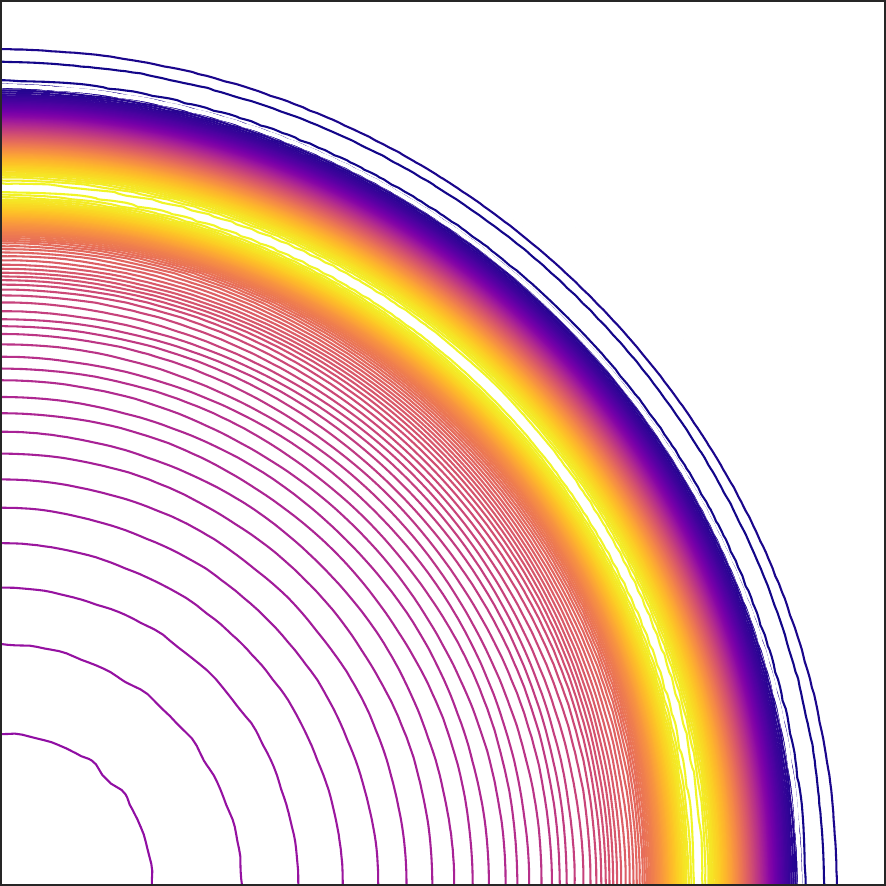}
			\caption{Noh-Bathe~\cite{nohExplicitTimeIntegration2013} with ${\alpha_k=\sqrt{2/3}}$}
		\end{subfigure}
		\begin{subfigure}[h]{0.22\textwidth}
			\centering
			\includegraphics[width=\textwidth]{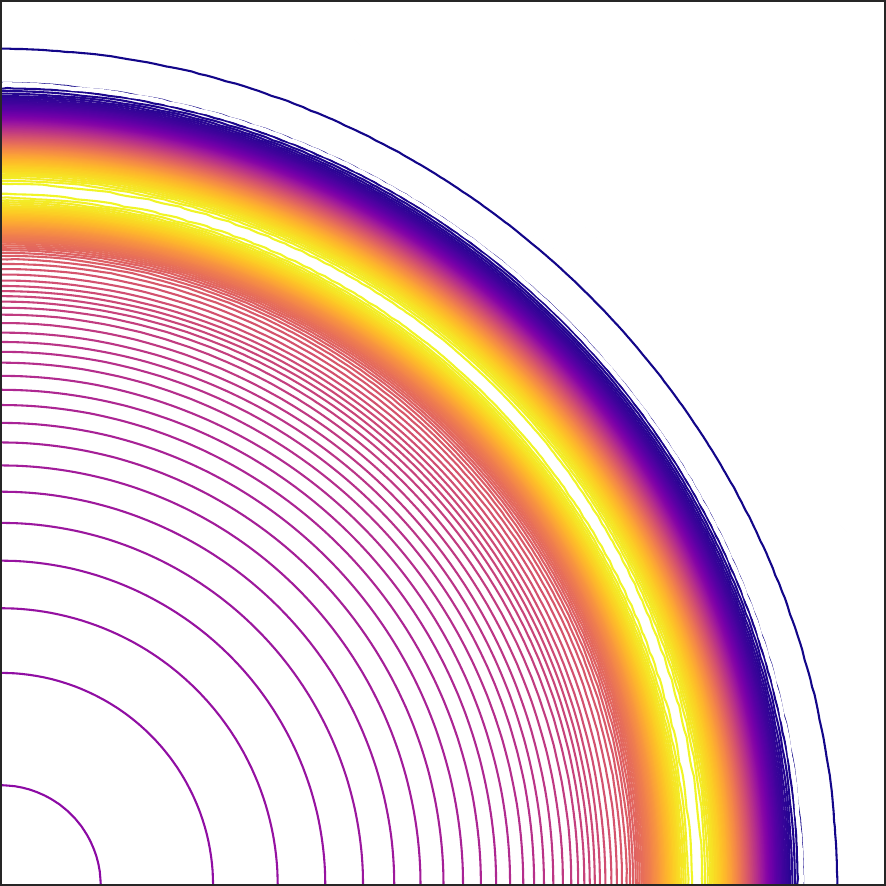}
			\caption{\novelalgrefss{1}{2} with $\alpha_k=\sqrt{2/3}$}
		\end{subfigure}
		\caption{Displacement contour plots of a 2D scalar wave computed by various explicit algorithms ($\mathrm{CFL=0.5}$ for single-solve algorithms and $\mathrm{CFL=1}$ for Noh-Bathe~\cite{nohExplicitTimeIntegration2013}) using $\alpha_k = \sqrt{1/3}$ and $\alpha_k = \sqrt{2/3}$}
		\label{fig:16}
	\end{figure}
	
	\begin{figure}[htbp]
		\centering
		\begin{subfigure}[h]{0.22\textwidth}
			\centering
			\includegraphics[width=\textwidth]{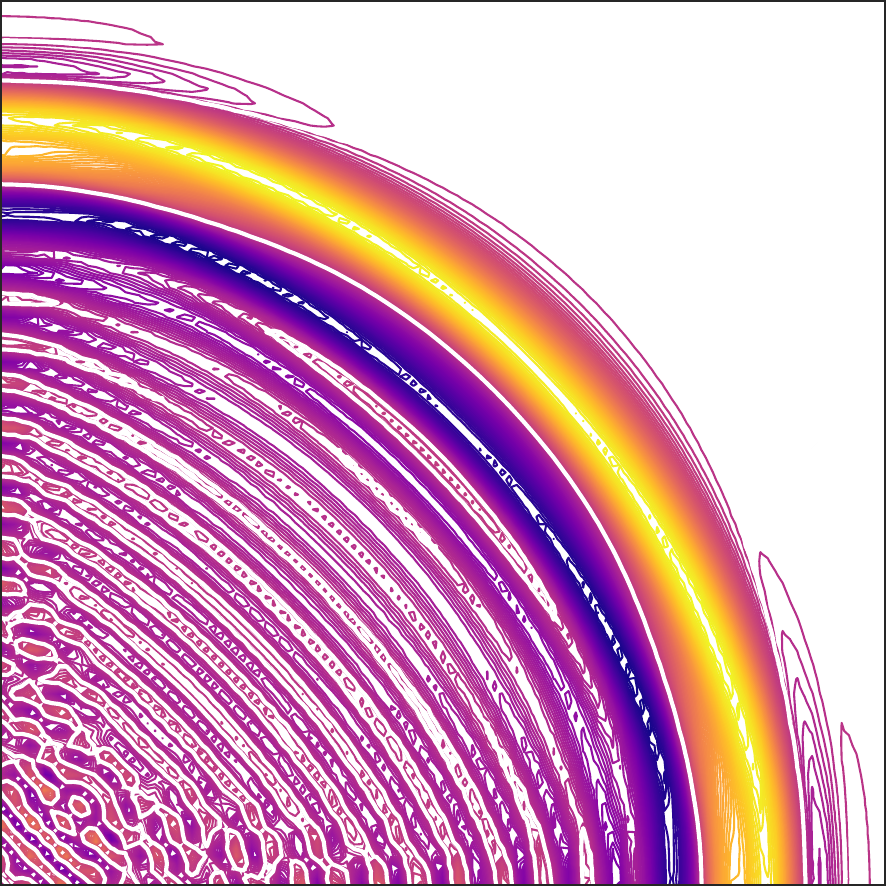}
			\caption{NE~\cite{newmarkMethodComputationStructural1959} with $\alpha_k=\sqrt{1/3}$}
		\end{subfigure}
		\begin{subfigure}[h]{0.22\textwidth}
			\centering
			\includegraphics[width=\textwidth]{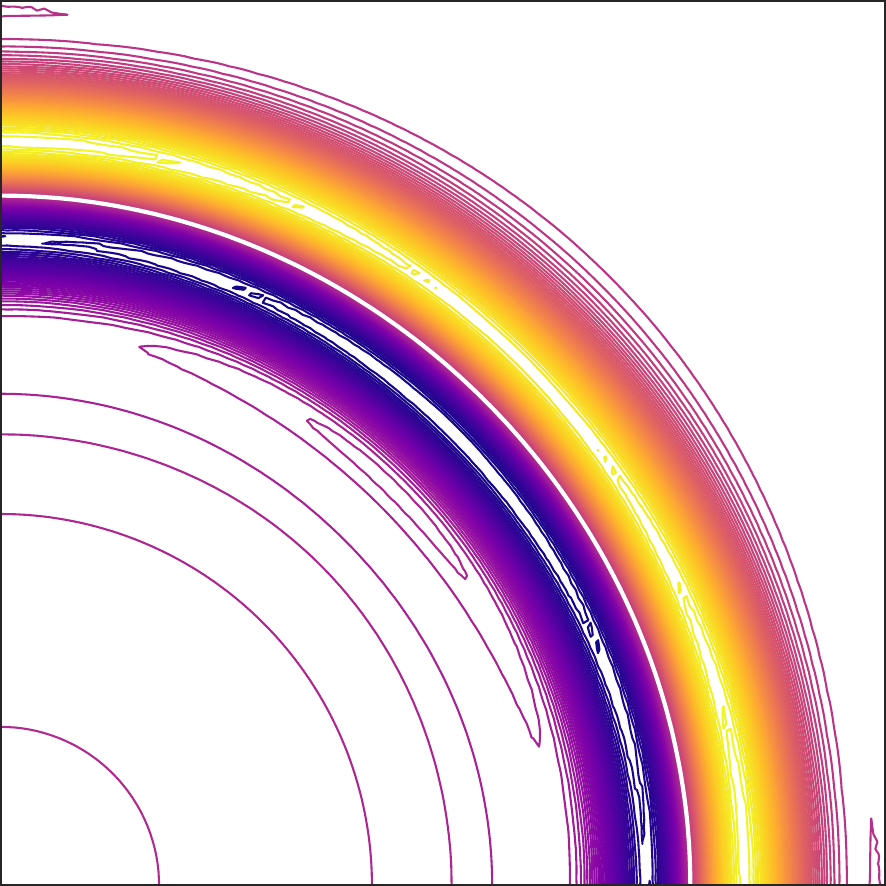}
			\caption{GSSE~\cite{liIdenticalSecondOrder2021} with $\alpha_k=\sqrt{1/3}$}
		\end{subfigure}
		\begin{subfigure}[h]{0.22\textwidth}
			\centering
			\includegraphics[width=\textwidth]{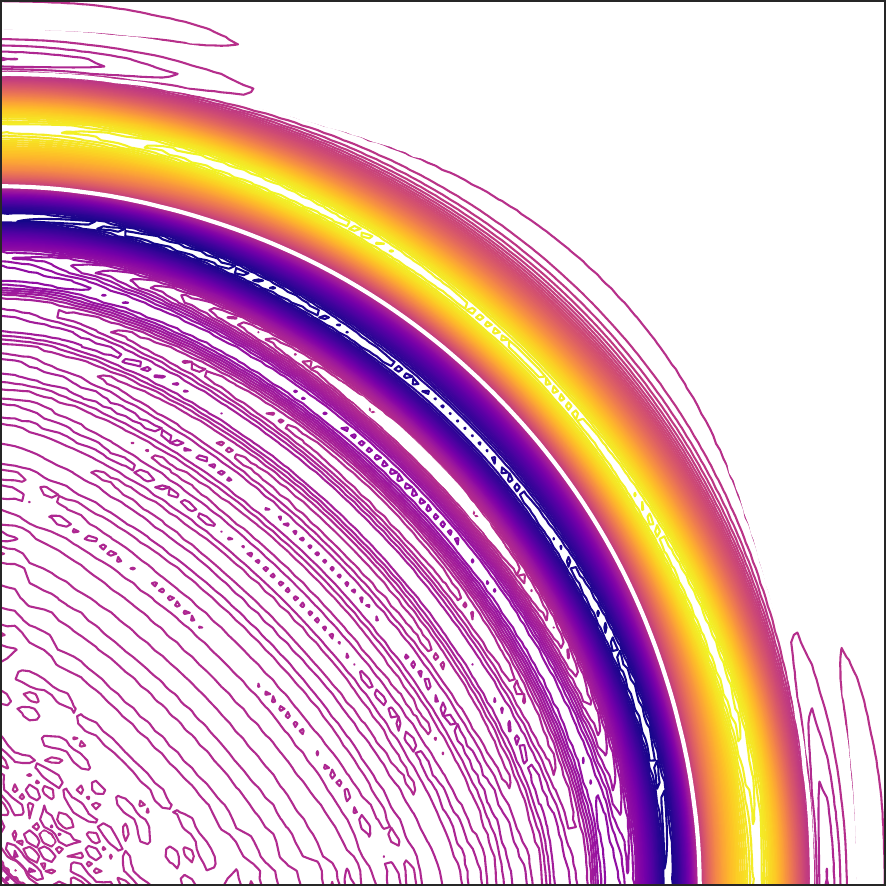}
			\caption{Noh-Bathe~\cite{nohExplicitTimeIntegration2013} with $\alpha_k=\sqrt{1/3}$}
			\label{fig:noh}
		\end{subfigure}
		\begin{subfigure}[h]{0.22\textwidth}
			\centering
			\includegraphics[width=\textwidth]{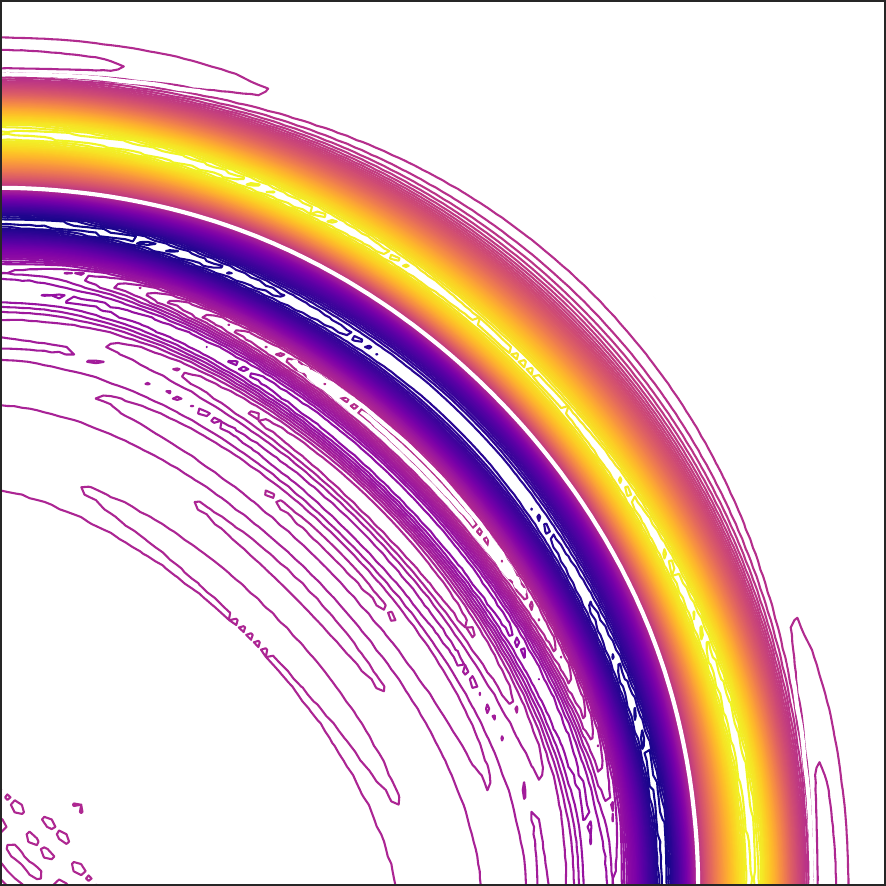}
			\caption{\novelalgrefss{1}{2} with $\alpha_k=\sqrt{1/3}$}
		\end{subfigure}
		\\
		\begin{subfigure}[h]{0.22\textwidth}
			\centering
			\includegraphics[width=\textwidth]{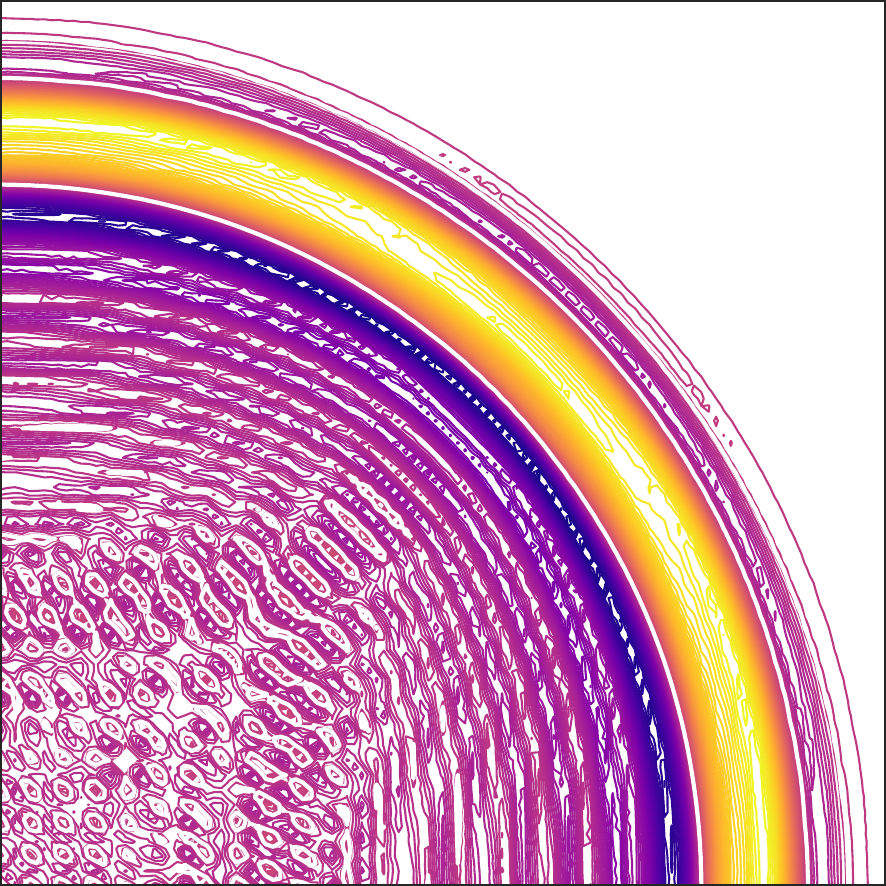}
			\caption{NE~\cite{newmarkMethodComputationStructural1959} with $\alpha_k=\sqrt{2/3}$}
		\end{subfigure}
		\begin{subfigure}[h]{0.22\textwidth}
			\centering
			\includegraphics[width=\textwidth]{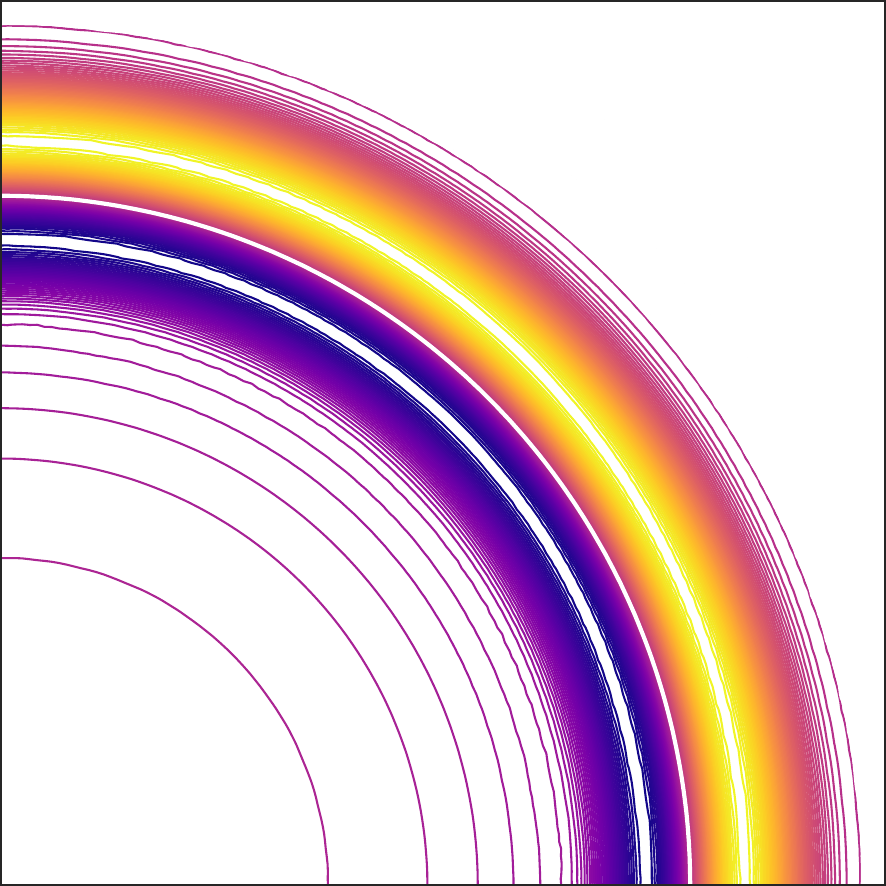}
			\caption{GSSE~\cite{liIdenticalSecondOrder2021} with $\alpha_k=\sqrt{2/3}$}
		\end{subfigure}
		\begin{subfigure}[h]{0.22\textwidth}
			\centering
			\includegraphics[width=\textwidth]{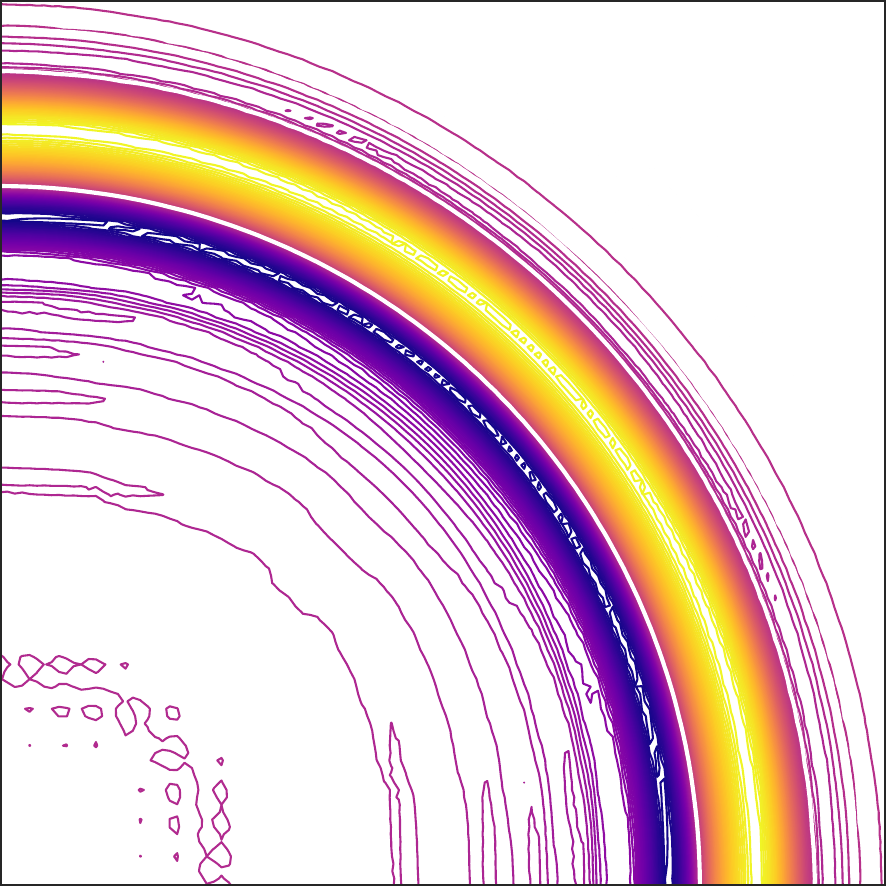}
			\caption{Noh-Bathe~\cite{nohExplicitTimeIntegration2013} with $\alpha_k=\sqrt{2/3}$}
		\end{subfigure}
		\begin{subfigure}[h]{0.22\textwidth}
			\centering
			\includegraphics[width=\textwidth]{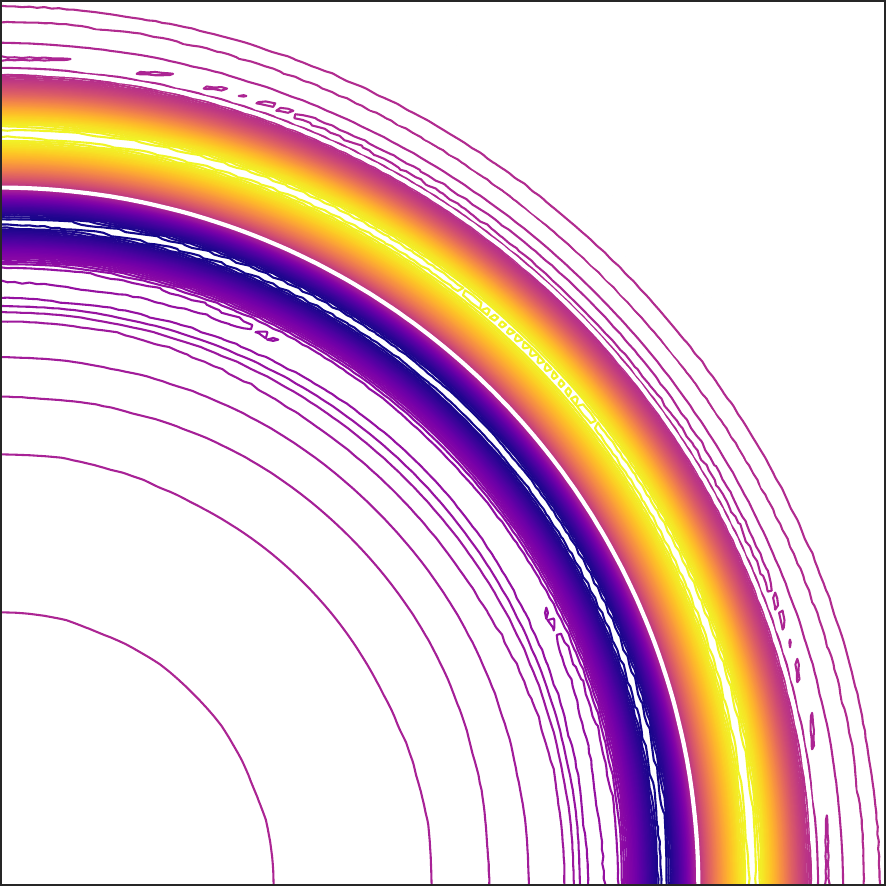}
			\caption{\novelalgrefss{1}{2} with $\alpha_k=\sqrt{2/3}$}
		\end{subfigure}
		\caption{Velocity contour plots of a 2D scalar wave computed by various explicit algorithms ($\mathrm{CFL=0.5}$ for single-solve algorithms and $\mathrm{CFL=1}$ for Noh-Bathe~\cite{nohExplicitTimeIntegration2013}) using $\alpha_k = \sqrt{1/3}$ and $\alpha_k = \sqrt{2/3}$}
		\label{fig:17}
	\end{figure}
		
	To further visualize the exceptional isotropy of \novelalgrefs{1}{2}, \cref{fig:16,fig:17} present the displacement and velocity contour plots predicted by different algorithms (only one-quarter of the solution region is plotted due to symmetry). The results show that, in terms of displacement, a combination of \cref{fig:14,fig:15,fig:16} leads to the conclusion that when $\alpha_k = 1/\sqrt{3}$, GSSE~\cite{liIdenticalSecondOrder2021} exhibits the best isotropy. However, its solution accuracy is not as high as that of \novelalgrefs{1}{2}. When $\alpha_k = \sqrt{2/3}$, the two-sub-step Noh-Bathe method~\cite{nohExplicitTimeIntegration2013} exhibits improved isotropy, but it still falls short of \novelalgrefs{1}{2}, which not only maintain perfect isotropy but also deliver the most accurate solution.
	A similar trend appears in the velocity responses shown in \cref{fig:14,fig:15,fig:17}. NE~\cite{newmarkMethodComputationStructural1959} and Noh-Bathe methods~\cite{nohExplicitTimeIntegration2013} perform very poorly when $\alpha_k = \sqrt{1/3}$. For $\alpha_k = \sqrt{2/3}$, the Noh-Bathe method~\cite{nohExplicitTimeIntegration2013} slightly reduces anisotropy and oscillation amplitude. In contrast, \novelalgrefs{1}{2} not only achieve perfect isotropy in velocity at $\alpha_k = \sqrt{2/3}$ but also deliver the most accurate results.
	
	Overall, the results from this test, along with those in \cref{sec:rod}, highlight the significant advantage of \novelalgrefs{1}{2} in handling semi-discretized equations with spurious high-frequency components.
	
	\subsection{Gap nonlinearity in folding rudder}
	
	In this section, a numerical simulation is performed on a large-scale, complex engineering structure—the folding rudder—which exhibits strong gap-induced nonlinear behavior due to its mechanical connections. The geometric model of the folding rudder shown in \cref{fig:18a}, illustrates the assembly between the rudder shaft and the rudder surface. The four contact points at the root, highlighted in blue, represent the mechanical connections between these two components. Due to manufacturing tolerances, gaps exist at these connections, leading to substantial nonlinear effects during motion. For a detailed mathematical model and analysis of gap nonlinearity in folding rudders, please refer to the article~\cite{NonlinearAeroelasticAnalysis2020}.

	\begin{figure}[h]
		\centering
		\begin{subfigure}[b]{0.49\textwidth}
			\centering
			\includegraphics[width=0.9\textwidth]{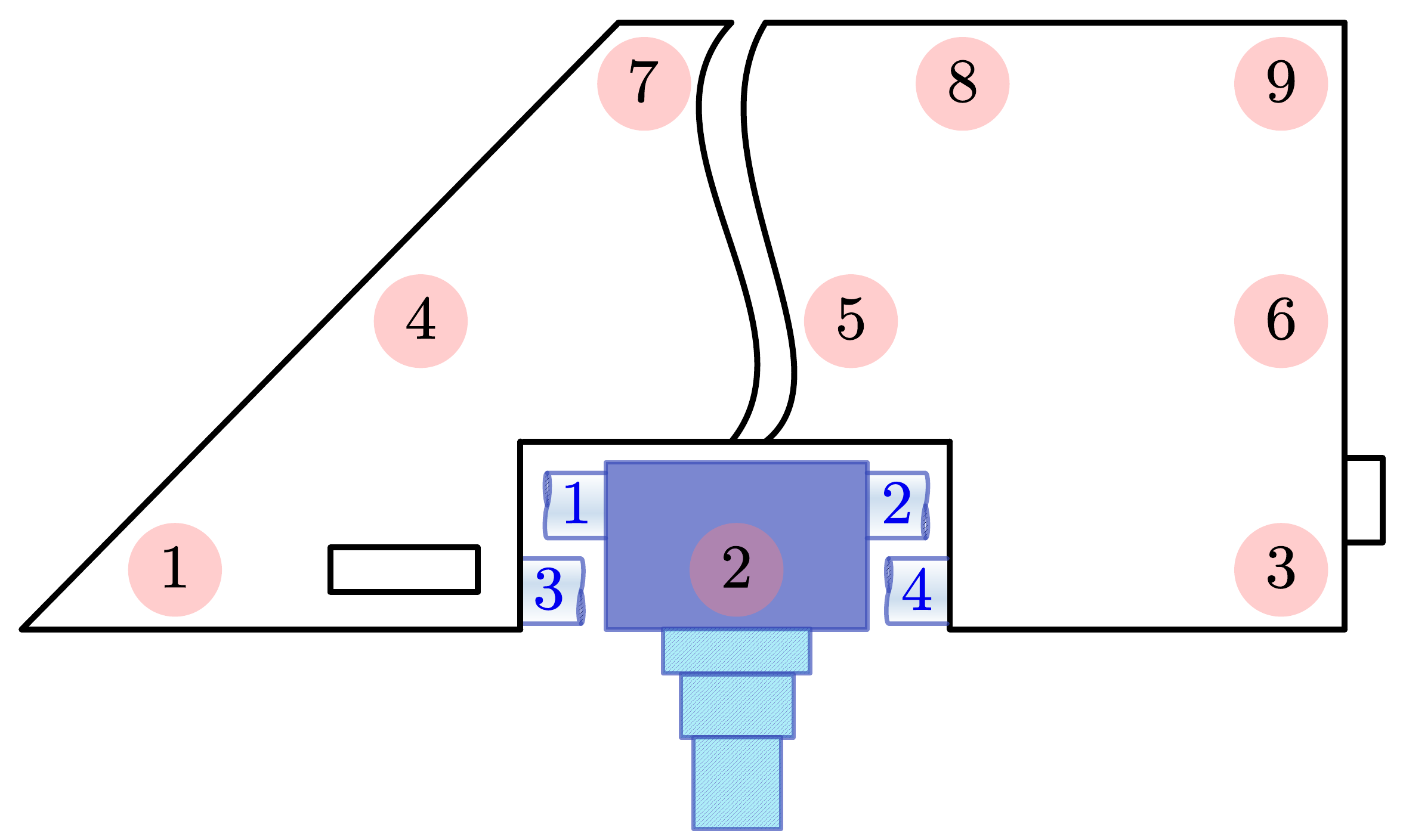}
			\caption{Geometric model of the gap folding rudder}
			\label{fig:18a}
		\end{subfigure}
		\begin{subfigure}[b]{0.49\textwidth}
			\centering
			\includegraphics[width=0.8\textwidth]{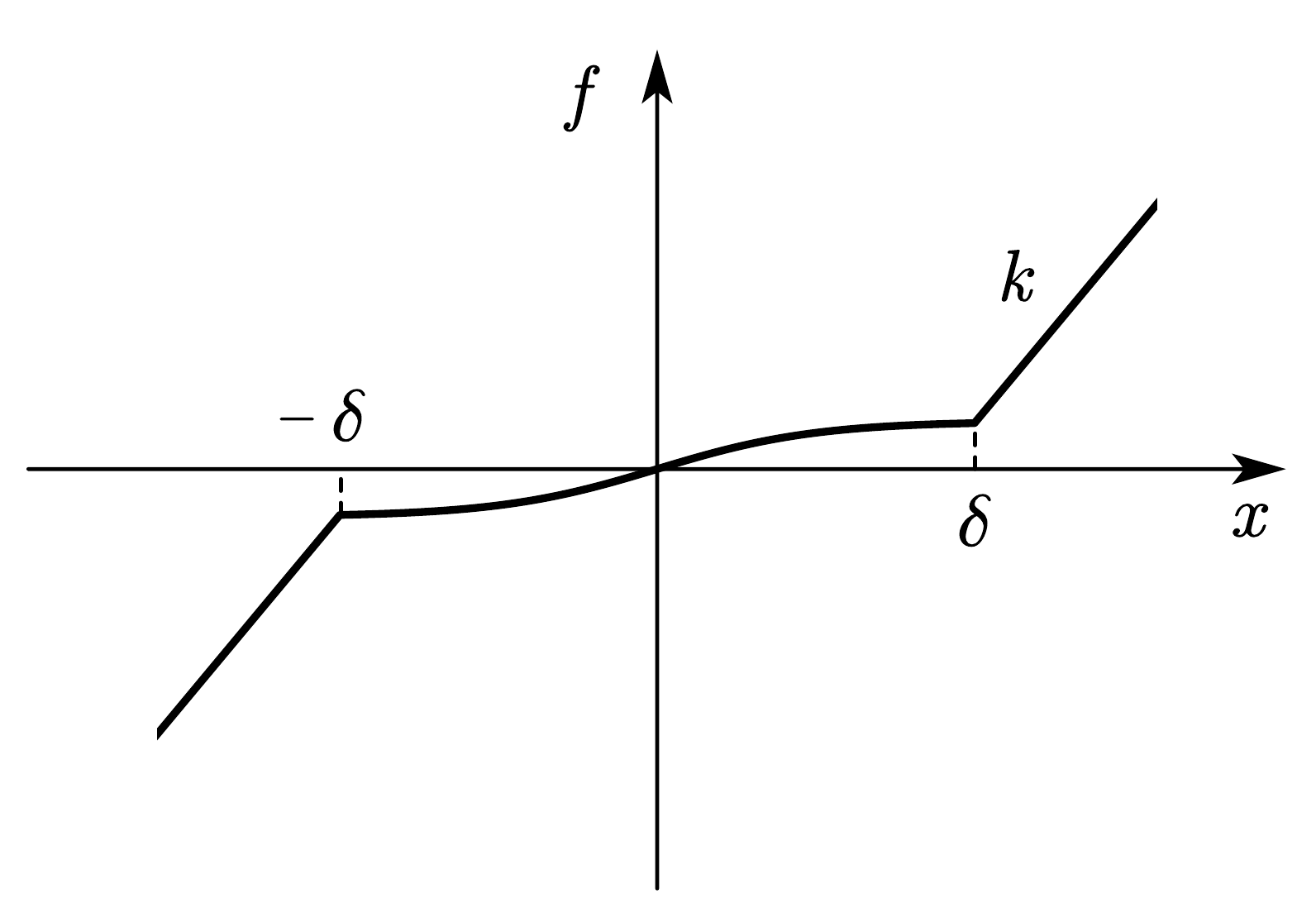}
			\caption{Gap nonlinearity constitutive relationship}
			\label{fig:18b}
		\end{subfigure}
		\caption{Geometric schematic of the gap folding rudder and its nonlinear constitutive relationship}
	\end{figure}
	
	The four contact nonlinearities are modeled equivalently as four gap springs. After calibration using experimental results, the mathematical model for numerical simulations is illustrated in \cref{fig:18b}. When the local relative displacement is smaller than the gap $\delta$, the equivalent stiffness decreases as the amplitude increases. Once the motion exceeds the gap, the stiffness increases accordingly. A sinusoidal sweep excitation is applied, and different explicit time integration algorithms are employed. For clarity, the displacement at Pin Point 1 in \cref{fig:18a} is chosen for comparison and visualization. It should be noted that this model represents a practical engineering application and, after finite element meshing, possesses a very high number of degrees of freedom; for visualization purposes, only the root‐node degrees of freedom are plotted and compared.
	
		\begin{figure}[htbp]
		\centering
		\includegraphics[width=\linewidth]{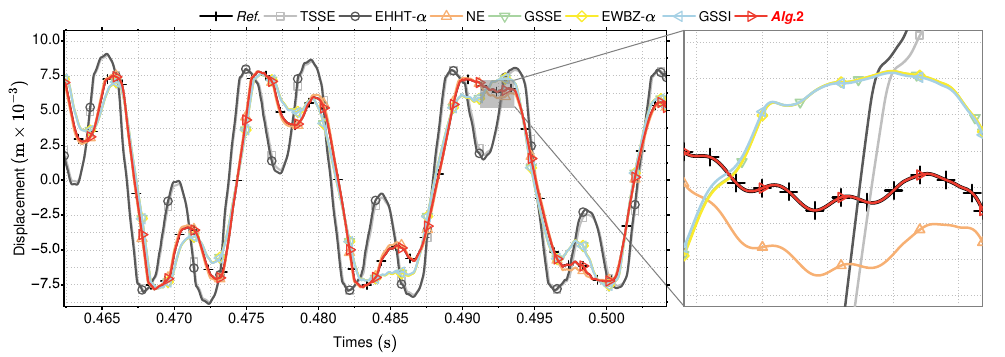}
		\caption{Time-domain response at pin point 1 of the gap folding rudder computed by different self-starting single-solve explicit algorithms}
		\label{fig:19}
	\end{figure}
	
	To begin, various self-starting single-solve explicit algorithms are used to simulate the system, employing the highly dissipative member of each scheme with a time step of $\Delta t = 2 \times 10^{-6}$ seconds. The resulting time-domain responses are shown in \cref{fig:19}, with a reference solution obtained using a fourth-order explicit Runge-Kutta method ($\Delta t = 10^{-9}$ seconds). The results demonstrate that, for such a complex, high-degree-of-freedom nonlinear system, all published single-solve explicit algorithms—whether first- or second-order accurate, fully explicit or velocity-implicit, exhibit significant errors in the time-domain response. In contrast, the proposed algorithm \novelalgref{2} accurately matches the reference solution with high fidelity.
	\begin{figure}[H]
		\centering
		\includegraphics[width=\linewidth]{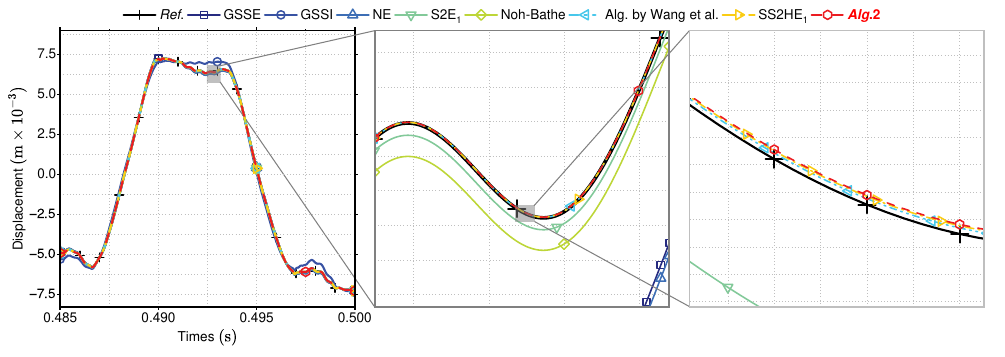}
		\caption{Time-domain response at pin point 1 of the gap folding rudder computed by different self-starting explicit algorithms including two-sub-step schemes}
		\label{fig:22}
	\end{figure}
	
	To highlight the third-order convergence of \novelalgref{2}, a finer time step of $\Delta t = 1 \times 10^{-6}$ is selected for further simulations. In this case, two-sub-step explicit algorithms are also included for comparison. The time-domain displacement at Pin Point 1 is plotted in \cref{fig:22}.
	The results show that the errors of single-solve algorithms (GSSE~\cite{liIdenticalSecondOrder2021}, GSSI~\cite{zhaoSelfstartingDissipativeAlternative2023}, NE~\cite{newmarkMethodComputationStructural1959}) remain substantial. Even the two-sub-step methods (S2E\textsubscript{1}~\cite{liDevelopmentCompositeSubstep2021} and Noh-Bathe~\cite{nohExplicitTimeIntegration2013}) deviate from the reference solution, despite their higher computational cost. In contrast, the third-order algorithms, the proposed algorithm by Wang et al.~\cite{wang_GeneralizedSinglestepMultistage_2025}, SS2HE\textsubscript{1}~\cite{liTwoThirdorderExplicit2022} and \novelalgref{2},  produce solutions that nearly coincide with the reference trajectory.

	In computational mechanics, the trade-off between computational cost and accuracy is a critical consideration. An algorithm that achieves higher accuracy at the same or lower cost is highly advantageous and more likely to be adopted in practical engineering applications. To clearly illustrate this advantage, \cref{fig:20} summarizes the efficiency of different time integrators in simulating this complex nonlinear structure, comparing both error and computational cost. Computational cost is quantified by the CPU runtime under identical hardware and software environments, with repeated measurements to ensure reliability. Errors are evaluated against the reference solution using Eq.~\eqref{eq:51}.
	Given the practical relevance of this structure, two classical implicit schemes (TPO/G-$\alpha$~\cite{shaoThree,chungTimeIntegrationAlgorithm1993} and Newmark-$\beta$ methods~\cite{newmarkMethodComputationStructural1959}) are also included for comparison. The results demonstrate the superiority of explicit algorithms in time-domain simulations of nonlinear dynamics. As shown in \cref{fig:20}, among self-starting single-solve methods, the proposed \novelalgref{2} significantly reduces numerical error while maintaining comparable computational cost.
	In comparison with two-sub-step methods, \novelalgref{2} yields smaller errors than both the Noh-Bathe method~\cite{nohExplicitTimeIntegration2013} and S2E\textsubscript{1}~\cite{liDevelopmentCompositeSubstep2021}. Furthermore, it provides accuracy on par with the algorithms proposed by Wang et al.~\cite{wang_GeneralizedSinglestepMultistage_2025} and SS2HE\textsubscript{1}~\cite{liTwoThirdorderExplicit2022}, but with only half the computational cost.
	
	In summary, \novelalgref{2} demonstrates superior computational accuracy and minimal cost in a real-world engineering application, highlighting its significant practical potential as a high-performance time integration method.
	
	\begin{figure}[ht]
		\centering
		\includegraphics[width=\linewidth]{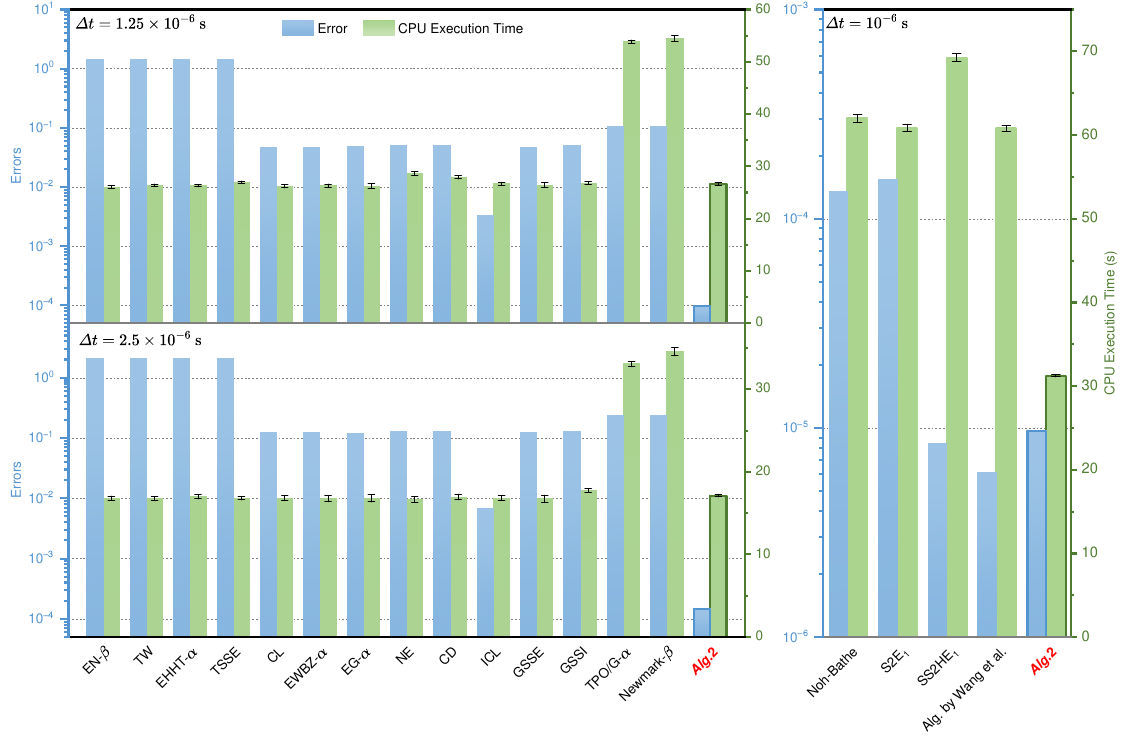}
		\caption{Errors and CPU computation time for the gap-folding rudder computed by different algorithms using various time steps}
		\label{fig:20}
	\end{figure}
	
	\section{Conclusions}
	\label{sec:conclusion}	
    A unified self-starting single-solve integration framework has been formulated and shown to reveal a fundamental limitation: fully explicit single-solve schemes cannot attain third-order accuracy for general dynamics, whereas explicit schemes incorporating the implicit treatment of velocity  can. Two new self-starting explicit algorithms, \novelalgrefs{1}{2}, are proposed and possess the following primary characteristics:
	\begin{enumerate}[(1)]

		\item \textbf{Single-solve per time step}. \novelalgref{1} and \novelalgref{2} require at most one equilibrium solution per time step, in contrast to composite sub-step schemes, thereby ensuring straightforward implementation and minimal overhead.
		\item \textbf{Explicitness}. \novelalgref{1} is fully explicit and never invokes iterative solves, even for nonlinear damping forces.  \novelalgref{2} is an explicit algorithm with implicit treatment of velocity, remaining fully explicit for undamped systems. When nonlinear damping problems are analyzed, \novelalgref{2} necessitates the use of Newton-like iteration schemes.
		\item \textbf{Third-order accuracy}. In the undamped case, \novelalgref{1} achieves third-order accuracy in displacement and velocity and reverts to second order for damping problems.  \novelalgref{2} maintains third-order accuracy in displacement and velocity for general dynamical problems.
		\item \textbf{Dissipative feature}. The built-in numerical dissipation of \novelalgrefs{1}{2} automatically suppresses spurious high-frequency oscillations.
		\item \textbf{No adjustable parameters}. \novelalgrefs{1}{2} have no adjustable parameters, offering the user a straightforward algorithm that is easy to implement in programming.
		
	\end{enumerate}
	
		Extensive numerical validation is carried out across a diverse set of benchmarks to assess the practical performance of the proposed schemes. This included two well‐known nonlinear oscillators—the Van der Pol system and the spring–pendulum model—to probe their behavior under strong nonlinearity; high‐frequency dissipation tests on wave propagation problems, specifically an impact‐loaded one‐dimensional elastic rod and a two‐dimensional square membrane, to evaluate filtering of spurious components; and a large‐scale engineering scenario featuring gap‐induced nonlinearity in a folding rudder assembly. In every case, \novelalgrefs{1}{2} demonstrated clear superiority over existing explicit methods, delivering markedly improved accuracy while requiring minimal computational effort.
	
	\section*{Acknowledgments}
This work is supported by the Postdoctoral Fellowship Program of CPSF (No.~GZC20233464), China Postdoctoral Science Fundation (No.~2024M764165), National Natural Science Foundation of China (No.~12272105), and the Postdoctoral Foundation from Heilongjiang Province (No.~LBH-Z23153).

	\section*{Data Availability Statement}
Data sharing is not applicable to this article as no datasets were generated or analyzed during the current study.

\section*{Conflicts of Interest}
The authors declare no conflicts of interest.
	
	\bibliography{references}
		
	\appendix
	
	\renewcommand{\arraystretch}{1.5}
	
	\section{Butcher tables for some explicit algorithms}
	\label{app1}
	\textbf{EN-$\beta$($0\le {{\rho }_{b}}\le 1$)}\cite{hughesImplicitexplicitFiniteElements1978}
	\begin{equation}
		\begin{array}{*{20}{c}}
			0  & \vline & 0                  & 0     & \vline & \dfrac{3{{\rho }_{b}}-1}{2({{\rho }_{b}}+1)} & 0                                           & \vline & {} & {} \\[6pt]
			\hline \rule{0pt}{20pt}
			{} & \vline & \dfrac{1}{2}-\beta & \beta & \vline & \dfrac{3{{\rho }_{b}}-1}{2({{\rho }_{b}}+1)} & \dfrac{3-{{\rho }_{b}}}{2({{\rho }_{b}}+1)} & \vline & 0  & 1
		\end{array}.
	\end{equation}
	
	\textbf{TW($0\le {{\rho }_{b}}\le 1$)}\cite{maheoNumericalDampingSpurious2013}
	\begin{equation}
		\begin{array}{*{20}{c}}
			1  & \vline & \dfrac{2}{{{\rho }_{b}}+1} & 0 & \vline & 1 & 0 & \vline & {} & {} \\[6pt]
			\hline \rule{0pt}{20pt}
			{} & \vline & \dfrac{2}{{{\rho }_{b}}+1} & 0 & \vline & 1 & 0 & \vline & 0  & 1
		\end{array}.
	\end{equation}
	
	\textbf{EDV1($0\le {{\rho }_{b}}\le 1$)}\cite{zhangTwoNovelExplicit2019}
	\begin{equation}
		\begin{array}{*{20}{c}}
			1  & \vline & \dfrac{3-{{\rho }_{b}}}{2({{\rho }_{b}}+1)} & 0            & \vline & 0 & 0 & \vline & {} & {} \\[6pt]
			\hline \rule{0pt}{20pt}
			{} & \vline & 0                                           & \dfrac{1}{2} & \vline & 1 & 0 & \vline & {} & {}
		\end{array}.
	\end{equation}
	
	\textbf{TSSE($0\le {{\rho }_{b}}\le 1$)}\cite{liDevelopmentCompositeSubstep2021}
	\begin{equation}
		\begin{array}{*{20}{c}}
			1  & \vline & \dfrac{3-{{\rho }_{b}}}{2({{\rho }_{b}}+1)} & 0            & \vline & 0 & 0 & \vline & {} & {} \\[6pt]
			\hline \rule{0pt}{20pt}
			{} & \vline & 0                                           & \dfrac{1}{2} & \vline & 1 & 0 & \vline & 1  & 0
		\end{array}.
	\end{equation}
	
	\textbf{EHHT-$\alpha$($1/2\le {{\rho }_{b}}\le 1$)}\cite{mirandaImprovedImplicitExplicit1989}
	\begin{equation}
		\begin{array}{*{20}{c}}
			\dfrac{2{{\rho }_{b}}}{{{\rho }_{b}}+1} & \vline & \dfrac{(\rho _{b}^{2}+2{{\rho }_{b}}-1){{\rho }_{b}}}{{{({{\rho }_{b}}+1)}^{3}}} & 0                                    & \dfrac{{{\rho }_{b}}(3{{\rho }_{b}}-1)}{{{({{\rho }_{b}}+1)}^{2}}} & 0                                           & \vline & {} & {} \\[6pt]
			\hline \rule{0pt}{20pt}
			{}                                      & \vline & \dfrac{\rho _{b}^{2}+2{{\rho }_{b}}-1}{2{{({{\rho }_{b}}+1)}^{2}}}               & \dfrac{1}{{{({{\rho }_{b}}+1)}^{2}}} & \dfrac{3{{\rho }_{b}}-1}{2({{\rho }_{b}}+1)}                       & \dfrac{3-{{\rho }_{b}}}{2({{\rho }_{b}}+1)} & \vline & 0  & 1
		\end{array}.
	\end{equation}
	
	\textbf{CL($1/2\le {{\rho }_{b}}\le 1$)}\cite{chungNewFamilyExplicit1994}
	\begin{equation}
		\begin{array}{*{20}{c}}
			0  & \vline & 0                                                                                 & 0                                                                  & \vline & 0             & 0            & \vline & 0 & 0 \\
			\hline \rule{0pt}{20pt}
			{} & \vline & \dfrac{\rho _{b}^{3}-5\rho _{b}^{2}-3{{\rho }_{b}}-1}{2{{({{\rho }_{b}}+1)}^{3}}} & \dfrac{4\rho _{b}^{2}+3{{\rho }_{b}}+1}{{{({{\rho }_{b}}+1)}^{3}}} & \vline & -\dfrac{1}{2} & \dfrac{3}{2} & \vline & 0 & 1
		\end{array}.
	\end{equation}
	
	\textbf{ICL($1/2\le {{\rho }_{b}}\le 1$)}\cite{kimSimpleExplicitSingle2019}
	\begin{equation}
		\begin{array}{*{20}{c}}
			1  & \vline & \dfrac{1}{2}                                                            & 0                                                                 & \vline & 1            & 0            & \vline & {} & {} \\[6pt]
			\hline \rule{0pt}{20pt}
			{} & \vline & \dfrac{(3{{\rho }_{b}}+1)(\rho _{b}^{2}+1)}{2{{({{\rho }_{b}}+1)}^{3}}} & \dfrac{\rho _{b}^{2}(1-{{\rho }_{b}})}{{{({{\rho }_{b}}+1)}^{3}}} & \vline & \dfrac{1}{2} & \dfrac{1}{2} & \vline & 0  & 1
		\end{array}.
	\end{equation}
	
	\textbf{EWBZ-$\alpha$($0\le |{{\rho }_{s}}|\le {{\rho }_{b}}\le 1$)}\cite{hulbertExplicitTimeIntegration1996}
	\begin{subequations}
		\begin{equation}
			\begin{array}{*{20}{c}}
				0  & \vline & 0                                                       & 0                                 & \vline & 0                               & 0                                                & \vline & {}                                         & {}                           \\
				\hline \rule{0pt}{20pt}
				{} & \vline & \dfrac{2\beta +{{\alpha }_{m}}-1}{2({{\alpha }_{m}}-1)} & \dfrac{\beta }{1-{{\alpha }_{m}}} & \vline & \dfrac{1}{2({{\alpha }_{m}}-1)} & \dfrac{2{{\alpha }_{m}}-3}{2({{\alpha }_{m}}-1)} & \vline & \dfrac{{{\alpha }_{m}}}{{{\alpha }_{m}}-1} & \dfrac{1}{1-{{\alpha }_{m}}}
			\end{array}
		\end{equation}
		
		\text{where}
		\begin{align}
			{{\alpha }_{m}} &= \dfrac{2{{\rho }_{b}}{{\rho }_{s}}+{{\rho }_{b}}-1}{({{\rho }_{s}}+1)({{\rho }_{b}}+1)}  \\
			\beta &= \dfrac{({{\rho }_{b}}-1)\rho _{s}^{2}+2(\rho _{b}^{2}+{{\rho }_{b}}-2){{\rho }_{s}}-3{{\rho }_{b}}-5}{({{\rho }_{b}}{{\rho }_{s}}-{{\rho }_{s}}-2)({{\rho }_{s}}+1){{({{\rho }_{b}}+1)}^{2}}}.
		\end{align}
	\end{subequations}

	\textbf{EG-$\alpha$($0\le \left| {{\rho }_{s}} \right|\le {{\rho }_{b}}\le 1$)}\cite{liIdenticalSecondOrder2021}
	\begin{subequations}
		\begin{equation}
			\begin{array}{*{20}{c}}
				\dfrac{3}{2}-{{\alpha }_{m}} & \vline & \left( \dfrac{1}{2}-\beta \right)\left( \dfrac{3}{2}-{{\alpha }_{m}} \right) & 0                                 & \vline & \left( \dfrac{3}{2}-{{\alpha }_{m}} \right) & 0 & \vline & {}                                         & {}                           \\[6pt]
				\hline \rule{0pt}{20pt}
				{}                           & \vline & \dfrac{2\beta +{{\alpha }_{m}}-1}{2({{\alpha }_{m}}-1)}                      & \dfrac{\beta }{1-{{\alpha }_{m}}} & \vline & 1                                           & 0 & \vline & \dfrac{{{\alpha }_{m}}}{{{\alpha }_{m}}-1} & \dfrac{1}{1-{{\alpha }_{m}}}
			\end{array}
		\end{equation}
		
		\text{where}	
		\begin{align}
			{{\alpha }_{m}} &= \dfrac{2{{\rho }_{b}}{{\rho }_{s}}+{{\rho }_{b}}-1}{({{\rho }_{s}}+1)({{\rho }_{b}}+1)}  \\
			\beta &= \dfrac{{{({{\rho }_{b}}-1)}^{3}}\rho _{s}^{2}+(\rho _{b}^{2}-\rho _{b}^{3}+13{{\rho }_{b}}-5){{\rho }_{s}}+2\rho _{b}^{2}-10}{2({{\rho }_{b}}+1)(3{{\rho }_{b}}{{\rho }_{s}}+{{\rho }_{b}}-{{\rho }_{s}}-3)({{\rho }_{b}}{{\rho }_{s}}-{{\rho }_{s}}-2)}.
		\end{align}
	\end{subequations}
	
	\textbf{E-GSSSS($0\le \left| {{\rho }_{3b}} \right|\le {{\rho }_{b}}\le 1$)}\cite{shimadaIIntegrationFrameworkAlgorithms2013}
	\begin{subequations}
		\begin{equation}
			\begin{array}{*{20}{c}}
				{{W}_{1}}{{\Lambda }_{1}} & \vline & {{W}_{2}}{{\Lambda }_{2}}-{{W}_{3}}{{\Lambda }_{3}}                                                        & 0                                                                & \vline & {{W}_{1}}{{\Lambda }_{4}}-{{W}_{2}}{{\Lambda }_{5}}                                           & 0                                                   & \vline & {}                                                             & {}                                   \\
				\hline \rule{0pt}{20pt}
				{}                        & \vline & \dfrac{{{W}_{1}}{{\Lambda }_{6}}{{\lambda }_{2}}-{{\lambda }_{3}}{{\eta }_{3}}}{{{W}_{1}}{{\Lambda }_{6}}} & \dfrac{{{\lambda }_{3}}{{\eta }_{3}}}{{{W}_{1}}{{\Lambda }_{6}}} & \vline & \dfrac{{{W}_{1}}{{\Lambda }_{6}}{{\lambda }_{4}}-{{\lambda }_{5}}}{{{W}_{1}}{{\Lambda }_{6}}} & \dfrac{{{\lambda }_{5}}}{{{W}_{1}}{{\Lambda }_{6}}} & \vline & \dfrac{{{W}_{1}}{{\Lambda }_{6}}-1}{{{W}_{1}}{{\Lambda }_{6}}} & \dfrac{1}{{{W}_{1}}{{\Lambda }_{6}}}
			\end{array}
		\end{equation}
		
		\text{where}
		\begin{equation}
			\left\{
			\begin{aligned}
				&{W_1}{\Lambda _6} = \dfrac{{2 + {\rho _{3b}} - {\rho _{3b}}{\rho _b}}}{{(1 + {\rho _b})(1 + {\rho _{3b}})}},\quad {W_1}{\Lambda _1} + {\lambda _5} = \dfrac{{5 + 3{\rho _{3b}} + {\rho _b} - {\rho _{3b}}{\rho _b}}}{{2(1 + {\rho _b})(1 + {\rho _{3b}})}}\\
				&{W_2}{\Lambda _2} - {W_3}{\Lambda _3} + {\lambda _3} = \dfrac{{({\rho _{3b}}{\rho _b} - {\rho _{3b}} - 2)\left[
						2({\rho _{3b}}{\rho_b} - {\rho _{3b}} - 2) 
						- {\Lambda _6}({\rho _{3b}}{\rho _b} - 3{\rho _{3b}} - {\rho _b} - 5)
						\right]}}{{2\Lambda _6^2{{(1 + {\rho _{3b}})}^2}{{(1 + {\rho _b})}^2}}}\\
				&\qquad\qquad - \dfrac{{5 + 3{\rho _b} + {\rho _{3b}}(1 - {\rho _b})(4 + {\rho _{3b}} + 2{\rho _b})}}{{({\rho _{3b}}{\rho _b} - {\rho _{3b}} - 2)(1 + {\rho _{3b}}){{(1 + {\rho _b})}^2}}}\\
				&{\lambda _1} = {\lambda _4} = 1,\quad {\lambda _2} = \dfrac{1}{2},\quad {W_2}{\Lambda _5} = {W_1}({\Lambda _4} - {\Lambda _5}).
			\end{aligned}
			\right.
		\end{equation}
	\end{subequations}
	
	\textbf{NT}\cite{namburuGeneralizedGammaFamily1992}
	\begin{equation}
		\begin{array}{*{20}{c}}
			\dfrac{1}{2} & \vline & 0 & 0            & \vline & 0 & \dfrac{1}{2} & \vline & {} & {}      \\[6pt]
			\hline  \rule{0pt}{20pt}
			{}           & \vline & 0 & \dfrac{1}{2} & \vline & 0 & 1            & \vline & 0  & \quad 1
		\end{array}.
	\end{equation}
	
	\textbf{NE}\cite{newmarkMethodComputationStructural1959}
	\begin{equation}
		\begin{array}{*{20}{c}}
			1  & \vline & \dfrac{1}{2} & \quad 0 & \vline & \dfrac{1}{2} & \quad \dfrac{1}{2} & \vline & {} & {}      \\[6pt]
			\hline \rule{0pt}{20pt}
			{} & \vline & \dfrac{1}{2} & \quad 0 & \vline & \dfrac{1}{2} & \quad \dfrac{1}{2} & \vline & 0  & \quad 1
		\end{array}.
	\end{equation}
	
	\textbf{GSSE(${{\rho }_{s}},{{\rho }_{b}}$)($0\le \left| {{\rho }_{s}} \right|\le {{\rho }_{b}}\le 1$)}\cite{liIdenticalSecondOrder2021}
	\begin{subequations}
		\begin{equation}
			\begin{array}{*{20}{c}}
				p  & \vline & \dfrac{{{p}^{2}}}{2} & 0     & \vline & p               & 0             & \vline & {}             & {}           \\[6pt]
				\hline \rule{0pt}{20pt}
				{} & \vline & \dfrac{1}{2}-\beta   & \beta & \vline & 1-\dfrac{1}{2p} & \dfrac{1}{2p} & \vline & 1-\dfrac{1}{p} & \dfrac{1}{p}
			\end{array}
		\end{equation}
		
		where $p$ and $\beta$ are given as 
		\begin{align}
			p &= -\dfrac{{{\rho }_{b}}{{\rho }_{s}}-{{\rho }_{s}}-2}{({{\rho }_{s}}+1)({{\rho }_{b}}+1)} \label{eq:19} \\
			\beta &= \dfrac{2(1-{{\rho }_{b}}){{\rho }_{b}}\rho _{s}^{3}+(1-{{\rho }_{b}})(\rho _{b}^{2}+6{{\rho }_{b}}-1)\rho _{s}^{2}+2(1-5{{\rho }_{b}}){{\rho }_{s}}+2(1-{{\rho }_{b}})}{2({{\rho }_{s}}+1)({{\rho }_{b}}+1){{({{\rho }_{b}}{{\rho }_{s}}-{{\rho }_{s}}-2)}^{2}}}.\label{eq:20}
		\end{align}
	\end{subequations}
	
	\textbf{GSSI(${{\rho }_{s}},{{\rho }_{b}}$)($0\le \left| {{\rho }_{s}} \right|\le {{\rho }_{b}}\le 1$)}\cite{zhaoSelfstartingDissipativeAlternative2023}
	\begin{equation}
		\begin{array}{*{20}{c}}
			p  & \vline & \dfrac{{{p}^{2}}}{2} & 0     & \vline & \dfrac{p}{2}    & \dfrac{p}{2}  & \vline & {}             & {}           \\[6pt]
			  \hline \rule{0pt}{20pt}
			{} & \vline & \dfrac{1}{2}-\beta   & \beta & \vline & 1-\dfrac{1}{2p} & \dfrac{1}{2p} & \vline & 1-\dfrac{1}{p} & \dfrac{1}{p}
		\end{array}
	\end{equation}
	
	where $p$ and $\beta$ are given in \cref{eq:19,eq:20}.

\end{document}